\documentclass{amsart}%most commn

\usepackage[title]{appendix}

\usepackage[french,english]{babel}
\usepackage[utf8]{inputenc}
\usepackage[T1]{fontenc}
\usepackage{hyperref}

\usepackage{comment}

\usepackage{palatino}

\usepackage[autostyle]{csquotes}

\usepackage{amsmath}
\usepackage{amsfonts,amssymb,amsmath,latexsym,mathrsfs,bbm,stmaryrd}
\usepackage[all]{xy}
\usepackage[usenames]{color}
\usepackage{epsfig}

\usepackage{graphicx}
\usepackage{subcaption}
\usepackage{float,enumerate}
\usepackage{amsthm}
\usepackage{thmtools}
\usepackage{thm-restate}
\declaretheorem[numberwithin=section]{theorem}
\declaretheorem[sibling=theorem]{proposition}
\declaretheorem[sibling=theorem]{definition}
\declaretheorem[sibling=theorem]{corollary}
\declaretheorem[sibling=theorem]{lemma}

\declaretheorem[sibling=theorem,style=remark]{remark}
\declaretheorem[sibling=theorem,style=remark]{example}

\newtheorem{thmy}{Theorem}
\newtheorem*{condition-non}{Condition A}
\newtheorem{condition-test}{Condition A}
\usepackage[utf8]{inputenc} 
\usepackage[T1]{fontenc}
\usepackage{tikz}
\usetikzlibrary{arrows.meta}
\usetikzlibrary{matrix}
\usepackage{caption}
\usepackage{subcaption}
\usepackage{tikz}
\usepackage{xparse}% So that we can have two optional parameters

\NewDocumentCommand\DownArrow{O{2.0ex} O{black}}{%
	\mathrel{\tikz[baseline] \draw [<-, line width=0.5pt, #2] (0,0) -- ++(0,#1);}
}
\numberwithin{equation}{section}

\usepackage[mathcal]{euscript}
\usepackage{times}

\def\C{\mathbb C}

\def\d{\delta}

\def\1{\mathbbm{1}}

\def\adots{
	\mathinner{\mkern1mu\raise1pt\hbox{.}\mkern2mu\raise4pt\hbox{.}
		\mkern2mu\raise7pt\vbox{\kern7pt\hbox{.}}\mkern1mu}}

\long\def\symbolfootnote[#1]#2{\begingroup%
\def\thefootnote{\fnsymbol{footnote}}\footnote[#1]{#2}\endgroup}

\newcommand{\x}{X}

\def\C{{\mathbb C}}

\def\CA{{\mathcal A}}
\def\CB{{\mathcal B}}

\def\unit{{\bf 1}}
\def\i{{\rm i}}

\def\d{{\rm d}}

\def\<{{\langle}}
\def\>{{\rangle}}

\def\supp{{\text{\rm supp}}}

\begin{document}

\title[Brown measure of the sum with a triangular elliptic operator]{The Brown measure of a sum of two free random variables, one of which is triangular elliptic}
\author{
Serban Belinschi, Zhi Yin, and Ping Zhong
}

\address{
\parbox{\linewidth}{Serban Belinschi, 
CNRS, UPS, F-310621 Toulouse, France.
\\
\texttt{serban.belinschi@math.univ-toulouse.fr}}
}

\address{
\parbox{\linewidth}{Zhi Yin,
School of Mathematics and Statistics, 
Central South University,\\
Changsha, Hunan 410083, China.\\
\texttt{hustyinzhi@163.com}}
}

%\author{
\address{
	\parbox{\linewidth}{Ping Zhong,
	Department of Mathematics and Statistics,
	University of Wyoming,\\
	Laramie, WY 82070, USA.\\
	\texttt{pzhong@uwyo.edu}}
	}
%}}

\date{\today}
\maketitle

\begin{abstract}
The triangular elliptic operators are natural extensions of the elliptic deformation of circular operators. We obtain a Brown measure formula for the sum of a triangular elliptic operator $g_{_{\alpha, \beta, \gamma}}$ with a random variable $x_0$, which is $*$-free from $g_{_{\alpha, \beta, \gamma}}$ with amalgamation over certain unital subalgebra. Let $c_t$ be a circular operator. We prove that the Brown measure of $x_0 + g_{_{\alpha, \beta, \gamma}}$ is the push-forward measure of the Brown measure of $x_0 + c_t$ by an explicitly defined map on $\mathbb{C}$ for some suitable $t$.  We show that the Brown measure of $x_0+c_t$ is absolutely continuous with respect to the Lebesgue measure on $\mathbb{C}$ and its density is bounded by $1/(\pi{t})$. This work generalizes earlier results on the addition with a circular operator, semicircular operator, or elliptic operator to a larger class of operators. We extend operator-valued subordination functions, due to Biane and Voiculescu, to certain unbounded operators. This allows us to extend our results to unbounded operators.
\end{abstract}

%\keywords{free probability, Brown measure, random matrices, free Brownian motion, circular law}

\tableofcontents

\section{Introduction}

\subsection{Brown measure and its regularization}
The limit objects of random matrix models can often be identified as operators in free probability theory. In \cite{Brown1986}, Brown introduced the Brown measure of operators in the context of semi-finite von Neumann algebras, which is a natural generalization of the eigenvalue counting measure of square matrices with finite dimensions. In his breakthrough paper \cite{Voiculescu1991}, Voiculescu discovered that independent Gaussian random matrices are asymptotically free as the dimension of the matrices goes to infinity. The Brown measure of a free random variable is regarded as a candidate for the limit empirical spectral distribution of suitable random matrix models. 

Let $\mathcal{A}$ be a von Neumann algebra with a faithful, normal, tracial state $\phi$. Recall that the Fuglede-Kadison determinant of $x\in \mathcal{A}$ is defined by 
\[
  \Delta(x)=\exp \left( \int_0^\infty \log t d\mu_{\vert x\vert}(t) \right),
\]
where $\mu_{\vert x\vert}$ is the spectral measure of $\vert x\vert$ with respect to $\phi$ \cite{FugledeKadison1952}. 
It is known \cite{Brown1986} that the function $\lambda\mapsto\log\Delta(x-\lambda)$ is a subharmonic function  whose Riesz measure is the unique probability measure $\mu_x$ on $\mathbb{C}$ with the property that
\[
  \log\Delta(x-\lambda)=\int_\mathbb{C}\log |z-\lambda| d\mu_x(z), \qquad \lambda \in \mathbb{C}. 
\]
The measure $\mu_x$ is called the \emph{Brown measure} of $x,$ and 
it can be recovered as the distributional Laplacian of the function $\log\Delta(x-\lambda)$ as follows
\begin{equation}
\label{eq:brown} 
\mu_x = \frac{1}{2\pi}\nabla^2 \log\Delta(x-\lambda)
= \frac{2}{\pi} \frac{\partial}{\partial
\lambda}  \frac{\partial}{\partial \bar{\lambda}}  \log
\Delta(x-\lambda).
\end{equation}
When $\mathcal{A}=M_n(\C)$ and $\phi$ is the normalized trace on $M_n(\mathbb{C})$,  for $x\in M_n(\C)$, we have
\[
  \mu_x= \frac{1}{n}\left( \delta_{\lambda_1}+\cdots +\delta_{\lambda_n}\right)
\]
where $\lambda_1,\cdots, \lambda_n$ are the eigenvalues of $x$. 

Let $\widetilde{\mathcal A}$ be the set of operators which are affiliated with $\mathcal{A}$. Denote by $\log^+(\mathcal{A})$ the set of operators $T\in\widetilde{\mathcal A}$ satisfying 
\[
\phi(\log^+\vert T\vert)= \int_0^\infty \log^+(t)\,\d\mu_{\vert T\vert}(t)<\infty.
\]
It is known \cite{HaagerupSchultz2007} that the set $\log^+(\mathcal{A})$ is a subspace of $\widetilde{\mathcal A}$. In particular, for $T\in \log^+(\mathcal{A})$ and $\lambda\in\mathbb{C}$, $T-\lambda  \in \log^+(\mathcal{A})$. Moreover, the function $\lambda\mapsto \log\Delta(T-\lambda)$
is again subharmonic in $\mathbb{C}$. 
In the noncommutative probability space $(\mathcal{A},\phi)$, one can extend the definition of the Brown measure to $\log^+(\mathcal{A})$ \cite{Brown1986, HaagerupSchultz2007} using the same formula \eqref{eq:brown}. 

It is useful to consider the regularized function
\begin{equation}
\label{eqn:defn-S}
    S(x,\lambda,\varepsilon)=\phi\bigg(\log \big( (x-\lambda)^*(x-\lambda)+\varepsilon^2\big) \bigg).
\end{equation}
Then, $S(x,\lambda,0)=\lim_{\varepsilon\rightarrow 0} S(x,\lambda,\varepsilon)$. The regularized  Brown measure is defined as
\[
   \mu_{x}^{(\varepsilon)}=\frac{1}{4\pi}\nabla^2 S(x,\lambda,\varepsilon).
\]
It is known that $\mu_{x}^{(\varepsilon)}\rightarrow\mu_x$ weakly as $\varepsilon\rightarrow 0$. Moreover, it turns out that the regularized Brown measure approximates the Brown measure in a much stronger sense for the operators we are interested in. 
\subsection{Main results}
Given two $*$-free random variables $x, y$ in a $W^*$-probability space $(\mathcal{A},\phi)$, it is desirable to find a general method to calculate the Brown measure of $x+y$. In the special case when $y$ is an elliptic operator, the third author obtained Brown measure formulas for an arbitrary $x$ in \cite{Zhong2021Brown}. This generalized previous results of Hall--Ho \cite{HoHall2020Brown} and Ho--Zhong \cite{HoZhong2020Brown}, and introduced a new approach based on the Hermitian reduction method and subordination functions. 
The purposes of this paper are two-folds. First, we adapt techniques developed in \cite{BSS2018,Zhong2021Brown} to study the Brown measure of the sum of a larger family of operators $g_{_{\alpha, \beta,\gamma}}$, called triangular elliptic operators, and another random variable $x_0$ that is $*$-free from $g_{_{\alpha, \beta,\gamma}}$ in the sense of operator-valued free probability. Second, we extend our main results to unbounded operator $x_0$. Our results for the case $\alpha=\beta$ extend main results obtained by the third author \cite{Zhong2021Brown}, and generalize main results in \cite{Ho2021unbounded} obtained by Ho for the special case $x_0$ being self-adjoint, and another result by Bordenave--Caputo--Chafa\"{\i} \cite{Bordenave-Caputo-Chafai-2014cpam} for unbounded normal operators. 

 It turns out the family of Brown measures $x_0+g_{_{\alpha, \beta,\gamma}}$ can be retrieved from the Brown measure $x_0+c_t$ by a family of explicitly constructed self-maps of $\mathbb{C}$, where $c_t$ is a circular operator with variance $t$. Hence, it is crucial to understand the Brown measure $\mu_{x_0+c_t}$. 
 We introduce a new approach to study regularity properties for the Brown measure of $x_0+c_t$ via the regularized Brown measure, and obtain a complete description of the Brown measure $\mu_{x_0+c_t}$. We prove that the density of the regularized Brown measure $\mu_{x_0+c_t}^{(\varepsilon)}$ is bounded by $1/(\pi{t})$. Consequently, we show that $\mu_{x_0+c_t}$ has no singular part, and this settles an open question in \cite{Zhong2021Brown}. 
 
 Our method is completely different from the PDE method used by Ho \cite{Ho2021unbounded}, where he studied the sum of an unbounded self-adjoint operator and an elliptic operator (corresponding to $\alpha=\beta$). Although our approach is an extension of the method used in \cite{Zhong2021Brown}, the details are substantially more technical, and some new ideas are developed to study the distinguished measure $\mu_{x_0+c_t}$. 
 In addition, we extend the subordination functions in operator-valued free probability, due to Biane and Voiculescu, to unbounded operators under some assumptions. The subordination result is of interest by its own, and it allows us to extend our results to unbounded operators. 

Consider a random matrix $A_N=(a_{ij})_{1\leq i,j\leq N}$ such that the joint distribution of random variables $(\Re a_{ij}, \Im a_{ij})_{1\leq i,j\leq N}$ is centered Gaussian with covariance given by 
\[
   \mathbb{E} a_{ij}\overline{a_{kl}}=
     \begin{cases}
       \frac{\alpha}{N}\delta_{ik}\delta_{jl} \quad & \text{if} \,i<j,\\
    \frac{\alpha+\beta}{2N}\delta_{ik}\delta_{jl} \quad & \text{if} \, i=j,\\
       \frac{\beta}{N}\delta_{ik}\delta_{jl} \quad & \text{if} \,i>j,
     \end{cases} 
     \quad \text{and} \quad
     \mathbb{E}a_{ij} a_{kl}=\frac{\gamma}{N}\delta_{il}\delta_{jk},
\]  
where $\alpha,\beta>0$ and $\gamma\in\mathbb{C}$ such that $\vert \gamma\vert\leq \sqrt{\alpha\beta}$. One can show \cite{DykemaHaagerup2001JFA, BSS2018} that $A_N$ converges in $*$-moments to some operator $g_{\alpha, \beta,\gamma}$, called a triangular elliptic operator, in $\mathcal{A}$.

When $\alpha=\beta=t$ and $\gamma=0$, then $A_N$ converges in $*$-moments to the circular operator $c_t$ with variance $t$, which has the same distribution as $\sqrt{t/2}(s_1+is_2)$,
where $\{s_1, s_2\}$ is a semicircular system. Let $\{g_{t_1}, g_{t_2}\}$ be freely independent semicircular operators with mean zero and variances $t_1, t_2$ respectively. Choose $t_1=(t+\vert \gamma\vert)/2, t_2=(t-\vert \gamma\vert)/2$ and $\theta\in [0,2\pi]$ such that $e^{i2\theta}=\gamma/\vert \gamma\vert$. When $\alpha=\beta=t$ and $\vert \gamma\vert\leq t$, then $A_N$ converges in $*$-moments to twisted elliptic operator $g_{t,\gamma}$, which has the same distribution as $e^{i\theta}(g_{t_1}+ig_{t_2})$.

Consider the open set
\begin{equation}
\label{defn:Xi-t}
 \Xi_t=\left\{ \lambda\in\mathbb{C} : \phi \left[\big( (x_0-\lambda)^*(x_0-\lambda)\big)^{-1}\right]>\frac{1}{t}  \right\}.
\end{equation}
Fix $\lambda\in \Xi_t$ and let $w=w(0;\lambda,t)$ be a function of  $\lambda$ taking positive values such that
\[
  \phi \left[ ((x_0-\lambda)^*(x_0-\lambda)+w^2)^{-1} \right]=\frac{1}{t},
\]
and let $w(0;\lambda,t)=0$ for $\lambda\in\mathbb{C}\backslash\Xi_t$. 
We denote
\[
  \Phi_{t,\gamma} (\lambda)= \lambda+\gamma\cdot p_\lambda^{(0)}( w  ), 
\]
 where 
  \[
    p_\lambda^{(0)}( w ) =
     -\phi\bigg[ (x_0-\lambda)^*\big( (x_0-\lambda)(x_0-\lambda)^*+w(0;\lambda,t)^2 \big)^{-1}\bigg].
 \]
\begin{theorem}
	\label{thm:1.1-introduction}
\cite[Theorem 4.2 and 5.4]{Zhong2021Brown}
The Brown measure of $x_0+c_t$ has no atom and it is supported in the closure of $\Xi_t$. It is absolutely continuous with respect to the Lebesgue measure in $\Xi_t$. Moreover, the density in $\Xi_t$ is strictly positive and can be expressed explicitly in terms of $w(0;\lambda,t)$. 

For any $\gamma\in\mathbb{C}$ such that $\vert \gamma\vert\leq t$, the Brown measure of $x_0+g_{t,\gamma}$ is the push-forward measure of the Brown measure of $x_0+c_t$ under the map $\Phi_{t,\gamma}$. 
\end{theorem}

The push-forward property of a similar type was first observed for some special cases when $g_{t,\gamma}$ is a semicircular element, and the self-adjoint operator $x_0$ in \cite{HoZhong2020Brown} by PDE methods. In a follow-up work \cite{HoHall2020Brown}, it was proved for the sum of an imaginary multiple of semicircular elements and a self-adjoint operator $x_0$. 

When $\alpha\neq \beta$, one can show \cite{DykemaHaagerup2004-DT, BSS2018} that the limit operator $y=g_{_{\alpha, \beta,\gamma}}$ of $A_N$ lives in some operator-valued $W^*$-probability space $(\mathcal{A},\mathbb{E},\mathcal{B})$ where the unital subalgebra $\mathcal{B}$ can be identified as $\mathcal{B}=L^\infty[0,1]$ and $\mathbb{E}:\mathcal{A}\rightarrow\mathcal{B}$ is the conditional expectation (see Section \ref{section:tri-elliptic-review}). Given $x\in\mathcal{A}$, by identifying $\mathbb{E}(x)=f$ for some $f\in L^\infty[0,1]$, 
the tracial state $\phi:\mathcal{A}\rightarrow\mathbb{C}$ is determined by 
\[
  \phi(x)=\phi(\mathbb{E}(x))=\int_0^1f(s)ds.
\]
Let $x_0\in\log^+(\mathcal{A})$ be an operator that is $*$-free from the family of operators $\{g_{_{\alpha, \beta,\gamma}}\}_{\vert \gamma\vert\leq \sqrt{\alpha\beta}}$ with amalgamation over $\mathcal{B}$ such that $x_0$ is $*$-free from $\mathcal{B}$ with respect to the tracial state $\phi$. Our main interest is the Brown measure of $x_0+g_{\alpha, \beta,\gamma}$. 

We now describe our main results. Given $\alpha,\beta>0$, we denote
\[
 	{\Xi_{\alpha,\beta}}=\left\{ \lambda: \phi (\vert x_0-\lambda\vert^{-2}) > \frac{\log\alpha-\log\beta}{\alpha-\beta}  \right\}.
\] 
For $\lambda\in{\Xi_{\alpha,\beta}}$, let $s(\lambda)$ be the positive number $s$ determined by
\begin{equation*}
  \frac{\log\alpha-\log\beta}{\alpha-\beta} =
 \phi\left\{ \left[ 
 (\lambda-x_0)^*(\lambda-x_0) + s^2
 \right] ^{-1} \right\},
\end{equation*}
and for $\lambda\in \mathbb{C}\backslash {\Xi_{\alpha,\beta}}$, we put $s(\lambda)=0$. 
For $\gamma\in\mathbb{C}$ such that $\vert \gamma\vert\leq \sqrt{\alpha\beta}$, we define the map $\Phi_{\alpha, \beta,\gamma}:\mathbb{C}\rightarrow\mathbb{C}$ by
\[
  \Phi_{\alpha, \beta,\gamma}(\lambda)=\lambda+\gamma \phi \left\{ (\lambda-x_0)^* \left[ (\lambda-x_0)(\lambda-x_0)^* + s(\lambda)^2\right]^{-1}\right\}.
\]

\begin{theorem}[See Theorem \ref{thm:Brown-formula-gamma-0} and Theorem \ref{thm:Brown-push-forward-property}]
	\label{thm:1.2-Brown-push-forward-property}
We assume that $\mathcal{M}, \mathcal{N}$ are unital von Neumann subalgebras of $\mathcal{A}$ such that: 
\begin{itemize}
\item $\mathcal{M}, \mathcal{N}$ are free with amalgamation over $\mathcal{B}$ with respect to the conditional expectation $\mathbb{E}$; 
\item the tracial state $\phi$ on $\mathcal{A}$ satisfies $\phi(x)=\phi(\mathbb{E}(x))$ for any $x\in \mathcal{A}$;
\item $\mathbb{E}(x)=\phi(x)$ for any $x\in \mathcal{N}$.
\end{itemize}
Let $x_0\in\log^+(\mathcal{N})$ and $\{g_{_{\alpha, \beta,\gamma}}\}_{\vert \gamma\vert\leq \sqrt{\alpha\beta}}$ be a family of operators in $\mathcal{M}.$ Given $\alpha,\beta>0$ such that $\alpha\neq\beta$, 
the support of the Brown measure $\mu_{x_0+g_{_{\alpha, \beta,0}}}$ of $x_0+g_{_{\alpha, \beta,0}}$ is the closure of the set ${\Xi_{\alpha,\beta}}$. The Brown measure $\mu_{x_0+g_{_{\alpha, \beta,0}}}$ is absolutely continuous and the density is strictly positive in ${\Xi_{\alpha,\beta}}$. Moreover, the density formula can be expressed in terms of $s(\lambda)$. 

Given $\gamma\in\mathbb{C}$ such that $\vert \gamma\vert\leq \sqrt{\alpha\beta}$,  the Brown measure $\mu_{x_0+g_{_{\alpha, \beta,\gamma}}}$ is the push-forward measure of $\mu_{x_0+g_{_{\alpha, \beta,0}}}$ under the map $\Phi_{\alpha, \beta,\gamma}$. If, in addition, the map  $\Phi_{\alpha, \beta,\gamma}$ is non-singular at any $\lambda\in{\Xi_{\alpha,\beta}}$, then it is also one-to-one in ${\Xi_{\alpha,\beta}}$.
\end{theorem}

The above Theorem \ref{thm:1.2-Brown-push-forward-property} recovers the Brown measure formula obtained in \cite[Section 4]{HaagerupAagaard2004} and \cite[Section 6]{BSS2018} as a special case for $x_0=0$; see Example \ref{example:tri-elliptic-0}. 
We note that the map $\Phi_{_{\alpha, \beta,\gamma}}$ could be singular in general. This restriction does not allow us to calculate the Brown measure $\mu_{x_0+g_{_{\alpha, \beta,\gamma}}}$ directly. We adapt an approach introduced in \cite{Zhong2021Brown} as follows. We show that the regularized Brown measure $\mu_{x_0+g_{_{\alpha, \beta,\gamma}}}^{(\varepsilon)}$ is the pushforward measure of 
the Brown measure $\mu_{x_0+g_{_{\alpha, \beta,0}}}^{(\varepsilon)}$ by some regularized map $\Phi_{_{\alpha, \beta,\gamma}}^{(\varepsilon)}$. We then show that the pushforward connection is preserved by passing to the limit. Hence, the following diagram commutes. 

	\begin{center}
	\begin{tikzpicture}
	\matrix (m) [matrix of math nodes,row sep=3em,column sep=4em,minimum width=2em]
	{
		\mu_{x_0+g_{_{\alpha, \beta,0}}}^{(\varepsilon)} & \mu_{x_0+g_{_{\alpha, \beta,\gamma}}}^{(\varepsilon)} \\
		\mu_{x_0+g_{_{\alpha, \beta,0}}} & \mu_{x_0+g_{_{\alpha, \beta,\gamma}}}\\};
	\path[-stealth]
	(m-1-1) edge node [left] {$\varepsilon\rightarrow 0$} (m-2-1)
	edge [double] node [below] {$\Phi_{\alpha, \beta,\gamma}^{(\varepsilon)}$} (m-1-2)
	(m-2-1.east|-m-2-2) edge [double] node [below] {{$\Phi_{\alpha, \beta,\gamma}$}}
	(m-2-2)
	(m-1-2) edge node [right] {$\varepsilon\rightarrow 0$} (m-2-2)
	(m-2-1);
	\end{tikzpicture}
\end{center}

In \cite{Zhong2021Brown}, it was left open if the Brown measure of $\mu_{x_0+c_t}$ has any singular continuous part. We show that it is absolutely continuous with respect to the Lebesgue measure on $\mathbb{C},$ and the statement of Theorem \ref{thm:1.1-introduction} still holds for unbounded operators. 
\begin{theorem}[See Theorem \ref{thm:BrownFormula-x0-ct-general} and Theorem \ref{thm:Brown-pushforwd-elliptic-case}]
	\label{thm:1.3-introduction}
	The Brown measure $\mu_{x_0+c_t}$ is absolutely continuous with respect to the Lebesgue measure on $\mathbb{C}$ and 
	the result in Theorem \ref{thm:1.1-introduction} holds for any $x_0\in\log^+(\mathcal{A})$ that is $*$-free from $\{c_t, g_{t,\gamma}\}$. Moreover, the density function $\rho_{t,x_0}$ of the Brown measure $\mu_{x_0 + c_t}$ has an upper bound
	\[
	   \rho_{t,x_0}(\lambda)\leq \frac{1}{\pi t}. 
	\] 
\end{theorem}

In \cite{Bordenave-Caputo-Chafai-2014cpam}, Bordenave, Caputo and Chafa\"{\i} investigated the spectrum of the Markov generator of random walk on a randomly weighted oriented graph. They identified the limit distribution as the Brown measure of $a+c$, where $a$ is a Gaussian distributed normal operator and $c$ is the standard circular operator, freely independent from $a$. Hence, 
Theorem \ref{thm:1.3-introduction} implies that $\mu_{a+c}$ is absolutely continuous and its density is bounded by $1/\pi$. 
Our main results also provide potential applications to unify the deformed i.i.d. random matrix model, the deformed Wigner random matrix model, and the deformed elliptic random matrix model. We hope to study this question somewhere else.

By comparing the precise formulas in Theorem \ref{thm:1.1-introduction} with Theorem \ref{thm:1.2-Brown-push-forward-property}, we can view the main results in \cite{Zhong2021Brown} as the special case of our results when $\alpha=\beta=t$, where $\frac{\alpha-\beta}{\log\alpha-\log\beta}$ is regarded as $t$.
\begin{corollary}
Under the same assumption as Theorem \ref{thm:1.2-Brown-push-forward-property}, 
given $\alpha,\beta>0$, set $t=\frac{\alpha-\beta}{\log\alpha-\log\beta}$. Suppose $x_0$ is $*$-free from $\{c_t, g_{t,\gamma}\}_{\vert \gamma\vert\leq t}$ in $(\mathcal{A},\phi)$. Then, 
\begin{enumerate}[{\rm (1)}]
\item the Brown measure of $x_0+g_{_{\alpha, \beta,0}}$ is the same as the Brown measure of $x_0+c_t$. In particular, it is absolutely continuous in $\mathbb{C}$ and its density is bounded by $1/(\pi {t})$;
\item the push-forward map $\Phi_{\alpha, \beta,\gamma}$ between Brown measures $\mu_{x_0+g_{_{\alpha, \beta,0}}}$ and 
		$\mu_{x_0+g_{_{\alpha, \beta,\gamma}}}$ is the same as the push-forward map $\Phi_{t,\gamma}$ between Brown measures $\mu_{x_0+c_t}$ and $\mu_{x_0+g_{t,\gamma}}$. 
\end{enumerate}
\end{corollary}

The paper has six more sections. We review some results from free probability theory in Section 2. We study the subordination functions for unbounded operators in Section 3. We calculate the Brown measure formulas and prove the push-forward property for regularized Brown measures in Section 4 and Section 5. In Section 6, we study some further properties of the push-forward map. In Section 7, we study the Brown measure of an addition with an elliptic operator.  

%%%%%%%%%%%%%%%%%%%%%%%%%%%%%%%

\section{Preliminaries}
\label{Prelim}
\subsection{Free probability and subordination functions}
\label{FreeProb}
An \emph{operator-valued $W^*$-probability space} $(\CA, \mathbb{E}, \CB)$ consists of a von Neumann algebra $\CA$, a unital $*$-subalgebra $\CB\subset \CA$, and  a \emph{conditional expectation}
$\mathbb{E}:\CA\rightarrow \mathcal{B}$. Thus, $\mathbb{E}$ is a unital linear positive map satisfying:
(1) $\mathbb{E}(b)=b$ for all $b\in\CB$, and
(2) $\mathbb{E}(b_1xb_2)=b_1\mathbb{E}(x)b_2$ for all $x\in\CA$, $b_1, b_2\in\CB$. 
Let $(\CA_i)_{i\in I}$ be a family of subalgebras such that $\CB\subset \CA_i\subset \CA$. We say that $(\CA_i)_{i\in I}$ are \emph{free with amalgamation} over $\CB$ with respect to the conditional expectation $\mathbb{E}$ \emph{(}or free with amalgamation in $(\CA, \mathbb{E}, \CB)) $ if for every $n \in \mathbb{N},$
\[
\mathbb{E}(x_1x_2\cdots x_n)=0
\]
whenever $x_j\in \CA_{i_j} (j=1, 2, \cdots, n)$ for some indices $i_1, i_2, \cdots, i_n \in I$ such that 
$i_1\neq i_2, i_2\neq i_3, \cdots, i_{n-1}\neq i_n$, and $\mathbb{E}(x_j)=0$ for all $j=1, 2, \cdots, n$. 

Let $(\CA, \mathbb{E}, \CB)$ be an operator-valued $W^*$-probability space. The elements in $\CA$ are called (noncommutative) random variables.  We call 
\[
\mathbb{H}^+(\CB)=\{ b\in\CB: \exists \varepsilon>0, \Im (b)\geq \varepsilon \unit \}
\]
the Siegel upper half-plane of $\CB$, where we use the notation $\Im (b)=\frac{1}{2i}(b-b^*)$.
We set $\mathbb{H}^-(\CB)=\{-b: b\in  \mathbb{H}^+(\CB)\}$. The $\mathcal{B}$-valued Cauchy transform $G_X$ of any self-adjoint operator $X\in \CA$ is defined by
\[
G_X(b)=\mathbb{E}[ (b-X)^{-1} ],  \quad  b\in  \mathbb{H}^+(\CB).
\] 
Note that the $\mathcal{B}$-valued Cauchy transform $G_X$ is a map from $\mathbb{H}^+(\CB)$ to $\mathbb{H}^-(\CB)$, and it is one-to-one in $\{b\in \mathbb{H}^+(\CB): ||b^{-1}||<\varepsilon \}$ for sufficiently small $\varepsilon$. Voiculescu's amalgamated $R$-transform is now defined for $X\in\CA$ by
\[
R_X(b)=G_X^{\langle -1 \rangle}(b)-b^{-1}
\]
for $b$ being invertible element of $\CB$ suitably close to zero. 

Let $X, Y$ be two self-adjoint bounded operators that are free with amalgamation in $(\mathcal{A}, \mathbb{E},\mathcal{B})$.  The $R$-transform linearizes the free convolution in the sense that if $X, Y$ are free with amalgamation in $(\CA, \mathbb{E}, \CB)$, then 
\[
R_{X+Y}(b)=R_X(b)+R_Y(b)
\]
for $b$ in some suitable domain. There exist two analytic self-maps $\Omega_1,\Omega_2$
of the upper half-plane $\mathbb{H}^+(\mathcal{B})$
so that
\begin{equation}\label{eqn:subord-operator}
(\Omega_1(b)+\Omega_2(b)-b)^{-1}=G_X(\Omega_1(b))=G_Y(\Omega_2(b))=G_{X+Y}(b),
\end{equation}
for all $b\in \mathbb{H}^+(\mathcal {B})$. We refer the reader to \cite{SpeicherAMS1998, DVV-operator-valued-1992} for basic operator-valued free probability theory and \cite{BelinschiTR2018-sub-operator-valued,  Biane1998, DVV-general} for operator-valued subordination functions.

\subsection{Hermitian reduction method and subordination functions}
Our approach is based on a Hermitian reduction method and (operator-valued) subordination functions. 
The Hermitian reduction method was used for the calculation of the Brown measure of quasi-nilpotent DT operators in Aagaard--Haagerup's work \cite{HaagerupAagaard2004}. The idea was refined in a recent work \cite{BSS2018}. The usage of such an idea in Brown measure calculation also appears earlier in some physical literatures (see \cite{FeinbergZee1997-a, Nowak1997-non-Hermitian} for example). The idea of the Hermitian reduction method has a connection to the work of Girko \cite{Girko1984} on circular law, though it is written in a different form. 

Let $(\CA, \mathbb{E}, \CB)$ be an operator-valued $W^*$-probability space, and let $\phi:\mathcal{A}\rightarrow\mathbb{C}$ be a tracial state on $\CA$ such that
$\phi(x)=\phi(\mathbb{E}(x))$ for any $x\in\CA$. 
We equip
the algebra $M_2(\CA)$, the $2\times 2$ matrices with entries from $\CA$, with the conditional expectation
$M_2(\mathbb{E}):M_2(\CA)\rightarrow M_2(\CB)$ given by \begin{equation}
\label{eq:conditional} M_2(\mathbb{E})
\begin{bmatrix} a_{11} &
a_{12} \\ a_{21} & a_{22} \end{bmatrix} = \begin{bmatrix} \mathbb{E}(a_{11}) & \mathbb{E}(a_{12}) \\
\mathbb{E}(a_{21}) & \mathbb{E}(a_{22}) \end{bmatrix},
\end{equation}
and $M_2(\phi):M_2(\CA)\rightarrow M_2(\mathbb{C})$ by
\begin{equation}
\label{eq:conditional-2} M_2(\phi)
\begin{bmatrix} a_{11} &
a_{12} \\ a_{21} & a_{22} \end{bmatrix} = \begin{bmatrix} \phi(a_{11}) & \phi(a_{12}) \\
\phi(a_{21}) & \phi(a_{22}) \end{bmatrix}.
\end{equation}
Then the triple  $(M_2(\mathcal{A}), M_2(\mathbb{E}), M_2(\mathcal{B}))$ is a $W^*$-probability space such that $M_2(\phi)\circ M_2(\mathbb{E})=M_2(\phi)$. 
Given $x\in\CA$,
let
\begin{equation}
\label{eq:X} \x=
\begin{bmatrix} 0 & x \\ x^\ast & 0 \end{bmatrix} \in M_2(\CA),
\end{equation}
which is a self-adjoint element in $M_2(\CA).$ 
The $M_2(\mathcal{B})$-valued Cauchy transform is defined as
\begin{equation}\label{def-matr-valued-Cauchy}
G_X(b)=M_2(\mathbb E)\left[(b-X)^{-1}\right], \qquad b\in \mathbb{H}^+( M_2(\mathcal{B})),
\end{equation}
which is an analytic function on $\mathbb{H}^+( M_2(\mathcal{B}))$.
In particular, for $\varepsilon>0$ the element
\begin{equation}
\label{defn:Theta}
\Theta
(\lambda, \varepsilon) =  \begin{bmatrix} i\varepsilon & \lambda \\
\bar{\lambda}  & i\varepsilon \end{bmatrix} \in  M_2(\C) \subset M_2(\mathcal{B})
\end{equation}
belongs to the domain of $G_X$, and
\begin{align*}
&(\Theta
(\lambda,\varepsilon) -\x)^{-1}\\
=  &\begin{bmatrix}
-i \varepsilon \big((\lambda-x)(\lambda-x)^\ast+\varepsilon^2\big)^{-1}
& (\lambda-x)\big((\lambda-x)^\ast (\lambda-x)+\varepsilon^2\big)^{-1} \\
(\lambda-x)^\ast \big(
(\lambda-x)(\lambda-x)^\ast+\varepsilon^2\big)^{-1}  & -i \varepsilon
\big( (\lambda-x)^\ast(\lambda-x)+\varepsilon^2 \big)^{-1}
\end{bmatrix}.
\end{align*}
Hence, by taking the entry-wise trace on the $M_2(\mathcal{B})$-valued Cauchy transform, we have
\begin{equation} \label{eq:ciemnosc} 
 \begin{aligned}
&M_2(\phi)\big(G_X
(\Theta
(\lambda, \varepsilon) ) \big)
=M_2(\phi) \big( (\Theta
(\lambda, \varepsilon) - \x)^{-1}\big)\\
&=\begin{bmatrix}
g_{X, 11}(\lambda,\varepsilon) & g_{X, 12}(\lambda,\varepsilon)\\
g_{X, 21}(\lambda,\varepsilon) & g_{X, 22}(\lambda,\varepsilon)
\end{bmatrix},
 \end{aligned}
\end{equation} 
where 
\begin{equation}
 \label{eqn:Cauchy-transform-entries-scalar}
 \begin{aligned}
  	g_{X,11}(\lambda, \varepsilon) & =  -i\varepsilon\phi\left(\big((\lambda-x)(\lambda-x)^*+\varepsilon^2
	\big)^{-1}\right),\\
	g_{X,12}(\lambda, \varepsilon) &=\phi\left((\lambda-x)\big((\lambda-x)^* (\lambda-x)+\varepsilon^2
	\big)^{-1}\right), \\
	g_{X,21}(\lambda, \varepsilon)&=\phi\left((\lambda-x)^*\big((\lambda-x)(\lambda-x)^*+
	\varepsilon^2\big)^{-1}\right),\\
	g_{X,22}(\lambda, \varepsilon) & =  -i\varepsilon\phi\left(\big( (\lambda-x)^*(\lambda-x)+\varepsilon^2 
	\big)^{-1}\right).
 \end{aligned}
\end{equation}
Recall that $S(x,\lambda,\varepsilon)$ is defined in \eqref{eqn:defn-S}. We see immediately that 
\[
ig_{X,11}(\lambda,\varepsilon)=\frac{1}{2}\frac{\partial S(x,\lambda,\varepsilon)}{\partial \varepsilon},
\qquad g_{X,21}(\lambda,\varepsilon)=\frac{\partial S(x,\lambda,\varepsilon)}{\partial \lambda}.
\]
Hence, the entries of operator-valued Cauchy transform carry important information about the Brown measure. 

Consider now two operators $x,y\in\mathcal{A}$ that are $*$-free with amalgamation over $\mathcal{B}$.  We have
to understand the $ M_2( \mathcal{B})$-valued distribution of
$$\begin{bmatrix} 0 & x+y \\ (x+y)^* & 0 \end{bmatrix}=X+Y=
\begin{bmatrix} 0 & x \\ x^* & 0 \end{bmatrix}+
\begin{bmatrix} 0 & y \\ y^* & 0 \end{bmatrix}$$
in terms of the $M_2(\mathcal{B})$-valued distributions of $X$ and of $Y$. Note that $X$ and $Y$ are free in $(M_2(\mathcal{A}), M_2(\mathbb{E}),M_2(\mathcal{B}))$ with amalgamation over $M_2(\mathcal{B})$.
The subordination functions in this context are two analytic self-maps  $\Omega_1,\Omega_2$
of the upper half-plane $\mathbb{H}^+(M_2(\mathcal{B}))$
so that
\begin{equation}\label{eqn:subord}
(\Omega_1(b)+\Omega_2(b)-b)^{-1}=G_X(\Omega_1(b))=G_Y(\Omega_2(b))=G_{X+Y}(b),
\end{equation}
for every $b\in \mathbb{H}^+(M_2(\mathcal{B}))$. We shall be concerned with a special
type of $b$, namely $b=\Theta(\lambda, \varepsilon)$.
More precisely, we want to understand the entries of the $M_2(\mathbb{C})$-valued Cauchy-transform
\[
  M_2(\phi)(G_{X+Y}(\Theta(\lambda,\varepsilon)))=M_2(\phi) \left(
    \begin{bmatrix}
      i\varepsilon & \lambda-x-y\\
      (\lambda-x-y)^* & i\varepsilon
    \end{bmatrix}^{-1}
  \right).
\]
 The idea of calculating the Brown measure of $x+y$ is to separate the information of $X$ and $Y$ in some tractable way using subordination functions \eqref{eqn:subord}.
 
 %%%%%%%%%%%%%%%%%%%%%%%%%%%%%%%%%%
 
 \section{Extension of subordination functions to unbounded operators}
 Let $(\mathcal{A}, \mathbb{E}, \mathcal{B})$ be an operator-valued $W^*$-probability space. Consider arbitrary $x,y\in\tilde{\mathcal A}$ which are $*$-free over $\mathcal B$ with respect to $\mathbb{E}$, meaning that $W^*(\mathcal B,x)$ and $W^*(\mathcal B,y)$ are free with respect to $\mathbb{E}$ (when $x$ is unbounded, $W^*(\mathcal B,x)$ 
denotes the von Neumann algebra generated by $\mathcal B$, $u$, and the spectral projections of $\sqrt{x^*x}$ from the polar decomposition $x=u\sqrt{x^*x}$). 

As in the previous section, we denote $X=\begin{bmatrix} 0 & x \\ x^* & 0 \end{bmatrix}$, $Y=\begin{bmatrix} 0 & y \\ y^* & 0 \end{bmatrix}$, and 
$G_X(b)=M_2(\mathbb E)\left[(b-X)^{-1}\right]$, $b\in\mathbb H^+(M_2(\mathcal B))$. In addition, for use in this section, we also introduce
\[
F_X(b)=G_X(b)^{-1},\quad h_X(b)=F_X(b)-b,\qquad b\in\mathbb H^+(M_2(\mathcal B)).
\]
It is known that $\Im F_X(b)\ge\Im b$, so that $\Im h_X(b)\ge0,b\in\mathbb H^+(M_2(\mathcal B))$ (see \cite[Remark 2.5]{BelinschPopaVinnikov2012}). 
By the similar result for bounded operators, then, $X, Y$ are free with amalgamation over $M_2(\mathcal B)$ with respect to 
$M_2(\mathbb E)\begin{bmatrix} a_{11} & a_{12} \\ a_{21} & a_{22} \end{bmatrix}
=\begin{bmatrix} \mathbb{E}(a_{11}) & \mathbb{E}(a_{12}) \\ \mathbb{E}(a_{21}) & \mathbb{E}(a_{22}) \end{bmatrix}$.

The following result is an extension of Voiculescu's subordination functions \cite{DVV-general} in the operator-valued framework (see also \cite{Biane1997b}). The method is adapted from the approach used in \cite{BelinschiTR2018-sub-operator-valued}. It is essentially covered by \cite[Corollary 6.6]{BelinschiV2017hyperbolic} and by Williams \cite{Williams2017}, but we provide here a more elementary argument. The reader is referred to \cite[Section 2]{BelinschiTR2018-sub-operator-valued} for some details. 

\begin{theorem}
\label{thm:subordination-unbounded}
Assume that either $\mathcal B$ is a finite dimensional $W^*$-algebra, or $y$ is a bounded operator
and $\Im M_2(\mathbb E)\left[(b_0-X)^{-1}\right]<0$ for some $b_0\in M_2(\mathcal B),\Im b_0>0$. Then there exist analytic maps 
\[
\Omega_1,\Omega_2\colon\mathbb H^+(M_2(\mathcal B))\to\mathbb H^+(M_2(\mathcal B))
\]
such that
\begin{equation}
\label{eqn:subordination-operator-in-thm}
\begin{aligned}
  \Omega_1(b)+\Omega_2(b)-b&=M_2(\mathbb E)\!\left[(\Omega_1(b)-X)^{-1}\right]^{-1}\!\!\\
  &=M_2(\mathbb E)\!\left[(\Omega_2(b)-Y)^{-1}\right]^{-1}\!\!
   =M_2(\mathbb E)\!\left[(b-\!X-\!Y)^{-1}\right]^{-1}\!.
\end{aligned}
\end{equation}
\end{theorem}

\begin{remark}\label{all}
The condition $\Im M_2(\mathbb E)\left[(b_0-X)^{-1}\right]<0$ for some $b_0\in M_2(\mathcal B),\Im b_0>0$ is in fact equivalent to
$\Im M_2(\mathbb E)\left[(b-X)^{-1}\right]<0$ for all $b\in\mathbb H^+( M_2(\mathcal B)).$ This is a fairly straightforward consequence of the classical
maximum principle for analytic functions. Indeed, first pick any state $\varphi$ on $M_2(\mathcal B)$. Define the map
\[
g\colon z\mapsto\varphi\left(M_2(\mathbb E)\left[(z\Im b_0+\Re b_0-X)^{-1}\right]\right).
\] 
It sends $\mathbb C^+$ inside the closure $($in $\mathbb C)$ of the lower half-plane.
Observing that 
\[
\Im g(i)=\varphi\left(\Im M_2(\mathbb E)\left[(b_0-X)^{-1}\right]\right)<0,
\] it follows that $g(\mathbb C^+)\subseteq\mathbb C^-$, so the 
statement holds for all $b=z\Im b_0+\Re b_0,z\in\mathbb C^+$. For $M>0$ suffciently large, $\Im (iM\Im b_0+z\Re b_0)>0$ for all $z$ in the disk centered at zero 
and of radius equal to two. The same reasoning applied to the map
\[
z\mapsto\varphi\left(M_2(\mathbb E)\left[(iM\Im b_0+z\Re b_0-X)^{-1}\right]\right)
\]
allows us to conclude 
for all $b=z\Im b_0,z\in\mathbb C^+$ as well, and yet one more application allows us to conclude for $z\Im b_0+c$ for any $c=c^*\in M_2(\mathcal B)$. 
Finally, pick an arbitrary $a\in M_2(\mathcal B),a>0$, and pick $M>0$ such that $M\Im b_0>a$. Then the map
$z\mapsto(ziM\Im b_0+(1-z)ia)$ takes values in $\mathbb H^+( M_2(\mathcal B))$ when $z$ is restricted to a small complex neighborhood of the real interval $[0,1]$
$($this is simply because $\Im(ziM\Im b_0+(1-z)ia)=\Re zMb_0+(1-\Re z)a=a+\Re z(M \Im b_0-a)$, which is clearly positive if $\Re z\ge0$ and remains positive for very small 
negative $\Re z$ since $a>0)$. We apply yet again the same reasoning to the map $z\mapsto\varphi\left(M_2(\mathbb E)\left[(ziM\Im b_0+(1-z)ia-X)^{-1}\right]\right)$
defined on a small neighborhood of $[0,1]$ in $\mathbb C$ in order to conclude for $za,z\in\mathbb C^+,$ and hence for all $c+ia$, $c=c^*\in M_2(\mathcal B)$.
 \end{remark}

\begin{proof}
{\bf Case one:} assume that both $x$ and $y$ are unbounded and dim$(\mathcal B)<\infty$. We shall obtain the $\Omega_1, \Omega_2$ as fixed points of an analytic map. 
We will use a slightly different trick, based on an observation due to Hari Bercovici \cite{Bercovici2004}. Namely, while generally neither of the maps $b\mapsto\Omega_j(b),j=1,2,$ is injective on $\mathbb H^+(M_2(\mathcal B))$, the map 
$\mathbb H^+(M_2(\mathcal B))\ni b\mapsto(\Omega_1(b),\Omega_2(b))\in \mathbb H^+(M_2(\mathcal B))\times \mathbb H^+(M_2(\mathcal B))$ {\em is} injective, and it satisfies 
$(\Omega_1(b)-h_Y(\Omega_2(b)),\Omega_2(b)-h_X(\Omega_1(b)))=(b,b)$, as the reader can verify by performing a few algebraic manipulations of \eqref{eqn:subord-operator}.
This has recently been expressed in \cite[Corollary 3.4]{BBH} as a result regarding inverses of analytic maps of two complex variables. We reproduce below an operator-valued
version of this result as

\begin{lemma}\label{f}
With the notations from Theorem \ref{thm:subordination-unbounded},
if $\boldsymbol{\Phi}\colon \mathbb H^+(M_2(\mathcal B))\times\mathbb H^+(M_2(\mathcal B))\allowbreak\to M_2(\mathcal B)\times M_2(\mathcal B)$ is given by
$\boldsymbol{\Phi}(w_1,w_2)=(w_1-h_Y(w_2),w_2-h_X(w_1))$, then there exists an analytic map $\boldsymbol{\Omega}\colon \mathbb H^+(M_2(\mathcal B))\times\mathbb H^+
(M_2(\mathcal B))\to \mathbb H^+(M_2(\mathcal B))\times\mathbb H^+(M_2(\mathcal B))$ such that $\boldsymbol{\Phi}\circ\boldsymbol{\Omega}={\rm Id}_{\mathbb H^+
(M_2(\mathcal B))\times\mathbb H^+(M_2(\mathcal B))}$. Differently stated, there exist two analytic maps $\Omega_1,\Omega_2\colon\mathbb H^+(M_2(\mathcal B))\times
\mathbb H^+(M_2(\mathcal B))\to \mathbb H^+(M_2(\mathcal B))$ such that
\begin{equation}
\boldsymbol{\Phi}(\Omega_1(b_1,b_2),\Omega_2(b_1,b_2))=(b_1,b_2),\quad (b_1,b_2)\in \mathbb H^+(M_2(\mathcal B))\times\mathbb H^+(M_2(\mathcal B)).
\end{equation}
\end{lemma}

The above result is proved in \cite{BBH} with $M_2(\mathcal B)$ replaced by $\mathbb C$. The proof in our first case (when $\mathcal B$ is finite dimensional) is virtually identical, but 
for the sake of completeness, and because of the technical difficulties that appear when one deals with several complex variables, we will provide it below.
\begin{proof}[Proof of Lemma \ref{f}]
We prove the lemma in two steps. First, we show that there exists some large $M>0$ such that the derivative $D\boldsymbol{\Phi}(iM \unit_{M_2(\mathcal B)},iM \unit_{M_2(\mathcal B)})$
is close to the identity map and thus, by the classical inverse mapping theorem for analytic maps, $\boldsymbol{\Phi}$ has an analytic inverse on some neighborhood
of $\boldsymbol{\Phi}(iM \unit_{M_2(\mathcal B)},iM \unit_{M_2(\mathcal B)})$ which sends $\boldsymbol{\Phi}(iM \unit_{M_2(\mathcal B)},iM \unit_{M_2(\mathcal B)})$ to 
$(iM \unit_{M_2(\mathcal B)},iM \unit_{M_2(\mathcal B)})$. Indeed, note that the derivative of $\boldsymbol{\Phi}$ is
\[
D\boldsymbol{\Phi}(w_1,w_2)=\begin{bmatrix} {\rm id}_{M_2(\mathcal B)} & -h_Y'(w_2) \\ -h_X'(w_1)  & {\rm id}_{M_2(\mathcal B)}\end{bmatrix}.
\]
We only need to show that $\|h'_X(iM\unit_{M_2(\mathcal B)})\|,\|h'_Y(iM\unit_{M_2(\mathcal B)})\|$ are small if $M$ is large. 
Recalling the definition of $h_X$ from above, one observes that
\[
h'_X\!(w)\!(\alpha)\!=\!M_2(\mathbb E)\!\left[(w\!-\!X)^{\!-1}\!\right]^{\!-1}\!\!M_2(\mathbb E)\!\left[(w\!-\!X)^{\!-1}\!\alpha(w\!-\!X)^{\!-1}\!\right]
\!M_2(\mathbb E)\!\left[(w\!-\!X)^{\!-1}\!\right]^{\!-1}\!\!-\alpha,
\]
for all $w\in H^+(M_2(\mathcal B)),\alpha\in M_2(\mathcal B)$. Observe that for any state $\psi$ on the algebra $M_2(\mathcal A)$ to which $X$ is affiliated,
$\psi(iM(iM\unit_{M_2(\mathcal B)}-X)^{-1})$ tends to $1$ as $M\to+\infty$. Indeed, since $X$ is self-adjoint, it has a resolution of unity, call it $e$. Then 
the map $z\mapsto \psi((z\unit_{M_2(\mathcal B)}-X)^{-1})$ is the Cauchy transform of the Borel probability measure $\psi\circ e$:
\[
\psi((z\unit_{M_2(\mathcal B)}-X)^{-1})=\int_\mathbb R\frac{1}{z-t}\,{\rm d}(\psi\circ e)(t),\quad z\in\mathbb C^+.
\]
Then the dominated convergence theorem tells us that $\lim_{M\to+\infty}\psi(iM(iM\unit_{M_2(\mathcal B)}-X)^{-1})=1$. Since the conditional expectation $M_2(\mathbb E)$
is weakly continuous, it follows that $\lim_{M\to+\infty}iMM_2(\mathbb E)\left[(iM\unit_{M_2(\mathcal B)}\!-\!X)^{-1}\right]=\unit_{M_2(\mathcal B)}$, weakly, but, since 
$\mathcal B$ is finite-dimensional, also in the norm topology on $M_2(\mathcal B)$. Thus,
\[
\lim_{M\to+\infty}\frac{M_2(\mathbb E)\left[(iM\unit_{M_2(\mathcal B)}\!-\!X)^{-1}\right]^{-1}}{iM}=\unit_{M_2(\mathcal B)},
\]
in the norm topology on $M_2(\mathcal B)$. Concerning $M_2(\mathbb E)\!\left[(w\!-\!X)^{\!-1}\!\alpha(w\!-\!X)^{\!-1}\!\right]$, let us observe first that
$\left\|iM(iM\unit_{M_2(\mathcal B)}\!-\!X)^{-1}\right\|\leq1$, and then 
\[\left\|\Re[iM(iM\unit_{M_2(\mathcal B)}\!-\!X)^{-1}]\right\|=\big\|M^2(M^2\unit_{M_2(\mathcal B)}+X^2)^{-1}\big\|
\le1,\]
 and 
\[
\left\|\Im[iM(iM\unit_{M_2(\mathcal B)}\!-\!X)^{-1}]\right\|=\big\|MX(M^2\unit_{M_2(\mathcal B)}+X^2)^{-1}\big\|\le\frac12.
\]
All these facts are direct consequences of continuous functional calculus for self-adjoint operators. Without loss of generality, we may assume that $\alpha>0$
in $M_2(\mathcal B)$. The boundedness of the quantities above together with the finite dimensionality of $\mathcal B$ imply that the cluster set as $M\to\infty$
(in the space of linear self-maps of $M_2(\mathcal B)$) of $M_2(\mathbb E)\left[iM(iM\unit_{M_2(\mathcal B)}\!-\!X)^{-1}\bullet (iM\unit_{M_2(\mathcal B)}\!-\!X)^{-1}iM\right]$
is a compact set. We claim that this cluster set contains only the identity map. Indeed, to begin with, it is clear that the norm of any cluster point is no more than one,
since by above 
\begin{align*}
 &\|M_2(\mathbb E)\left[iM(iM\unit_{M_2(\mathcal B)}\!-\!X)^{-1}\alpha (iM\unit_{M_2(\mathcal B)}\!-\!X)^{-1}iM\right]\|\\
  &\qquad\qquad\qquad\qquad\qquad\qquad\le\|iM(iM\unit_{M_2(\mathcal B)}\!-\!X)^{-1}\|^2
\|\alpha\|\le\|\alpha\|
\end{align*}
for all $M>0$. Next, we show that 
$\|M_2(\mathbb E)[\Re\{iM(iM\unit_{M_2(\mathcal B)}\!-\!X)^{-1}\}\alpha\Im\{ (iM\unit_{M_2(\mathcal B)}\!-\!X)^{-1}iM\}]\|$ tends to zero as $M\to+\infty$
for any fixed $\alpha>0$. 
One has
\begin{eqnarray*}
\lefteqn{\left[\Re\{iM(iM\unit_{M_2(\mathcal B)}\!-\!X)^{-1}\}\alpha\Im\{ (iM\unit_{M_2(\mathcal B)}\!-\!X)^{-1}iM\}\right]^*}\\
& & \mbox{}\times\left[\Re\{iM(iM\unit_{M_2(\mathcal B)}\!-\!X)^{-1}\}\alpha\Im\{ (iM\unit_{M_2(\mathcal B)}\!-\!X)^{-1}iM\}\right]\\
&\!\!\!=&\!\!MX\!(M^2\unit_{M_2(\mathcal B)}\!+\!X^2)^{-1}\alpha[M^2(M^2\unit_{M_2(\mathcal B)}\!+\!X^2)^{-1}]^2\alpha MX(M^2\unit_{M_2(\mathcal B)}\!+\!X^2)^{-1}\\
&\!\!\!\le&\!\!MX\!(M^2\unit_{M_2(\mathcal B)}\!+\!X^2)^{-1}\alpha\left\|[M^2(M^2\unit_{M_2(\mathcal B)}\!+\!X^2)^{-1}]^2\right\|\!\alpha MX(M^2\unit_{M_2(\mathcal B)}\!+\!X^2)^{-1}\\
&\!\!\!\le&\!\!MX\!(M^2\unit_{M_2(\mathcal B)}\!+\!X^2)^{-1}\alpha^2 MX(M^2\unit_{M_2(\mathcal B)}\!+\!X^2)^{-1}\\
&\!\!\!\le&\!\!\|\alpha\|^2[MX\!(M^2\unit_{M_2(\mathcal B)}\!+\!X^2)^{-1}]^2.
\end{eqnarray*}
For any state $\psi$ on $M_2(\mathcal A)$, one has
\[
\psi\left([MX\!(M^2\unit_{M_2(\mathcal B)}\!+\!X^2)^{-1}]^2\right)=\int_\mathbb R\frac{M^2t^2}{(M^2+t^2)^2}\,{\rm d}(\psi\circ e)(t).
\]
Since $M^2t^2(M^2+t^2)^{-2}\to0$ as $M\to+\infty$ pointwise and $|M^2t^2(M^2+t^2)^{-2}|\le1/4$ for all $t\in\mathbb R$, the dominated convergence theorem
applied to the probability measure $\psi\circ e$ allows us to conclude that $\lim_{M\to+\infty}\psi\left([MX\!(M^2\unit_{M_2(\mathcal B)}\!+\!X^2)^{-1}]^2\right)=0$
for any state $\psi$ on $M_2(\mathcal A)$. This, together with the above-displayed operator inequalities and the finite dimensionality of $\mathcal B$, allows us to conclude that
\begin{eqnarray*}
\lefteqn{\lim_{M\to+\infty}M_2(\mathbb E)\left[\left[\Re\{iM(iM\unit_{M_2(\mathcal B)}\!-\!X)^{-1}\}\alpha\Im\{ (iM\unit_{M_2(\mathcal B)}\!-\!X)^{-1}iM\}\right]^*\frac{}{}\right.}\\
& & \left.\frac{}{}\mbox{}\times\left[\Re\{iM(iM\unit_{M_2(\mathcal B)}\!-\!X)^{-1}\}\alpha\Im\{ (iM\unit_{M_2(\mathcal B)}\!-\!X)^{-1}iM\}\right]\right]=0,
\end{eqnarray*}
in norm, for all $\alpha>0,\alpha\in M_2(\mathcal B)$. Recall that $\Psi(AA^*)\ge\Psi(A)\Psi(A)^*$ for any completely positive map $\Psi$ on an operator algebra. 
Applying this fact to $\Psi=M_2(\mathbb E)$ and $A=\Re\{iM(iM\unit_{M_2(\mathcal B)}\!-\!X)^{-1}\}\alpha\Im\{ (iM\unit_{M_2(\mathcal B)}\!-\!X)^{-1}iM\}$,
it follows that
\[
\lim_{M\to+\infty}M_2(\mathbb E)\left[\Re\{iM(iM\unit_{M_2(\mathcal B)}\!-\!X)^{-1}\}\alpha\Im\{ (iM\unit_{M_2(\mathcal B)}\!-\!X)^{-1}iM\}\right]=0,
\]
as claimed. 

The methods from above can be used to show that 
\[
\lim_{M\to+\infty}M_2(\mathbb E)\left[\Im\{iM(iM\unit_{M_2(\mathcal B)}\!-\!X)^{-1}\}\alpha\Im\{ (iM\unit_{M_2(\mathcal B)}\!-\!X)^{-1}iM\}\right]=0.
\]
We leave this as an exercise.

Thus, we have shown that
\begin{eqnarray*}
\lefteqn{\lim_{M\to+\infty}M_2(\mathbb E)\left[iM(iM\unit_{M_2(\mathcal B)}\!-\!X)^{-1}\alpha (iM\unit_{M_2(\mathcal B)}\!-\!X)^{-1}iM\right]}\\
&= & \lim_{M\to+\infty}M_2(\mathbb E)\left[\Re\{iM(iM\unit_{M_2(\mathcal B)}\!-\!X)^{-1}\}\alpha\Re\{ (iM\unit_{M_2(\mathcal B)}\!-\!X)^{-1}iM\}\right],
\end{eqnarray*}
provided that either of those two quantities exists. In particular, they would coincide along any subsequence tending to infinity.
Let us show that the right-hand side exists and equals $\alpha$. Applying again $\Psi(AA^*)\ge\Psi(A)\Psi(A)^*$, we obtain
\begin{eqnarray*}
\lefteqn{M_2(\mathbb E)\left[\Re\{iM(iM\unit_{M_2(\mathcal B)}\!-\!X)^{-1}\}\alpha\Re\{ (iM\unit_{M_2(\mathcal B)}\!-\!X)^{-1}iM\}\right]}\\
&= & M_2(\mathbb E)\left[M^2(M^2\unit_{M_2(\mathcal B)}\!+\!X^2)^{-1}\alpha (M^2\unit_{M_2(\mathcal B)}\!+\!X^2)^{-1}M^2\right]\\
&\ge&M_2(\mathbb E)\left[M^2(M^2\unit_{M_2(\mathcal B)}\!+\!X^2)^{-1}\alpha^\frac12\right]M_2(\mathbb E)\left[\alpha^\frac12M^2(M^2\unit_{M_2(\mathcal B)}\!+\!X^2)^{-1}\right]\\
& = & M_2(\mathbb E)\left[M^2(M^2\unit_{M_2(\mathcal B)}\!+\!X^2)^{-1}\right]\alpha M_2(\mathbb E)\left[M^2(M^2\unit_{M_2(\mathcal B)}\!+\!X^2)^{-1}\right].
\end{eqnarray*}
From above, the right-hand side converges to $\alpha$ as $M\to+\infty$. Thus, any cluster point at infinity of 
$M_2(\mathbb E)\left[\Re\{iM(iM\unit_{M_2(\mathcal B)}\!-\!X)^{-1}\}\bullet\Re\{ (iM\unit_{M_2(\mathcal B)}\!-\!X)^{-1}iM\}\right]$ 
is a completely positive map $P$ which satisfies the condition that $P(\alpha)\ge\alpha$ for any $\alpha>0$. However, as already seen,
\[
\left\|M_2(\mathbb E)\left[\Re\{iM(iM\unit_{M_2(\mathcal B)}\!-\!X)^{-1}\}\alpha\Re\{ (iM\unit_{M_2(\mathcal B)}\!-\!X)^{-1}iM\}\right]\right\|\le\|\alpha\|
\]
for all $\alpha\in M_2(\mathcal B)$. Thus, necessarily $\|P(\alpha)\|\le\|\alpha\|$ as well. This forces $P={\rm Id}_{M_2(\mathcal B)}$.

We have now shown that 
\begin{eqnarray*}
\lefteqn{\lim_{M\to+\infty}M_2(\mathbb E)\left[(iM\unit_{M_2(\mathcal B)}\!-\!X)^{-1}\right]^{-1}}\\
& & \!\!\!\!\!\!\!\mbox{}\times M_2(\mathbb E)\left[(iM\unit_{M_2(\mathcal B)}\!-\!X)^{-1}\alpha
(iM\unit_{M_2(\mathcal B)}\!-\!X)^{-1}\right]M_2(\mathbb E)\left[(iM\unit_{M_2(\mathcal B)}\!-\!X)^{-1}\right]^{-1}\\
& = & \alpha
\end{eqnarray*}
for all $\alpha\in M_2(\mathcal B),\alpha>0$. By continuity, this holds as well for all $\alpha\ge0$. Since any element in a $C^*$-algebra is a linear combination of 
at most four positive operators, the above automatically holds for all $\alpha\in M_2(\mathcal B)$. Thus, for all $\alpha\in M_2(\mathcal B)$,
one has $\lim_{M\to+\infty}h_X'(iM\unit_{M_2(\mathcal B)})(\alpha)=0$, with a similar statement for $h_Y$. As in a finite dimensional algebra pointwise convergence and 
norm convergence coincide, it follows that there exists an $M_0>0$ depending only on $X,Y\in M_2(\mathcal A)$ such that 
\[
D\boldsymbol{\Phi}(iM_1\unit_{M_2(\mathcal B)},iM_2\unit_{M_2(\mathcal B)})=\begin{bmatrix} {\rm id}_{M_2(\mathcal B)} & -h_Y'(iM_2\unit_{M_2(\mathcal B)}) \\ 
-h_X'(iM_1\unit_{M_2(\mathcal B)})  & {\rm id}_{M_2(\mathcal B)}\end{bmatrix}
\]
is invertible for all $M_1,M_2>M_0$. This shows that for any $(iM_1\unit_{M_2(\mathcal B)},iM_2\unit_{M_2(\mathcal B)})$ thus chosen,
there exists a neighborhood of it in $\mathbb H^+(M_2(\mathcal B))\times\mathbb H^+(M_2(\mathcal B))$ such that the map $\boldsymbol{\Phi}$ 
is an analytic diffeomorphism from that neighborhood onto its image. This guarantees the existence of the map $\boldsymbol{\Omega}$ on the image via 
$\boldsymbol{\Phi}$ of any such a neighborhood.

The second step of the proof of our lemma is to show that any such locally defined map $\boldsymbol{\Omega}$ extends (necessarily uniquely) to all of 
$\mathbb H^+(M_2(\mathcal B))\times\mathbb H^+(M_2(\mathcal B))$. We do this via an approximation argument.
For $N>0$, observe that $\chi_{[-N,N]}(Y)$ is well-defined and belongs to $M_2(W^*(\mathcal B,y))$. Indeed, $-N\le Y\le N\iff Y^2\le N^2\iff
\begin{bmatrix} yy^* & 0 \\ 0 & y^*y \end{bmatrix}\le N^2\iff yy^*\le N^2.$ Now, $y=\mathfrak u\sqrt{y^*y}$, $y^*=\mathfrak v\sqrt{yy^*}$
for some partial isometries $\mathfrak{u,v}\in W^*(\mathcal B,y)$, so that $\begin{bmatrix}0&y\\y^*&0 \end{bmatrix}=Y=\mathfrak U\sqrt{Y^*Y}
=\begin{bmatrix} 0 & \mathfrak u\\ \mathfrak v & 0  \end{bmatrix}\begin{bmatrix} \sqrt{yy^*} & 0 \\ 0 & \sqrt{y^*y} \end{bmatrix}$.
Thus, $Y_N:=\chi_{[-N,N]}(Y)Y$ is free from $X$ over $M_2(\mathcal B)$ with respect to $M_2(\mathbb E)$ as well.
A similar statement holds for $X_N=\chi_{[-N,N]}(X)X$, and $X_N,Y_N\in M_2(\mathcal A)$ are  bounded self-adjoint operators, free from each other
with respect to $M_2(\mathbb E)$. If we denote $\boldsymbol{\Phi}_N(w_1,w_2)=(w_1-h_{Y_N}(w_2),w_2-h_{X_N}(w_1))$, then 
the existence of an $\boldsymbol{\Omega}_N\colon \mathbb H^+(M_2(\mathcal B))\times\mathbb H^+
(M_2(\mathcal B))\to\mathbb H^+(M_2(\mathcal B))\times\mathbb H^+(M_2(\mathcal B))$ such that 
$\boldsymbol{\Phi}_N\circ\boldsymbol{\Omega}_N={\rm Id}_{\mathbb H^+(M_2(\mathcal B))\times\mathbb H^+(M_2(\mathcal B))}$ follows.
Indeed, for an arbitrary pair $(b_1,b_2)\in\mathbb H^+(M_2(\mathcal B))\times\mathbb H^+(M_2(\mathcal B)),$
define 
\[
f_{(b_1,b_2),N}\colon\mathbb H^+(M_2(\mathcal B))\times\mathbb H^+(M_2(\mathcal B))\to\mathbb H^+(M_2(\mathcal B))\times\mathbb H^+(M_2(\mathcal B))
\]
by $f_{(b_1,b_2),N}(w_1,w_2)=(b_1+h_{Y_N}(w_2),b_2+h_{X_N}(w_1))$. 
We immediately observe that $\Im f_{(b_1,b_2),N}(w_1,w_2)\ge(\Im b_1,\Im b_2)>(0,0)$.
Pick $\varepsilon\in(0,+\infty)$ such that $\Im b_j>2\varepsilon\cdot\unit_{M_2(\mathcal B)}$. 
It has been shown in \cite[Lemma 2.12]{BelinschiTR2018-sub-operator-valued} that there exists an $M=M(X_N,Y_N,\varepsilon)>0$ such that
$\|h_{X_N}(w)\|<M,\|h_{Y_N}(w)\|<M$ for all $w\in \mathbb H^+(M_2(\mathcal B))$, $\Im w\ge\varepsilon\cdot\unit_{M_2(\mathcal B)}$
($M$ does depend on $N$ and may be tending to infinity if $N\to\infty$). This shows that $f_{(b_1,b_2),N}$
maps the set $\{(w_1,w_2)\in\mathbb H^+(M_2(\mathcal B))\times\mathbb H^+(M_2(\mathcal B))\colon\Im w_j\ge\varepsilon\cdot\unit_{M_2(\mathcal B)},j=1,2\}$
strictly inside (meaning at positive Euclidean distance from the topological boundary of) the bounded set
\[
\left\{(w_1,w_2)\in\mathbb H^+(M_2(\mathcal B))\times\mathbb H^+(M_2(\mathcal B))\colon
\|w_j\|<M,\Im w_j\ge2\varepsilon\cdot\unit_{M_2(\mathcal B)},j=1,2\right\}.
\]
The Earle-Hamilton Theorem guarantees that $f_{(b_1,b_2),N}$ has a unique attracting fixed point inside the above-displayed set,
call it $\boldsymbol{\Omega}_N(b_1,b_2)$. Elementary arithmetics shows that indeed $f_{(b_1,b_2),N}(\boldsymbol{\Omega}_N(b_1,b_2))=\boldsymbol{\Omega}_N(b_1,b_2)$
implies $\boldsymbol{\Phi}_N(\boldsymbol{\Omega}_N(b_1,b_2))=(b_1,b_2).$ One could prove the analyticity of the correspondence $(b_1,b_2)\mapsto
\boldsymbol{\Omega}_N(b_1,b_2)$ by direct means, but in this case it is clearly not necessary because the correspondence 
\[
(b_1,b_2)\mapsto
f_{(b_1,b_2),N}^{\circ k}(w_1,w_2)
\]
is obviously analytic for any fixed $(w_1,w_2)\in\mathbb H^+(M_2(\mathcal B))\times\mathbb H^+(M_2(\mathcal B)),k\in\mathbb N,$
and thus $\boldsymbol{\Omega}_N(b_1,b_2)=\lim_{k\to\infty}f_{(b_1,b_2),N}^{\circ k}(w_1,w_2)$ is a limit of analytic functions ($\mathcal B$ being finite dimensional,
a simple normal families argument forces the above pointwise limit to be analytic). We have now proved that for any $N\in\mathbb N$, 
there exists an analytic map $\boldsymbol{\Omega}_N\colon\mathbb H^+(M_2(\mathcal B))\times\mathbb H^+(M_2(\mathcal B))\to
\mathbb H^+(M_2(\mathcal B))\times\mathbb H^+(M_2(\mathcal B))$ such that 
\[
\boldsymbol{\Phi}_N(\boldsymbol{\Omega}_N(b_1,b_2))=(b_1,b_2),\quad (b_1,b_2)\in\mathbb H^+(M_2(\mathcal B))\times\mathbb H^+(M_2(\mathcal B)).
\]

To conclude the proof of the lemma, observe that $h_{X_N}$ converges to $h_X$ as $N\to\infty$. Indeed, it is enough to show that
$\lim_{N\to\infty}M_2(\mathbb E)\left[(w-X_N)^{-1}\right]=M_2(\mathbb E)\left[(w-X)^{-1}\right].$ 
Pick an arbitrary element $w\in\mathbb H^+( M_2(\mathcal B))$.
One has
\begin{eqnarray*}
\lefteqn{M_2(\mathbb E)\left[(w-X_N)^{-1}\right]-M_2(\mathbb E)\left[(w-X)^{-1}\right]}\\
& = & M_2(\mathbb E)\left[(w-X_N)^{-1}-(w-X)^{-1}\right]\\
& = & M_2(\mathbb E)\left[(w-X_N)^{-1}(X-X_N)(w-X)^{-1}\right]\\
& = & M_2(\mathbb E)\left[(w-X)^{-1}X(\unit_{M_2(\mathcal B)}-\chi_{[-N,N]}(X))(w-X\chi_{[-N,N]}(X))^{-1}\right].
\end{eqnarray*}
It is quite clear that $(w-X)^{-1}X(\unit_{M_2(\mathcal B)}-\chi_{[-N,N]}(X))(w-X\chi_{[-N,N]}(X))^{-1}$ tends to zero in the strong operator topology as $N\to\infty$. Indeed,
using well-known operator inequalities and estimates on resolvents,
\begin{eqnarray*}
\lefteqn{\left[(w-X)^{-1}X(\unit_{M_2(\mathcal B)}-\chi_{[-N,N]}(X))(w-X\chi_{[-N,N]}(X))^{-1}\right]}\\
&  & \mbox{}\times\left[(w-X)^{-1}X(\unit_{M_2(\mathcal B)}-\chi_{[-N,N]}(X))(w-X\chi_{[-N,N]}(X))^{-1}\right]^*\\
& \le & \|w^{-2}\|(w-X)^{-1}X(\unit_{M_2(\mathcal B)}-\chi_{[-N,N]}(X))X(w^*-X)^{-1}.
\end{eqnarray*}
The operator $(w-X)^{-1}X=(w-X)^{-1}w-1$ is bounded and does not depend on $N$. Borel functional calculus for self-adjoint operators tells us that the
operator $\unit_{M_2(\mathcal B)}-\chi_{[-N,N]}(X)$ tends to zero in the strong operator topology, hence also in the weak operator topology. 
Thus, for any vector $\xi$, one has 
\begin{eqnarray*}
\lefteqn{\left\|\left[(w-X)^{-1}X(\unit_{M_2(\mathcal B)}-\chi_{[-N,N]}(X))(w-X\chi_{[-N,N]}(X))^{-1}\right]^*\xi\right\|^2_2}\\
& \!\!= & \left\langle\left[(w-X)^{-1}X(\unit_{M_2(\mathcal B)}-\chi_{[-N,N]}(X))(w-X\chi_{[-N,N]}(X))^{-1}\right]\frac{}{}\right.\\
& & \left.\mbox{}\times\left[(w-X)^{-1}X(\unit_{M_2(\mathcal B)}-\chi_{[-N,N]}(X))(w-X\chi_{[-N,N]}(X))^{-1}\right]^*\xi,\xi\right\rangle\\
&\!\! \le & \|w^{-2}\|\left\langle(w-X)^{-1}X(\unit_{M_2(\mathcal B)}-\chi_{[-N,N]}(X))X(w^*-X)^{-1}\xi,\xi\right\rangle\\
&\!\!=& \|w^{-2}\|\left\langle(\unit_{M_2(\mathcal B)}\!-\chi_{[-N,N]}(X))X(w^*-\!X)^{-1}\xi,X(w^*-\!X)^{-1}\xi\right\rangle\!
\end{eqnarray*}
which goes to zero as $N\rightarrow \infty$. 
Since the conditional expectation $M_2(\mathbb E)$ is continuous in the strong operator topology, we conclude that $\lim_{N\to\infty}
M_2(\mathbb E)\big[(w-X)^{-1}X(\unit_{M_2(\mathcal B)}-\chi_{[-N,N]}(X))(w-X\chi_{[-N,N]}(X))^{-1}\big]=0$ in the strong operator topology, which on the finite dimensional
$M_2(\mathcal B)$ coincides with the norm topology. Thus, 
\[
\lim_{N\to\infty}M_2(\mathbb E)\!\left[(w-\!X_N)^{-1}\right]\allowbreak=M_2(\mathbb E)\left[(w-X)^{-1}\right]
\]
in norm, as claimed. By a normal families argument, this limit is uniform on compact subsets of $\mathbb H^+(M_2(\mathcal B)$), so 
\[
\lim_{N\to\infty}h_{X_N}=h_X,\quad\lim_{N\to\infty}h_{Y_N}=h_Y,
\]
uniformly on compact subsets of $\mathbb H^+(M_2(\mathcal B)$). It follows that in the norm topology,
\[
\lim_{N\to\infty}\boldsymbol{\Phi}_N=\boldsymbol{\Phi},
\]
uniformly on compact subsets of $\mathbb H^+(M_2(\mathcal B))\times\mathbb H^+(M_2(\mathcal B))$. We have shown 
in the first step that $\boldsymbol{\Phi}$ has a local inverse on a sufficiently small neighborhood of $(iM_1\unit_{M_2(\mathcal B)},iM_2\unit_{M_2(\mathcal B)})$ for all $M_1,M_2
\in(0,+\infty)$ sufficiently large. The same argument shows that so does $\boldsymbol{\Phi}_N.$ However, we have shown that $\boldsymbol{\Phi}_N$
has also a unique global inverse $\boldsymbol{\Omega}_N$. Since the uniform convergence of $\boldsymbol{\Phi}_N$ to $\boldsymbol{\Phi}$ on the small neighborhood
of $(iM_1\unit_{M_2(\mathcal B)},iM_2\unit_{M_2(\mathcal B)})$ implies also the convergence of the local inverse (on a possibly smaller neighborhood), one obtains that
locally $\lim_{N\to\infty}\boldsymbol{\Omega}_N=\boldsymbol{\Omega}$. By normality of the family $\{\boldsymbol{\Omega}_N\}_{N\in\mathbb N}$, 
we conclude that 
\[
\lim_{N\to\infty}\boldsymbol{\Omega}_N=\boldsymbol{\Omega}\text{ on }\mathbb H^+(M_2(\mathcal B))\times\mathbb H^+(M_2(\mathcal B)),
\]
and 
\[
\boldsymbol{\Phi}\circ\boldsymbol{\Omega}={\rm Id}_{\mathbb H^+(M_2(\mathcal B))\times\mathbb H^+(M_2(\mathcal B))}.
\]
This concludes the proof of our lemma.
\end{proof}

Now we can conclude the proof of our theorem in the first case, namely that of a finite dimensional $\mathcal B$:
the relation $\boldsymbol{\Phi}\circ\boldsymbol{\Omega}={\rm Id}_{\mathbb H^+(M_2(\mathcal B))\times\mathbb H^+(M_2(\mathcal B))}$
translates to $\Omega_1(b_1,b_2)-h_Y(\Omega_2(b_1,b_2))=b_1$ and $\Omega_2(b_1,b_2)-h_X(\Omega_1(b_1,b_2))=b_2$
(recall that $\boldsymbol{\Omega}=(\Omega_1,\Omega_2)$). 
Setting $b_1=b_2=b$, we have $\Omega_1(b,b)=h_Y(\Omega_2(b,b))+b=F_Y(\Omega_2(b,b))+b-\Omega_2(b,b)$
and similarly $\Omega_2(b,b)=h_X(\Omega_1(b,b))+b=F_X(\Omega_1(b,b))+b-\Omega_1(b,b)$. 
This forces 
\[
\Omega_1(b,b)+\Omega_2(b,b)-b=F_X(\Omega_1(b,b))=F_Y(\Omega_2(b,b)),\quad b\in\mathbb H^+(M_2(\mathcal B)).
\]
Using again the approximation of $X$ by $X_N$ and $Y$ by $Y_N$, and the fact (known from \cite{BelinschiTR2018-sub-operator-valued, DVV-general}) 
that for bounded variables one has $\Omega_{1,N}(b,b)+\Omega_{2,N}(b,b)-b=F_{X_N}(\Omega_{1,N}(b,b))=F_{Y_N}(\Omega_{2,N}(b,b))=F_{X_N+Y_N}(b)$,
we conclude that
\[
\Omega_1(b,b)+\Omega_2(b,b)-b=F_X(\Omega_1(b,b))=F_Y(\Omega_2(b,b))=F_{X+Y}(b),\quad b\in\mathbb H^+(M_2(\mathcal B)),
\]
We have now established that analytic subordination functions exist for any $\mathcal B$-free $x,y\in\tilde{\mathcal A}$ if dim$(\mathcal B)<\infty$,
and thus concluded the proof of the first case.

\bigskip

{\bf Case two:} assume that $y$ is bounded and $x$ is unbounded, but such that the Cauchy transform $M_2(\mathbb E)\left[(b-X)^{-1}\right]$ takes values in the operator lower 
half-plane whenever $b\in\mathbb H^+(\mathcal B)$ (this is allowed thanks to Remark \ref{all} and the hypothesis of Case two). That immediately implies $h_X$ is a well-defined 
analytic map, with values in the closure of $\mathbb H^+(M_2(\mathcal B))$. 

Pick $b_1,b_2\in\mathbb H^+(M_2(\mathcal B))$ (since there are no supplementary conceptual difficulties
involved in proving the full analogue of the result from Case one and Lemma \ref{f}, we choose to prove it, even though we will only apply it with $b_1=b_2=b$). Define 
\[
f_{(b_1,b_2)}(w)=b_1+h_Y(b_2+h_X(w)), \quad w\in\mathbb H^+(M_2(\mathcal B)).
\] 
Since $\Im\mathbb E\left[(w-X)^{-1}\right]<0$ for all $w$ in the operator upper half-plane, it follows that $h_X$ is an analytic map defined on $\mathbb H^+(M_2(\mathcal B))$ 
and taking values in its closure in $M_2(\mathcal B)$. Thus, $b_2+h_X(\mathbb H^+(M_2(\mathcal B)))\subsetneq\mathbb H^+(M_2(\mathcal B))+{b_2}/{2}$. 
By \cite[Lemma 2.12]{BelinschiTR2018-sub-operator-valued}  it follows that $f_{(b_1,b_2)}(\mathbb H^+(M_2(\mathcal B)))$ is a bounded set which is bounded away in 
norm from the complement of $\mathbb H^+(M_2(\mathcal B))$. Specifically, the lemma shows that if $\varepsilon\in(0,+\infty)$ is fixed, then for all 
$w\in \mathbb H^+(M_2(\mathcal B))+i\varepsilon\unit_{M_2(\mathcal B)}$, one has 
\[
h_Y(\mathbb H^+(M_2(\mathcal B))+i\varepsilon \unit_{M_2(\mathcal B)})\subset\{v\in\mathbb H^+(M_2(\mathcal B))\colon\|v\|\le 4\|Y\|(1+2\varepsilon^{-1}\|Y\|)\}.
\] 
Thus, if we choose 
$\varepsilon>0$ such that $\Im b_j>\varepsilon1$, $j=1,2$, then 
\[
  \begin{aligned}
  f_{(b_1,b_2)}(\mathbb H^+(M_2(\mathcal B)))\!&\subset\!b_1+h_Y(\mathbb H^+(M_2(\mathcal B))+i\varepsilon\unit_{M_2(\mathcal B)})\\
     &\subset\!\{v\in\mathbb H^+(M_2(\mathcal B))+i\varepsilon\unit_{M_2(\mathcal B)}\colon\|v\|\le 4\|Y\|
(1+2\varepsilon^{-1}\|Y\|)\}.
  \end{aligned}
\]
Thus, one may apply the Earle-Hamilton Theorem (see \cite[Section 2.1]{BelinschiTR2018-sub-operator-valued}) to the function $f_{(b_1,b_2)}$ defined on the set 
$\{v\in\mathbb H^+(M_2(\mathcal B))+i\frac\varepsilon2\unit_{M_2(\mathcal B)}\colon\|v\|\le 4\|Y\|(1+2\varepsilon^{-1}\|Y\|)+2\|b_1\|+2\|b_2\|\}$ in order to conclude that
$\{f_{(b_1,b_2)}^{\circ n}\}_{n\in\mathbb N}$ converges in norm to a fixed point $\Omega_1(b_1,b_2)$ which depends analytically on $b_1$ and $b_2$. 
One simply defines $\Omega_2(b_1,b_2)=b_2+h_X(\Omega_1(b_1,b_2))$. Choosing $b_1=b_2=b$ proves the second case of our theorem. (Note that the requirement
$\Im M_2(\mathbb E)[(b_0-X)^{-1}]<0$ was only necessary in order to guarantee the existence of $h_X$.)
\end{proof}

\begin{remark}
The results in Theorem \ref{thm:subordination-unbounded} have a fully matricial extension following \cite{DVV-general, DVVfreeanallysis2004}. For any $n\in\mathbb{N}$, we define the map 
$G_{X, n}(b)=(\mathbb{E}\otimes Id_{\mathbb{C}^{n\times n}})[(b-X\otimes I_n)^{-1}]$ for $b\in \mathcal{B}\otimes M_n(\mathbb{C})$ such that $\Im (b)>\varepsilon$ for some $\varepsilon>0$. 
The map is a noncommutative (nc) function and a remarkable observation due to Voiculescu is that the family $\{ G_{X, n} \}$
can retrieve the distribution of $X$. In this framework, by applying the same proof to Theorem \ref{thm:subordination-unbounded}, we deduce that there are nc functions $\Omega_{j, n} (j=1, 2)$ such that
\[
   G_{X+Y, n}(b)=G_{X, n}( \Omega_{1, n}(b))=G_{Y, n}( \Omega_{2, n}(b)).
\]
The interested reader is referred to \cite{BelinschiTR2018-sub-operator-valued, DVV-general, DVVfreeanallysis2004} for some details. 
\end{remark}

%%%%%%%%%%%%%%%%%%%%%%%%%%%

\section{The sum with a triangular elliptic operator}\label{section:triangular-elliptic}
\subsection{A review of the triangular elliptic operator}\label{section:tri-elliptic-review}
By enlarging the operator-valued $W^*$-probability space $(\mathcal{A},\mathbb{E},\mathcal{B})$ if necessary, we assume that $\mathcal{M}, \mathcal{N}$ are unital von Neumann subalgebras of $\mathcal{A}$ such that: 
\begin{itemize}
\item $\mathcal{M}, \mathcal{N}$ are free with amalgamation over $\mathcal{B}$ with respect to the conditional expectation $\mathbb{E}$; 
\item the tracial state $\phi$ on $\mathcal{A}$ satisfies $\phi(x)=\phi(\mathbb{E}(x))$ for any $x\in \mathcal{A}$;
\item $\mathbb{E}(x)=\phi(x)$ for any $x\in \mathcal{N}$.
\end{itemize}
The space we work on is the special case when $\mathcal{B}= L^\infty [0,1]$. The tracial state $\phi: \mathcal{A} \rightarrow \mathbb{C}$ satisfy $\phi (y) = \phi({\mathbb{E}}(y))$ for $y \in \mathcal{A}$, and 
$$\phi (f) = \int_ 0^1 f(s)ds$$ 
for $f \in \mathcal{B}$. The existence of such a subalgebra $\mathcal{N}$ can be verified as follows: Let 
$\mathcal{B}\subset \mathcal{M}$ and 
let ${\mathcal{N}}$ be a unital subalgebra that is $*$-free from $\mathcal{B}$ with respect to the tracial state $\phi$ and let ${\mathcal{N}}_1$ be the unital subalgebra generated by $\mathcal{N}\cup \mathcal{B}$. We then choose $\mathcal{A}=\mathcal{M}*_{\mathcal{B}} {\mathcal{N}}_1$ and identify $\mathcal{N}$ as some subalgebra in ${\mathcal{N}}_1\subset\mathcal{A}$. We note that $\mathcal{N}$ is $*$-free from $\mathcal{B}$ in $(\mathcal{A},\phi)$. 

Let us review some basic properties of DT-operators introduced by Dykema--Haagerup \cite{DykemaHaagerup2004-DT, DykemaHaagerup2001JFA}. 
Assume that $(X_i)_{i=1}^\infty\subset \mathcal{M}$ is a standard semicircular family of random variables that are $*$-free from $\mathcal{B}$ in the $W^*$-probability space $(\mathcal{A},\phi)$. Then for $i\in\mathbb{N}$, one can construct a quasi-nilpotent DT operator $T_i$ as the norm limit of 
\[
   T^{(i)}_n=\sum_{j=1}^n 1_{\left[ \frac{j-1}{n}, \frac{j}{n} \right]} X_i 1_{\left[ \frac{j}{n},1 \right]}
\]
as $n\rightarrow \infty$, where $1_{[a,b]}$ is viewed as an element in $\mathcal{B}$. 
It follows that $T_i \in W^*(\mathcal{B}\cup \{X_i\})$ for all $i\in\mathbb{N}$ and $(T_i)_{i=1}^\infty$ are free with amalgamation over $\mathcal{B}$ in $(\mathcal{A},\mathbb{E},\mathcal{B})$. It is shown that $(T_i, T_i^*)_{i=1}^\infty$ is a centered $\mathcal{B}$-Gaussian family in \cite[Appendix A]{DykemaHaagerup2001JFA}.

Recall that $\alpha,\beta > 0$, and $\gamma \in \mathbb{C}$ such that $\vert \gamma\vert\leq \sqrt{\alpha \beta}$. The \emph{triangular elliptic element} $y=g_{_{\alpha, \beta,\gamma}} \in \mathcal{M}$ is an operator, introduced in \cite[Section 6]{BSS2018}, whose only nonzero free cumulants are given by
\begin{equation}
  \label{eqn:cumulant-y-4}
 \begin{aligned}
\kappa(y, f y) (t) &= \gamma \int_0^1 f(s)ds, \\
\kappa(y^*, fy^*)(t) &= \bar{\gamma} \int_0^1 f(s)ds,\\
\kappa(y, fy^*)(t) &= \alpha \int_t^1 f(s)ds+ \beta \int_0^t f(s)ds,\\
\kappa(y^*, fy)(t) &= \alpha \int_0^t f(s)ds+ \beta \int_t^1 f(s)ds
\end{aligned}
\end{equation}
for every $f \in \mathcal{B}$. We say that such triangular elliptic operator $y$ has parameter $(\alpha,\beta,\gamma)$. One can construct the triangular elliptic operator ($\mathcal{B}$-circular element as studied by Dykema \cite{Dykema-2005-Bcircular}) as follows. Let $T$ be a quasi-nilpotent DT operator. By \cite[Appendix A]{DykemaHaagerup2001JFA}, 
the distribution of the pair $T, T^*$ is a $\mathcal{B}$-Gaussian distribution determined by the free cumulants given by
\[
    \kappa (T, fT^*)(t)=\int_t^1 f(s)ds, \qquad \kappa (T^*, fT)(t)=\int_0^t f(s)ds,
\]
and $\kappa (T, fT)=\kappa(T^*, fT^*)=0$.  Write $\gamma=e^{2\theta i}\vert \gamma\vert$. Then the triangular elliptic operator $y=g_{\alpha, \beta,\gamma}$ has the same $*$-moments as $e^{i\theta}(\sqrt{\alpha}T+\sqrt{\beta}T^*)$. In particular, a quasi-nilpotent DT operator is a triangular elliptic element with parameter $\alpha=1, \beta=\gamma=0$. 
When $\alpha=\beta=t$, the operator $g_{\alpha, \alpha,0}$ is just the Voiculescu's circular operator $c_t$ of variance $t$, and the operator $g_{\alpha, \alpha,\gamma}$ is the so-called \emph{twisted elliptic operator} $g_{t,\gamma}$ as studied in \cite[Section 2.4]{Zhong2021Brown}.

Let $x_0\in \log^+(\mathcal{N})$ that is $*$-freely independent from $g_{\alpha, \beta,\gamma}$ with amalgamation over $\mathcal{B}$. The main object of this paper is to calculate the Brown measure of $x_0+g_{\alpha, \beta,\gamma}$.  
In the  $M_2(\mathcal{B})$-valued $W^*$-probability space$(M_2(\mathcal{A}), M_2(\mathbb{E}),M_2(\mathcal{B}))$, we denote
\[X = \begin{bmatrix}
0 & x_0\\
x_0^* & 0 \end{bmatrix}, \;\; Y = \begin{bmatrix}
0 & y\\
y^* & 0 \end{bmatrix}.
\]
Then $X, Y$ are free over $M_2(\mathcal{B})$. This can be deduced from the connection between $\mathcal{B}$-valued free cumulants and $M_2(\mathcal{B})$-valued free cumulants similar to \cite[Section 9.3, Proposition 13]{MingoSpeicherBook}. 

The following result adapts the identification used in \cite[Lemma 5.6]{Aagaard2004jfa}, where the case $x_0=0$ was considered. 
\begin{proposition}
Given $\alpha\geq \beta=t>0$, $\vert \gamma\vert\leq t$, and a bounded operator $x_0\in \mathcal{N}$, let $g_{t,\gamma}$ be a twisted elliptic operator $*$-free from $\mathcal{M}\cup\mathcal{N}$ with respect to $\phi$. Then,
\[
  x_0+g_{\alpha, \beta,\gamma} \mathop{\sim}\limits^{\mathrm{*-dist.}}
   x_0+g_{(\alpha-\beta), 0,0}+ g_{t,\gamma},
\]
where $g_{(\alpha-\beta), 0,0}$ denotes a triangular elliptic operator in $\mathcal{M}$ with parameters $(\alpha-\beta,0,0)$. 

In particular, when $\alpha=\beta$, the operator $g_{\alpha, \beta,\gamma}$ is $*$-free from $\mathcal{N}$ in the scalar-valued $W^*$-probability space $(\mathcal{A},\phi)$.
\end{proposition}
\begin{proof}
It is clear that $g_{\alpha, \beta,\gamma}$ has the same$*$-distribution 
as $g_{t,\gamma}+g_{(\alpha-\beta),0,0}$, where $g_{t,\gamma}$ and $g_{(\alpha-\beta), 0,0}$ are free with amalgamation over $\mathcal{B}$. We express $g_{t,\gamma}$ as a linear combination $e^{2\theta i} \sqrt{t}(T_i+T_i^*)$ for some quasi-nilpotent DT operators $T_i$.
The operator-valued free cumulants of $g_{t,\gamma}$ in $(\mathcal{A}, \mathbb{E}, \mathcal{B})$ are complex numbers. By applying \cite[Theorem 3.6]{NSS2002-characterization}, it follows that $g_{t,\gamma}$ is $\mathbb{C}$-Gaussian and is $*$-free from $\mathcal{B}$ with respect to $\phi$. Since $T_j\in W^*(\mathcal{B}\cup\{X_j\})$ and  $(X_i)_{i=1}^\infty\subset \mathcal{M}$ is a standard semicircular family of random variables that are $*$-free from $\mathcal{B}$ in the $W^*$-probability space $(\mathcal{A},\phi)$,  it follows that $g_{t,\gamma}=e^{2\theta i} \sqrt{t}(T_i+T_i^*)$ is $*$-free with respect to $\phi$ from other quasi-nilpotent DT operator $T_j$ in $(\mathcal{A}, \phi)$ for $j\neq i$. Hence, $g_{t,\gamma}$ is also $*$-free from $x_0+g_{(\alpha-\beta), 0,0}$ with respect to $\phi$. 
\end{proof}

Hence, if $\vert \gamma\vert\leq \min\{\alpha,\beta\}$, the Brown measure of $x_0+g_{\alpha, \beta,\gamma}$ is the same as the Brown measure of $y_0+g_{t,\gamma}$ for $y_0=x_0+g_{|\alpha-\beta|,0,0}$  where $g_{|\alpha-\beta|, 0,0}$ denotes a triangular elliptic operator with parameter $(\vert\alpha-\beta\vert,0,0)$ and has the same distribution as a scalar multiple of a quasi-nilpotent DT operator. By \cite{Zhong2021Brown}, we can express the Brown measure formula in terms of certain subordination functions determined by $y_0$ if $\vert \gamma\vert\leq \min\{\alpha,\beta\}$ and $x_0$ is bounded. We will extend the method in \cite{Zhong2021Brown} such that we are able to express the Brown measure formula for $x_0+g_{\alpha, \beta,\gamma}$ in terms of some functions determined directly by $x_0$ for any $\gamma$ such that $\vert \gamma\vert\leq \sqrt{\alpha\beta}$. In addition, we show that the main results hold for any $x_0\in \log^+(\mathcal{N})$.

\subsection{The operator-valued subordination function}
For parameter ${\bf t}=(\alpha,\beta, \gamma)$ where $\alpha,\beta>0$ and $\gamma\in\mathbb{C}$ such that
$\vert \gamma\vert\leq \sqrt{\alpha\beta}$, we denote $y = g_{_{\alpha, \beta,\gamma}}$, and we take the convention that $y=0$ if $\alpha=\beta=\gamma=0$.  For $\lambda\in\mathbb{C}$, we have
\begin{equation}
\label{eqn:defn-S-x0-y-epsilon}
S(x_0+y, \lambda, \varepsilon) = \log \Delta \left[ (x_0+ y-\lambda)^*(x_0+y-\lambda)+ \varepsilon^2  \right], \; \; \varepsilon>0. 
\end{equation}
We introduce the following notations
\begin{equation}
\label{eqn:p-lambda-def}
\begin{split}
p_\lambda^{{\bf t}}(\varepsilon) &= \phi \left\{ (\lambda-x_0-y)^* \left[ (\lambda -x_0-y)(\lambda -x_0-y)^* + \varepsilon^2 \right]^{-1}\right \}, \\
p_{\bar{\lambda}}^{{\bf t}}(\varepsilon) &= \phi \left\{ (\lambda-x_0-y) \left[ (\lambda -x_0-y)^*(\lambda -x_0-y) + \varepsilon^2 \right]^{-1}\right \},\\
q_\varepsilon^{{\bf t}}(\lambda) &=  \varepsilon\phi \left\{ \left[ (\lambda -x_0-y)^*(\lambda -x_0-y) + \varepsilon^2 \right]^{-1}\right \},
\end{split}
\end{equation}
and 
\begin{equation}
\label{eqn:defn-P-Q-derivatives}
\begin{split}
P_\lambda^{{\bf t}}(\varepsilon) &= {\mathbb{E}} \left\{ (\lambda-x_0-y)^* \left[ (\lambda -x_0-y)(\lambda -x_0-y)^* + \varepsilon^2 \right]^{-1}\right \}, \\
P_{\bar{\lambda}}^{{\bf t}}(\varepsilon) &= {\mathbb{E}} \left\{ (\lambda-x_0-y) \left[ (\lambda -x_0-y)^*(\lambda -x_0-y) + \varepsilon^2 \right]^{-1}\right \},\\
Q_\varepsilon^{{\bf t}}(\lambda) &= \varepsilon {\mathbb{E}} \left\{ \left[ (\lambda -x_0-y)(\lambda -x_0-y)^* + \varepsilon^2 \right]^{-1}\right \},\\
\widetilde{Q}_\varepsilon^{{\bf t}}(\lambda) &=\varepsilon {\mathbb{E}} \left\{ \left[ (\lambda -x_0-y)^*(\lambda -x_0-y)+ \varepsilon^2 \right]^{-1}\right \}.
\end{split}
\end{equation}
Note that $p_\lambda^{{\bf t}}(\varepsilon)$, $p_{\bar{\lambda}}^{{\bf t}}(\varepsilon)$ and $q_\varepsilon^{{\bf t}}(\lambda)$ are derivatives of $S(x_0+y,\lambda,\varepsilon)$ with respect to $\lambda, \overline{\lambda}$ and $\varepsilon$ (up to some constant). They are also related to entries of the Cauchy transform \eqref{eqn:Cauchy-transform-entries-scalar} of the Hermitian reduction for $x_0+y$. 
It is clear that, for $\lambda\in \mathbb{C}$ and $\varepsilon>0$, we have
\begin{equation}
\phi (P_\lambda^{{\bf t}}(\varepsilon))= p_\lambda^{{\bf t}}(\varepsilon), \; \phi(P_{\bar{\lambda}}^{{\bf t}}(\varepsilon) )= p_{\bar{\lambda}}^{{\bf t}}(\varepsilon),
\end{equation}
and, by the tracial property, we have
\begin{equation}
\phi (Q_\varepsilon^{{\bf t}}(\lambda))= \phi (\widetilde{Q}_\varepsilon^{{\bf t}}(\lambda))= q_\varepsilon^{{\bf t}}(\lambda).
\end{equation}
Therefore, we have 
\[
G_{X+Y}\left( \begin{bmatrix}
i\varepsilon & \lambda\\
\overline{\lambda} & i\varepsilon
\end{bmatrix} \right)=
M_2(\mathbb{E})\left(\bigg(X+Y-\begin{bmatrix}
i\varepsilon & \lambda\\
\overline{\lambda} & i\varepsilon
\end{bmatrix}\bigg)^{-1}\right)
=
\renewcommand{\arraystretch}{1.3}
\begin{bmatrix}
-i Q_\varepsilon^{\bf t}(\lambda) & P_{\bar{\lambda}}^{{\bf t}}(\varepsilon) \\ 
P_\lambda^{{\bf t}}(\varepsilon)  &  -i \widetilde{Q}_\varepsilon^{{\bf t}}(\lambda)
\end{bmatrix}.
\]
For any $\varepsilon_1, \varepsilon_2\in \mathcal{B}$, set 
\begin{equation}
\label{eqn:defn-g-ij-entries}
\begin{split}
{g}_{11}(\lambda,\varepsilon_1,\varepsilon_2) &=-i \varepsilon_2{\mathbb{E}}\left\{  \left[ (\lambda-x_0)(\lambda-x_0)^* + \varepsilon_1\varepsilon_2\right]^{-1}\right\},\\
{g}_{12}(\lambda,\varepsilon_1,\varepsilon_2) &= {\mathbb{E}}\left\{ (\lambda-x_0) \left[ (\lambda-x_0)^*(\lambda-x_0)+ \varepsilon_1\varepsilon_2\right]^{-1}\right\},\\
{g}_{21}(\lambda,\varepsilon_1,\varepsilon_2) &= {\mathbb{E}}\left\{ (\lambda-x_0)^* \left[ (\lambda-x_0)(\lambda-x_0)^* + \varepsilon_1\varepsilon_2\right]^{-1}\right\},\\
{g}_{22} (\lambda,\varepsilon_1,\varepsilon_2)&=-i \varepsilon_1 {\mathbb{E}}\left\{  \left[ (\lambda-x_0)^*(\lambda-x_0) + \varepsilon_1\varepsilon_2\right]^{-1}\right\}.\\
\end{split}
\end{equation}
For convenience, we also denote
\[
   g_{ij}=g_{ij}(\lambda,\varepsilon_1,\varepsilon_2), \qquad i,j\in\{1,2\}.
\]
Then, we have 
\[
G_X\left( \begin{bmatrix}
i\varepsilon_1 & \lambda\\
\overline{\lambda}  & i\varepsilon_2
\end{bmatrix} \right)=
\begin{bmatrix}
{g}_{11} & {g}_{12}\\
{g}_{21} &{g}_{22}
\end{bmatrix}.
\]

\begin{proposition}
	\label{prop:R-transfomr-Y}
	The operator $Y$ is an operator-valued semicircular element in the operator-valued $W^*$-probability space $(M_2(\mathcal{A}),M_2(\mathbb{E}), M_2(\mathcal{B}))$. For any $b=\begin{bmatrix}
	a_{11} & a_{12}\\
	a_{21} & a_{22}
	\end{bmatrix}\in M_2(\mathcal{B})$, the $R$-transform of $Y$ is given by
	\[
	R_Y(b)= M_2(\mathbb{E}) (YbY)=
	 \begin{bmatrix}
	\kappa (y, a_{22}y^*) & \kappa (y, a_{21}y)\\
	\kappa (y^*, a_{12}y^*) & \kappa (y^*, a_{11}y)
	\end{bmatrix}.
	\]
\end{proposition}

\begin{proof}
It follows from a result connecting matrix-valued free cumulants with free cumulants as in \cite[Section 9.3, Proposition 13]{MingoSpeicherBook}, but we have to modify it appropriately by replacing scalar-valued free cumulants by operator-valued free cumulants in $(\mathcal{A},\mathbb{E},\mathcal{B})$. 
\end{proof}

\begin{lemma}
	\label{lemma:subordination-OA-1}
	Let $y= g_{_{\alpha, \beta,\gamma}} \in {\mathcal{M}}$ and $x_0\in \widetilde{\mathcal{N}}$ be a random variable that is $*$-free from $y$ with amalgamation over $\mathcal{B}$. For any $\varepsilon>0$ and $z \in \mathbb{C}$, set
	\begin{equation} 
	\label{eqn:lambda-to-z}
	 \lambda =z-  \kappa(y, P_z^{\bf t} (\varepsilon) y),
	\end{equation}
	where $P_z^{\bf t}(\varepsilon)$ is defined in \eqref{eqn:defn-P-Q-derivatives}. 
	For 
	$\Theta(z,\varepsilon)=  \begin{bmatrix}
	i \varepsilon & z\\
	\bar{z} & i \varepsilon \end{bmatrix}$ we have subordination relation
	\[
	G_{X+Y}(\Theta(z,\varepsilon)) = G_X(\Omega_1(\Theta(z,\varepsilon))), 
	\]
	where
	\begin{equation}
	\label{eqn:4.9-in-thm}
	\Omega_1 (\Theta(z,\varepsilon)) = \begin{bmatrix}
	i \varepsilon_1 & \lambda \\
	\bar{\lambda} & i \varepsilon_2 \end{bmatrix},
	\end{equation}
	and
	\begin{equation}
	  \label{eqn:formula-epsilon}
		 \varepsilon_1 = \varepsilon+  \kappa (y, \widetilde{Q}_\varepsilon^{\bf t} (z)y^*),\quad \varepsilon_2 = \varepsilon+  \kappa( y^*, Q_\varepsilon^{{\bf t}}(z) y).
	\end{equation}
	In other words, the subordination relation $G_{X+Y}(b) = G_X(\Omega_1(b))$ for $b=\Theta(z,\varepsilon)$ can be written as
	\begin{equation}
	\label{eqn:4.10-in-thm}
		\mathbb{E}\left( \begin{bmatrix}
		 i\varepsilon & z-(x_0+y)\\
		 \overline{z}-(x_0+y)^* & i\varepsilon
		\end{bmatrix}^{-1} \right)
		  =\mathbb{E}\left(\begin{bmatrix}
		   i\varepsilon_1 & \lambda-x_0\\
		   \overline{\lambda}-x_0^* & i\varepsilon_2
		  \end{bmatrix}^{-1} \right). 
	\end{equation}
	Moreover, the above subordination relation is equivalent to
		\begin{equation}
		\label{eqn:sub-operator-1}
	\begin{split}
	{g}_{11}  = -i  Q_\varepsilon^{\bf t} (z), \; {g}_{22} = -i  \widetilde{Q}_\varepsilon^{\bf t} (z), \; {g}_{12}  = P_{\bar{z}}^{\bf t} (\varepsilon), \; {g}_{21} = P_z^{\bf t} (\varepsilon),
	\end{split}
	\end{equation}
where ${\bf t}=(\alpha,\beta,\gamma)$ and $g_{ij}$ was defined in \eqref{eqn:defn-g-ij-entries}. 
\end{lemma}

\begin{proof}	
We first verify that Theorem \ref{thm:subordination-unbounded} applies to our context. Recall that we have  $(\mathcal{B}, \phi)=(L^\infty [0,1], ds)$ and $\mathbb{E}|_{\mathcal{N}}=\phi$. It follows that
\begin{align*}
    -\mathbb{E}[ (i -X)^{-1} ]
      &=-\mathbb{E}\left(\begin{bmatrix} i & -x_0 \\ -x_0^* & i\end{bmatrix}^{-1}\right)\\
    &= \begin{bmatrix} i\phi((1+x_0x_0^*)^{-1}) & \phi((1+x_0x_0^*)^{-1}x_0) \\ \phi(x_0^*(1+x_0x_0^*)^{-1}) & i\phi((1+x_0^*x_0)^{-1})\end{bmatrix}.
\end{align*}
The imaginary part of the above matrix is  
\[
-\Im\mathbb{E}[ (i -X)^{-1} ]=\begin{bmatrix} \phi((1+x_0x_0^*)^{-1}) & 0 \\ 0 & \phi((1+x_0^*x_0)^{-1})\end{bmatrix},
\]
which is strictly positive by the faithfulness of the trace. Hence, the condition for Case two in Theorem \ref{thm:subordination-unbounded} is satisfied. 

	For $ b=  \Theta(z,\varepsilon)$, by \eqref{eqn:subordination-operator-in-thm} in Theorem \ref{thm:subordination-unbounded},
	the subordination functions satisfy 
	\begin{equation}
	\label{eqn:3.7-proof}
	\Omega_1(b)+\Omega_2(b)=b+ (G_{X+Y}(b)  )^{-1}.
	\end{equation}
 The subordination relation	$G_Y(\Omega_2(b))=G_{X+Y}(b)$
	yields
	\[
	\Omega_2(b)=R_Y \big(G_{X+Y}(b) \big) +(G_{X+Y}(b)  )^{-1},
	\]
	provided that $\lVert b^{-1}\rVert$ is small enough. 
	It follows that
    \begin{equation}
     \label{eqn:Omega-composition-formula}
    		\Omega_1(b)=b-R_Y(G_{X+Y}(b)),
    \end{equation}
	and this is true for any $b\in \mathbb{H}^+(M_2(\mathbb{B}))$ because $R_Y\circ G_{X+Y}$ is defined for any $b\in \mathbb{H}^+(M_2(\mathbb{B}))$  thanks to Proposition \ref{prop:R-transfomr-Y}. Consequently, \eqref{eqn:Omega-composition-formula} holds for any $b=\Theta(z,\varepsilon)$ with $\varepsilon>0$. 
	
	By definitions \eqref{eqn:defn-P-Q-derivatives}, the Cauchy transform can be expressed as
	\begin{equation}
	\label{eq:Cauchy-X+Y}
	\begin{split}
	G_{X+Y} (b) &= M_2(\mathbb{E})\left[(b-X-Y)^{-1} \right]\\
	&=  
	\renewcommand{\arraystretch}{1.3}
	\begin{bmatrix}
	-i  Q_\varepsilon^{\bf t} (z) & P_{\bar{z}}^{\bf t}(\varepsilon)\\
	P_z^{\bf t}(\varepsilon) & -i  \widetilde{Q}_\varepsilon^{\bf t}(z)  \end{bmatrix}.
	\end{split}
	\end{equation}
	The formula for $R$-transform of $Y$ (Proposition \ref{prop:R-transfomr-Y}) implies  
	\begin{equation*}
	\begin{split}
	R_Y (G_{X+Y}(b)) = \begin{bmatrix}
	-i  \kappa (y, \widetilde{Q}_\varepsilon^{\bf t} (z)y^*) & \kappa(y, P_{z}^{\bf t} (\varepsilon)y) \\
	\kappa(y^*, P_{\bar{z}}^{\bf t}(\varepsilon) y^*) & -i  \kappa( y^*, Q_\varepsilon^{{\bf t}}(z) y)  \end{bmatrix}. 
	\end{split}
	\end{equation*}
	Hence 
	\begin{equation}
	  \label{eqn:subordination-formula-3.18}
	\begin{split}
	\Omega_1(b) &= b- R_Y (G_{X+Y}(b)) \\
	&= \begin{bmatrix}
	i \varepsilon+ i  \kappa (y, \widetilde{Q}_\varepsilon^{\bf t} (z)y^*) & z-\kappa(y, P_{z}^{\bf t} (\varepsilon)y) \\
	\bar{z} -\kappa(y^*, P_{\bar{z}}^{\bf t}(\varepsilon) y^*) & i \varepsilon+i  \kappa( y^*, Q_\varepsilon^{{\bf t}}(z) y)  \end{bmatrix}\\
	&= \begin{bmatrix}
	i \varepsilon_1 & \lambda\\
	\bar{\lambda} & i \varepsilon_2 \end{bmatrix},
	\end{split}
	\end{equation}
	where $\lambda = z-\kappa(y, P_{z}^{\bf t} (\varepsilon)y)\in \mathbb{C}$ (note that $\kappa(y, P_{z}^{\bf t} (\varepsilon)y)= \overline{\kappa(y^*, P_{\bar{z}}^{\bf t}(\varepsilon) y^*) }$).
	Therefore, we have 
	\begin{equation}
	\label{eqn:Cauchy-X-Omega-1}
	\begin{split}
	G_X (\Omega_1(b)) &= M_2(\mathbb{E}) \left[ ( \Omega_1(b)-X)^{-1}\right]\\
	&= \begin{bmatrix}
	{g}_{11} & {g}_{12} \\
	{g}_{21} & {g}_{22} \end{bmatrix} \in M_2(\mathcal{B}).
	\end{split}
	\end{equation}
	This establishes \eqref{eqn:4.9-in-thm} and \eqref{eqn:4.10-in-thm}.
	By comparing \eqref{eq:Cauchy-X+Y} with \eqref{eqn:Cauchy-X-Omega-1}, we have
	\begin{equation*}
	\begin{split}
{g}_{11}  = -i  Q_\varepsilon^{\bf t} (z), \; {g}_{22} = -i  \widetilde{Q}_\varepsilon^{\bf t} (z), \; {g}_{12}  = P_{\bar{z}}^{\bf t} (\varepsilon), \; {g}_{21} = P_z^{\bf t} (\varepsilon).
	\end{split}
	\end{equation*}
This finishes the proof. 
\end{proof}

\begin{lemma}
	\label{lemma:D-epsilon-product-det}
	Fix $z\in\mathbb{C}$ and $\varepsilon>0$. 
	Using notations in Lemma \ref{lemma:subordination-OA-1}, let
	\[
	D = {\mathbb{E}} \left[ ((\lambda-x_0)(\lambda-x_0)^*+\varepsilon_1\varepsilon_2)^{-1}\right],
	\]
	and
	\[
	\widetilde{D} = {\mathbb{E}} \left[ ((\lambda-x_0)^*(\lambda-x_0)+\varepsilon_1\varepsilon_2)^{-1}\right].
	\]
	If $\mathcal{N}$ is $*$-free from $\mathcal{B}$ in $(\mathcal{A}, \phi)$, 
	then $\varepsilon_1 \varepsilon_2$ and $D$ are constant functions in $\mathcal{B}=L^\infty [0,1].$ Consequently,
	\begin{equation}
	\begin{split}
	D & =\widetilde{D}= \phi \left[ ((\lambda-x_0)^*(\lambda-x_0)+\varepsilon_1\varepsilon_2)^{-1}\right]. 
	\end{split}
	\end{equation}
\end{lemma}

\begin{proof}
	We note that ${Q}_\varepsilon^{\bf t} (z), \widetilde{Q}_\varepsilon^{\bf t} (z)$ are strictly positive functions in $\mathcal{B}= L^\infty [0,1]$ by the definition \eqref{eqn:defn-P-Q-derivatives}. Hence, $\kappa (y, \widetilde{Q}_\varepsilon^{\bf t} (z)y^*)>0$ and $\kappa( y^*, Q_\varepsilon^{{\bf t}}(z) y)>0$. Consequently, by the defining identity \eqref{eqn:formula-epsilon}
	\[
	    \varepsilon_1 = \varepsilon+  \kappa (y, \widetilde{Q}_\varepsilon^{\bf t} (z)y^*)>\varepsilon.
	\]
	Similarly,  $\varepsilon_2>\varepsilon$. 
	Since for any $x\in\mathcal{N}$, we have
	\begin{equation}
	\label{eqn:3.18-in-proof}
		\mathbb{E} [p(x,x^*)] =\phi[ p(x,x^*)]
	\end{equation}
	where $p$ is an arbitrary polynomial of two indeterminates.
	Given any $x\in\mathcal{N}$, $b\in\mathcal{B}$ and $n\in\mathbb{N}$, since $\mathcal{N}$ is $*$-free from $\mathcal{B}$ in $(\mathcal{A},\phi)$, we hence have 
	\[
	   \mathbb{E}( (x^*xb)^n )=	   \mathbb{E}( (xx^*b)^n ),
	\] 
	where we used the tracial property of $\phi$. If $b\in\mathcal{B}$ is invertible and $\Vert b\Vert$ is large enough, we can write 
	\[
	(xx^*+b)^{-1}=\sum_{n=0}^\infty (-1)^n b^{-1}(xx^*b^{-1})^n.
	\] 
		It follows that, if $\Vert b\Vert$ is large enough, we have 
	\begin{equation}\label{eq:E-bound}
	\mathbb{E} \left[ (xx^*+b)^{-1}\right] = \mathbb{E} \left[ (x^*x+b)^{-1}\right].
	\end{equation}
	The function $b\mapsto {\mathbb{E}} \left[ (xx^*+b)^{-1}\right]$ is a holomorphic function at $b$ that is strictly positive in the sense that $b\geq \delta>0$ for some $\delta\in \mathbb{R}$ and $b\in \mathbb{H}^+(\mathcal{B})$. Hence, the identify \eqref{eq:E-bound} holds for any $b\geq\delta>0$ by the uniqueness of holomorphic functions. 
	
	For any $x \in \widetilde{\mathcal{N}},$ let $x = u\vert {x}\vert$ be the polar decomposition of $x$. Consider the truncated operator $x_N:= u\vert x\vert \cdot \chi_{[0, N]}(\vert {x}\vert).$ It is known that $x_N \in \mathcal{N}$ 
	and $x_N$ converges to $x$ respect to the strong operator topology. We may assume that $\mathcal{A}$ is a subalgebra of the set of all bounded operators $B(\mathrm{H})$ acting on some Hilbert space $\mathrm{H}$. By the identity $A^{-1}-B^{-1} = B^{-1}(B-A)A^{-1}$, we have
	\begin{equation*}
	\begin{split}
	 &\left\Vert \left[ (xx^*+b)^{-1} - (x_Nx_N^*+b)^{-1} \right] h \right\Vert\\
	  & = \left\| (x_Nx_N^*+b)^{-1}   (x_Nx_N^*- xx^*) (xx^*+b)^{-1} h \right\|\\
	& \leq  \left\Vert (x_Nx_N^*+b)^{-1} \right\Vert \cdot \left\Vert (x_Nx_N^*- xx^*) (xx^*+b)^{-1} h \right\Vert
	\end{split}
	\end{equation*}
	for any vector $h\in\mathrm{H}$. Note that $\left\Vert (x_Nx_N^*+b)^{-1} \right\Vert \leq \Vert {b^{-1}}\Vert$ and $x_Nx_N^*$ converges to $xx^*$ in the strong operator topology. Hence,  $(x_Nx_N^*+b)^{-1}$ converges to $(xx^*+b)^{-1}$ in the strong operator topology. Similarly, we have
	$(x_N^*x_N+b)^{-1}$ converges to $(x^*x+b)^{-1}.$
	In particular, by letting $x= \lambda - x_0$ and combining \eqref{eq:E-bound}, we have 
	 \begin{equation}
	   \label{eqn:D-identity-b}
	  	{\mathbb{E}} \left[ ((\lambda-x_0)(\lambda-x_0)^*+b)^{-1}\right]
	  ={\mathbb{E}} \left[ ((\lambda-x_0)^*(\lambda-x_0)+b)^{-1}\right].
	 \end{equation}
	 Since $\varepsilon_1\varepsilon_2 \in \mathcal{B}$ and $\varepsilon_1\varepsilon_2\geq \varepsilon^2$, this implies that \eqref{eqn:D-identity-b} holds for $b=\varepsilon_1\varepsilon_2$. Hence, 
	$D=\widetilde{D}$.
	
	We next show that $\varepsilon_1\varepsilon_2$ is a constant. 
	By differentiating respect to $t$, we have 
	\[
	\varepsilon_1'=( \varepsilon+  k (y, \widetilde{Q}_\varepsilon^{\bf t} (z)y^*))'=-(\alpha-\beta) \widetilde{Q}_\varepsilon^{\bf t} (z)
	\]
	and
	\[
	\varepsilon_2'=(  \varepsilon+  \kappa( y^*, Q_\varepsilon^{{\bf t}}(z) y))'=(\alpha-\beta)Q_\varepsilon^{{\bf t}}(z).
	\]
Lemma \ref{lemma:subordination-OA-1} implies that
\[
   \widetilde{Q}_\varepsilon^{\bf t} (z)= i g_{22}= \varepsilon_1 \widetilde{D}= \varepsilon_1 D
\]
and similarly ${Q}_\varepsilon^{\bf t} (z)= i g_{11}= \varepsilon_2 D$.
Hence
\[
 (\varepsilon_1 \varepsilon_2)'
 = (\alpha-\beta)\varepsilon_1\varepsilon_2 (D-D)=0.
\]
Hence $\varepsilon_1\varepsilon_2$ is a constant function in $L^\infty [0,1]$. 
Then \eqref{eqn:3.18-in-proof} implies that $D$ is given by
\[
  D={\mathbb{E}} \left[ ((\lambda-x_0)(\lambda-x_0)^*+\varepsilon_1\varepsilon_2)^{-1}\right]
  =\phi \left[ ((\lambda-x_0)(\lambda-x_0)^*+\varepsilon_1\varepsilon_2)^{-1}\right].
\]
This finishes the proof. 
\end{proof}

\begin{lemma}\label{lem:compute-tilde-g}
	Let $g_{ij}$ be entries of the $2\times 2$ matrix-valued Cauchy transform in \eqref{eqn:Cauchy-X-Omega-1} as defined in  \eqref{eqn:defn-g-ij-entries} and assume $\alpha\neq \beta$. Then 
	${g}_{12}$ and ${g}_{21}$ are constant functions given by
	$${g}_{12} = \phi \left\{ (\lambda-x_0) \left[ (\lambda-x_0)^*(\lambda-x_0)+ \varepsilon_1\varepsilon_2\right]^{-1}\right\},$$
	and 
	$${g}_{21} = \phi \left\{ (\lambda-x_0)^* \left[ (\lambda-x_0)(\lambda-x_0)^* + \varepsilon_1\varepsilon_2\right]^{-1}\right\}.$$
	Moreover, 
	\begin{equation}
	\label{eqn:formula-g-11-22}
	\begin{split}
	{g}_{11}(t) = \frac{i \varepsilon D (\alpha- \beta)}{\beta e^{(\alpha-\beta) D} -\alpha} e^{(\alpha-\beta)Dt}, \;\; {g}_{22}(t) = \frac{i \varepsilon D ( \beta-\alpha)}{\alpha e^{(\beta-\alpha) D} -\beta} e^{(\beta-\alpha)Dt}.
	\end{split}
	\end{equation}
\end{lemma}
\begin{proof}
By Lemma \ref{lemma:D-epsilon-product-det}, the product $\varepsilon_1\varepsilon_2$ is a constant. Recall that $\mathbb{E}(x)=\phi(x)$ for any $x\in \mathcal{N}$. It follows that $g_{12}, g_{21}$ are constants given by the formulas above. 

We next observe that
\[
   g_{11}=-i\varepsilon_2 D, \qquad g_{22}=-i\varepsilon_1 D.
\]
Taking the derivative as in the proof of Lemma \ref{lemma:D-epsilon-product-det}, we have 
\[
  g_{11}'=-iD(\varepsilon_2)'=-iD(\alpha-\beta)(Q_\varepsilon^{\bf t}(z))
    =(\alpha-\beta)D g_{11}
\]
where we used $g_{11}=-iQ_\varepsilon^{\bf t}(z)$ from Lemma \ref{lemma:subordination-OA-1}. Hence, by solving the above differential equation, we see that $g_{11}(t)=C e^{(\alpha-\beta)Dt}$ for some constant $C$. We rewrite the definition of $\varepsilon_2$ using the subordination relation in Lemma \ref{lemma:subordination-OA-1} as
\[
  \varepsilon_2= \varepsilon+  \kappa( y^*, Q_\varepsilon^{{\bf t}}(z) y)=\varepsilon+i\kappa(y^*, g_{11}y).
\]
Hence, $g_{11}(t)=C e^{(\alpha-\beta)Dt}$ and 
\[
  g_{11}(t)=-i\varepsilon_2 D=-(i\varepsilon-\kappa(y^*, g_{11}y)) D.
\]
The cumulant formula \eqref{eqn:cumulant-y-4} of $y$ reads
   \[
      \kappa(y^*, g_{11}y)=\alpha \int_0^t  g_{11}(s)ds+\beta\int_t^1  g_{11}(s)ds,
   \]
which determines the value $C$ by solving the integral equation. A direct verification shows 
\[
 C=\frac{i \varepsilon D (\alpha- \beta)}{\beta e^{(\alpha-\beta) D} -\alpha}.
\]
Similarly, one can obtain the formula for $g_{22}$. 
\end{proof}

\begin{corollary}	
	\label{cor:epsilon-1-2-formulas}
For $\alpha\neq\beta$, the functions $\varepsilon_1$ and $\varepsilon_2$ in $\mathcal{B} = L^\infty [0,1]$ are given by
\begin{equation}
\label{eqn:epsilon-1-formula}
	  \varepsilon_1(t)= \frac{\varepsilon(\alpha-\beta)}{(\alpha-\beta e^{(\alpha-\beta)D})}
	e^{(1-t)(\alpha-\beta)D},
\end{equation}
and
\[
  \varepsilon_2(t)= \frac{\varepsilon(\alpha-\beta)}{(\alpha-\beta e^{(\alpha-\beta)D})}
  e^{t(\alpha-\beta)D}.
\]
Consequently,
\begin{equation}
\label{eqn:epsilon-1-2-product-formula}
	  \varepsilon_1\varepsilon_2=\varepsilon^2\left( \frac{\alpha-\beta}{\alpha-\beta e^{(\alpha-\beta)D}} \right)^2 e^{(\alpha-\beta)D}.
\end{equation}
\end{corollary}

\begin{proof}
It follows directly from $g_{11}=-i\varepsilon_2 D$, $g_{22}=-i\varepsilon_1 D$ and the formulas for $g_{11}$ and $g_{22}$ as in	
	 (\ref{eqn:formula-g-11-22}). 
\end{proof}

\begin{remark}
	\label{remark:D-inequality}
For any $\varepsilon>0$, we have $D<\frac{\log\alpha-\log\beta}{\alpha-\beta}$ where $D$ is defined in Lemma \ref{lemma:D-epsilon-product-det}. Indeed, as we shown in the proof of Lemma \ref{lemma:D-epsilon-product-det}, we have $\varepsilon_1(t)>\varepsilon>0$, hence by the formula \eqref{eqn:epsilon-1-formula} for $\varepsilon_1$, we see that $\frac{\alpha-\beta}{\alpha-\beta e^{(\alpha-\beta)D}}>0$ which yields that $D<\frac{\log\alpha-\log\beta}{\alpha-\beta}$ whenever $\alpha\neq \beta$. 
\end{remark}

\subsection{Some functions associated with the subordination function and their regularities}
In this subsection, we will establish the subordination relation via a homeomorphism $\Phi_{\alpha, \beta, \gamma}^{(\varepsilon)}$ of the complex plane for any $\varepsilon>0$; see Equation \ref{eqn:defn-Phi-regularization-map} for the definition.  

It follows from Corollary \ref{cor:epsilon-1-2-formulas} that
\begin{equation}
\label{eqn:phi-epsilon-1}
\phi(\varepsilon_1)=\phi(\varepsilon_2)=\frac{\varepsilon (e^{(\alpha-\beta)D}-1)}{D(\alpha-\beta e^{(\alpha-\beta)D})}.
\end{equation}
Then, since $D$ is constant and $g_{11}=-i\varepsilon_2 D$, we have 
 \begin{equation}
 	\label{eqn:formula-g11-D}
 	i\phi(g_{11})
 	=\phi(\varepsilon_2) D
 	=\phi(\varepsilon_2) \phi \left\{  \left[ (\lambda-x_0)(\lambda-x_0)^* + \varepsilon_1\varepsilon_2\right]^{-1}\right\}.
 \end{equation}
Since $\frac{\alpha-\beta}{\alpha-\beta e^{(\alpha-\beta)D}}>0$,
we then set
\begin{equation}
\label{eqn:epsilon-0}
\varepsilon_0=\sqrt{\varepsilon_1\varepsilon_2}=
\frac{\varepsilon (\alpha-\beta)}{\alpha-\beta e^{(\alpha-\beta)D}}e^{(\alpha-\beta)D/2}.
\end{equation}
We can then rewrite $\phi(\varepsilon_i)$ as 
\[
\phi(\varepsilon_1)=\phi(\varepsilon_2)=\varepsilon_0\frac{(e^{(\alpha-\beta)D}-1)}{D(\alpha-\beta)} e^{-(\alpha-\beta)D/2}.
\]
Our approach is to fix $\varepsilon>0$ and $\lambda\in\mathbb{C}$, and we will show that $z, \varepsilon_1, \varepsilon_2$ can be regarded as functions of $\lambda$ and $\varepsilon$. 

In light of \eqref{eqn:epsilon-1-2-product-formula}, the constant $D$ satisfies the following implicit formula
\begin{equation}
\label{eqn:det-D-implicit-formula}
D= \phi \left\{  \left[ (\lambda-x_0)^*(\lambda-x_0) + \varepsilon^2\left( \frac{\alpha-\beta}{\alpha-\beta e^{(\alpha-\beta)D}} \right)^2 e^{(\alpha-\beta)D}\right]^{-1}\right\}. 
\end{equation}
We shall show that \eqref{eqn:det-D-implicit-formula} determines a unique solution for $D<\frac{\log\alpha-\log\beta}{\alpha-\beta}$. 
Put $\sigma=(\alpha-\beta)D$, and denote 
\[
F(\sigma,\lambda,\varepsilon)=\frac{1}{\sigma}\cdot \phi\left\{ \left[ 
(\lambda-x_0)^*(\lambda-x_0) + \varepsilon^2 \frac{(\alpha-\beta)^2}{(\alpha-\beta e^{\sigma})^2 e^{-\sigma}}
\right] ^{-1} \right\}.
\]
Then \eqref{eqn:det-D-implicit-formula} is rewritten as
\begin{equation}
\label{eqn:F-sigma-det-time}
F(\sigma,\lambda,\varepsilon)=\frac{1}{\alpha-\beta}.
\end{equation}

\begin{lemma}
	\label{lemma:F-momoton}
	Assume that $\alpha>\beta>0$, let $0<\sigma<\log (\alpha/\beta)$. For any fixed $\varepsilon>0$ and $\lambda\in\mathbb{C}$, the function $\sigma\mapsto F(\sigma,\lambda,\varepsilon)$ is 
	a decreasing function on $(0,\infty)$, and 
	\[
	\lim_{\sigma\rightarrow 0^+} F(\sigma,\lambda,\varepsilon)=\infty, \qquad
	\lim_{\sigma\rightarrow \log(\alpha/\beta)}F(\sigma,\lambda,\varepsilon)=0. 
	\]
\end{lemma}

\begin{proof}
	It is easy to check that $\sigma\mapsto (\alpha-\beta e^{\sigma})^2 e^{-\sigma}$ 
	is increasing if $\sigma>\log(\alpha/\beta)$ and is decreasing if $\sigma<\log(\alpha/\beta)$. Hence, $\frac{\partial F(\sigma,\lambda,\varepsilon)}{\partial\sigma}<0$ when $\sigma<\log(\alpha/\beta)$. 
\end{proof}

\begin{proposition}
	\label{prop:epsilon0-epsilon-analyticity}
	Fix $\lambda\in\mathbb{C}$ and $\varepsilon>0$, the equation \eqref{eqn:det-D-implicit-formula} determines $D$ uniquely, and the function $D=D(\lambda,\varepsilon)$ is a $C^\infty$-function of $(\lambda, \varepsilon)$ over $\mathbb{C}\times (0,\infty)$ determined by \eqref{eqn:det-D-implicit-formula}. Consequently, the function $\varepsilon_0$ defined by \eqref{eqn:epsilon-0} is a $C^\infty$-function of $(\lambda, \varepsilon)$ over $\mathbb{C}\times (0,\infty)$, which we denote by $\varepsilon_0=\varepsilon_0(\lambda,\varepsilon)$. 
\end{proposition}
\begin{proof}
	Without losing generality, we may assume that $\alpha>\beta$ (the proof for the case $\alpha<\beta$ is similar by symmetric consideration).
	It is known from Remark \ref{remark:D-inequality} that $D<\frac{\log\alpha-\log\beta}{\alpha-\beta},$ which means $\sigma=(\alpha-\beta)D<\log(\alpha/\beta)$. The uniqueness then follows from Lemma \ref{lemma:F-momoton}. The function $(\sigma, \lambda,\varepsilon)\mapsto F(\sigma,\lambda,\varepsilon)$ is real analytic in $(\Re\lambda,\Im\lambda)$ and complex analytic in $(\sigma,\varepsilon)$.   Hence $D(\lambda,\varepsilon)$ is a $C^\infty$-function by the implicit function theorem. 
\end{proof}

	For any $\varepsilon>0$ and $\lambda \in \mathbb{C},$
	the following equations
	\begin{equation}
	\label{eqn:system-D-epsilon}
	\begin{cases} D & =   \phi \left\{  \left[ (\lambda-x_0)^*(\lambda-x_0) + \varepsilon_0^2\right]^{-1}\right\}\\ 
	\varepsilon_0^2 & = \frac{\varepsilon^2 (\alpha-\beta)^2}{(\alpha-\beta e^{(\alpha-\beta)D})^2} \cdot e^{(\alpha-\beta)D}
	\end{cases} 
	\end{equation}
	determine uniquely $\varepsilon_0=\varepsilon_0(\lambda,\varepsilon)>0$ and $D=D(\lambda,\varepsilon)<\frac{\log\alpha-\log\beta}{\alpha-\beta}$.

We now define $\Phi_{\alpha, \beta, \gamma}^{(\varepsilon)}:\mathbb{C}\rightarrow\mathbb{C}$ as follows:
\begin{equation}
\label{eqn:defn-Phi-regularization-map}
\Phi_{\alpha, \beta, \gamma}^{(\varepsilon)}(\lambda)=\lambda+\gamma \phi \left\{ (\lambda-x_0)^* \left[ (\lambda-x_0)(\lambda-x_0)^* + \varepsilon_0(\lambda,\varepsilon)^2\right]^{-1}\right\}.
\end{equation}
We summarize the relation between $\lambda$ and $z$:
\begin{itemize}
 \item Fix $\varepsilon>0$, for any $z\in\mathbb{C}$, let $\lambda =z-  \kappa(y, P_z^{\bf t} (\varepsilon) y)$ as in Lemma \ref{lemma:subordination-OA-1}, then
 by the cumulant formula, we have 
  \begin{align}
   \lambda&=z-\int_0^1 \big[P_z^{\bf t} (\varepsilon)\big](t)dt
   =z-\phi(P_z^{\bf t} (\varepsilon))\nonumber\\
     &=z-\gamma \phi \left\{ ({z}-x_0-y)^* \left[ ({z}-x_0-y)({z}-x_0-y)^* + \varepsilon^2\right]^{-1}\right\}, \label{eqn:4.30-in-work}
  \end{align}
    where $y=g_{\alpha, \beta,\gamma}$. 
 \item Fix $\varepsilon>0$, for any $\lambda\in\mathbb{C}$, we set $z=\Phi_{\alpha, \beta, \gamma}^{(\varepsilon)}(\lambda)$. We would like to see if
 $\lambda$ can be retrieved by the previous construction \eqref{eqn:4.30-in-work}. 
\end{itemize}

The following result shows that there is a natural one-to-one correspondence between $\lambda$ and $z$ via the subordination relation \eqref{eqn:4.30-in-work}
 and $z=\Phi_{\alpha, \beta, \gamma}^{(\varepsilon)}(\lambda)$. 
\begin{lemma}
	\label{lemma:Phi-varepsilon-alphabeta-injective}
	The map $\Phi_{\alpha, \beta, \gamma}^{(\varepsilon)}$ is a homeomorphism of the complex plane for any $\varepsilon>0$. Its inverse map is 
   \begin{equation*}
   	 J^{(\varepsilon)}_{\alpha, \beta,\gamma}(z)=z-\gamma \phi \left\{ ({z}-x_0-y)^* \left[ ({z}-x_0-y)({z}-x_0-y)^* + \varepsilon^2\right]^{-1}\right\},
   \end{equation*}
	where $y=g_{\alpha, \beta,\gamma}$. 
\end{lemma}
\begin{proof}
	Using notation in the proof of Lemma \ref{lemma:subordination-OA-1}, for $b\in \mathbb{H}^+(\mathcal{B})$,  we have (see \ref{eqn:Omega-composition-formula})
	\[
	\Omega_1(b)=b-R_Y(G_{X+Y}(b))=b-R_Y(G_X(\Omega_1(b))).
	\] 
	Set $H_1(b)=b+R_Y(G_X(b))$, then $H_1(\Omega_1(b))=b$ for any $b\in \mathbb{H}^+(\mathcal{B})$. Hence,  for any $p\in \Omega_1( \mathbb{H}^+(\mathcal{B}) )$, we have $\Omega_1(H_1(p))=p$. Suppose that there are $z, z' \in \mathbb{C}$ such that
	\begin{equation*}
	 J^{(\varepsilon)}_{\alpha, \beta,\gamma}(z) =  J^{(\varepsilon)}_{\alpha, \beta,\gamma}(z').
	\end{equation*}
Then, by Lemma \ref{lemma:subordination-OA-1} and Equation \eqref{eqn:4.30-in-work}, we have 
\begin{equation*}
	\Omega_1(\Theta(z,\varepsilon) )= \Omega_1(\Theta(z',\varepsilon)).
	\end{equation*}
It follows that $\Theta(z,\varepsilon) = \Theta(z',\varepsilon)$ and thus, $z = z'.$ Therefore, $ J^{(\varepsilon)}_{\alpha, \beta,\gamma}$ is an injective map. 

	Observe that 
\[
\Vert({z}-x_0-y)^*\big( ({z}-x_0-y)({z} -x_0-y)^*+\varepsilon^2 \big)^{-1}\Vert
\leq \frac{1}{2\varepsilon}.
\]
Hence $J^{(\varepsilon)}_{\alpha, \beta,\gamma}(z)\approx z$ for large $z$. In particular,  $J^{(\varepsilon)}_{\alpha, \beta,\gamma}(\infty) = \infty.$ Hence 
$J^{(\varepsilon)}_{\alpha, \beta,\gamma}$ can be considered as a $C^\infty$ map from $\mathbb{C} \cup \{\infty\} = \mathbb{S}^2$ to itself. Suppose that $J^{(\varepsilon)}_{\alpha, \beta,\gamma}$
is not surjective, then there exists $z_0 \in \mathbb{C}$ such that $z_0$ does not belong to the image of $J^{(\varepsilon)}_{\alpha, \beta,\gamma}.$ Hence, $J^{(\varepsilon)}_{\alpha, \beta,\gamma}$ is a continuous map from $\mathbb{S}^2$ into $\mathbb{S}^2 /\{z_0\} \cong \mathbb{R}^2.$ The Borsuk-Ulam theorem for dimension two (see \cite{RotmanBook} for example) states that if $f:\mathbb{S}^2\rightarrow\mathbb{R}^2$ is a continuous function then there exists $z \in \mathbb{S}^2$ such that $f(z) = f (-z)$. This contradicts to the injectivity of $J^{(\varepsilon)}_{\alpha, \beta,\gamma}$.
We then deduce that $J^{(\varepsilon)}_{\alpha, \beta,\gamma}$ is a surjective self-map of $\mathbb{S}^2$. 

By Lemma \ref{lemma:subordination-OA-1} (see Equation \ref{eqn:subordination-formula-3.18}), we deduce that $\Omega_1\big( \{\Theta(z,\varepsilon):z\in\mathbb{C} \}\big)$ 
contains any element of the form 
\[
p(\lambda)=\begin{bmatrix}
i \varepsilon_1 & \lambda\\
\overline{\lambda} & i\varepsilon_2
\end{bmatrix}, \qquad \lambda\in\mathbb{C},
\]
where $\varepsilon_1,\varepsilon_2$ are defined as in Lemma \ref{lemma:subordination-OA-1}. 
Hence, $\varepsilon_1,\varepsilon_2$ are determined by \eqref{eqn:system-D-epsilon} and are given by explicit formulas in Corollary \ref{cor:epsilon-1-2-formulas}.

Suppose $z=\Phi_{\alpha, \beta, \gamma}^{(\varepsilon)}(\lambda)$ for $\lambda\in\mathbb{C}$. 
We want to show that $\lambda=J_{\alpha, \beta,\gamma}^{(\varepsilon)}(z)$. 
Suppose $p(\lambda)=\Omega_1(b)$, where $b=\Theta(z',\varepsilon)$ for some $z'\in\mathbb{C}$. Then
$b=H_1(\Omega_1(b))=H_1(p(\lambda))$. One can check that
\[
  H_1(p(\lambda))=\Theta(z,\varepsilon)=\begin{bmatrix}
  i\varepsilon & z\\
  \overline{z} & i\varepsilon
  \end{bmatrix},
\]
which yields that $b=\Theta(z',\varepsilon)=\Theta(z,\varepsilon)$. Hence, $z=z'$ and $p(\lambda)=\Omega_1(\Theta(z,\varepsilon))$. Therefore,
\[
  \lambda=z-\kappa(y, P_z^{\bf t}(\varepsilon) y)=J_{\alpha, \beta,\gamma}^{(\varepsilon)}(z). 
\]
This finishes the proof. 
\end{proof}

The relation \eqref{eqn:Phi-regularization-preliminary-2-z} between $\lambda$ and $z$ in Lemma \ref{lemma:subordination-OA-1} plays a fundamental role in our work. 
We will show that the regularized Brown measure $\mu_{x_0+g_{_{\alpha, \beta,\gamma}}}^{(\varepsilon)}$ is the push-forward measure of the regularized Brown measure $\mu_{x_0+g_{_{\alpha, \beta,0}}}^{(\varepsilon)}$ under the map $\Phi_{\alpha, \beta, \gamma}^{(\varepsilon)}$. 

\begin{theorem}
	 For $\alpha\neq \beta$, 
	let $y= g_{_{\alpha, \beta,\gamma}} \in \mathcal{M}$ and $x_0\in \widetilde{N}$ be a random variable that is $*$-free from $y$ with amalgamation over $\mathcal{B}$ in $(\mathcal{A}, \mathbb{E}, \mathcal{B})$, and $\mathcal{N}$ is $*$-free from $\mathcal{B}$ in $(\mathcal{A}, \phi)$.  
	For any $\varepsilon>0$ and $\lambda \in \mathbb{C},$
	let $\varepsilon_0>0$ and $D<\frac{\log\alpha-\log\beta}{\alpha-\beta}$ be the solution for the system of equations \eqref{eqn:system-D-epsilon},
	and put
	\begin{equation}
	 \label{eqn:Phi-regularization-preliminary-2-z}
	  z=\lambda+	\gamma \cdot \phi \left\{ (\lambda-x_0)^* \left[ (\lambda -x_0)(\lambda -x_0)^* + \varepsilon_0^2 \right]^{-1}\right \}.
	\end{equation}
	We have the subordination equation 
	\begin{equation}
	  \label{eqn:sub-4.29}
			M_2(\phi) [G_{X+Y}(b)] =M_2(\phi)[ G_X(\Omega_1(b))]
	\end{equation}
	where 
	\begin{equation}
	\Omega_1 \left(\begin{bmatrix}
	i \varepsilon & z\\
	\bar{z} & i \varepsilon \end{bmatrix} \right) = \begin{bmatrix}
	i \varepsilon_1 & \lambda \\
	\bar{\lambda} & i \varepsilon_2 \end{bmatrix},
	\end{equation}
	and $\varepsilon_1, \varepsilon_2$ are given by Corollary \ref{cor:epsilon-1-2-formulas}. 
	Moreover, the subordination equation \eqref{eqn:sub-4.29} 
	for $b=\begin{bmatrix}
	i \varepsilon & z\\
	\bar{z} & i \varepsilon \end{bmatrix}$ gives
	\begin{equation}
	\label{eqn:p-lambda-t}
	p_\lambda^{\bf t_0} (\varepsilon_0) = p_z^{\bf t} (\varepsilon), \;\; p_{\bar{\lambda}}^{\bf t_0} (\varepsilon_0) = p_{\bar{z}}^{\bf t} (\varepsilon),
	\end{equation}
	where ${\bf t}=(\alpha,\beta,\gamma)$ and ${\bf t_0}=(0,0,0)$ following notation in \eqref{eqn:p-lambda-def}, and
	\begin{equation}
	\label{eqn:q-epsilon-t}
	q_\varepsilon^{\bf t}(z) = \frac{\varepsilon(e^{(\alpha-\beta)D}-1) }{\alpha-\beta e^{(\alpha-\beta) D}}.
	\end{equation}
\end{theorem}

\begin{proof}
By Lemma \ref{lemma:Phi-varepsilon-alphabeta-injective}, fix $\varepsilon>0$, $\lambda$ and $z$ are determined by each other by two maps $\Phi_{_{\alpha, \beta,\gamma}}^{(\varepsilon)}$ and $J_{{\alpha, \beta,\gamma}}^{(\varepsilon)}$. 
	We note that, by Lemma \ref{lem:compute-tilde-g} and the subordination relation \eqref{eqn:sub-operator-1}, we have
	\[
	z=\lambda+\kappa (y, P_z^{\bf t}(\varepsilon) y)=
	\lambda+\kappa (y, g_{21} (\lambda,\varepsilon_1,\varepsilon_2) y)
	=\lambda+\gamma \cdot p_\lambda^{\bf t_0}(\varepsilon_0),
	\]
	where we used the fact that $g_{21}(\lambda,\varepsilon_1,\varepsilon_2)$ is constant and $\varepsilon_0^2=\varepsilon_1\varepsilon_2$ (see Lemma \ref{lem:compute-tilde-g} and \eqref{eqn:p-lambda-def}) to derive the second identity. 
	The formula for subordination function $\Omega_1$ is the same as Lemma \ref{lemma:subordination-OA-1}. 
	The \eqref{eqn:p-lambda-t} follows by taking the trace for \eqref{eqn:sub-operator-1}. Note that, by the first subordination relation in \eqref{eqn:sub-operator-1}, we have
	\begin{align*}
	i\phi(g_{11})=\phi (Q_\varepsilon^{\bf t}(z))=q_\varepsilon^{\bf t}(z),
	\end{align*}
	where $Q_\varepsilon^{\bf t}(z)$ and $q_\varepsilon^{\bf t}(z)$ are defined in \eqref{eqn:p-lambda-def}. We also have 
	 $i\phi(g_{11})=\phi(\varepsilon_2)D$. By \eqref{eqn:phi-epsilon-1}, we then have
	\[
	q_\varepsilon^{\bf t}(z)=\phi(\varepsilon_2)D=\frac{\varepsilon(e^{(\alpha-\beta)D}-1) }{\alpha-\beta e^{(\alpha-\beta) D}}.
	\]
	Finishing the proof. 
\end{proof}

\subsection{The regularized Brown measures and regularized map}

The following result is a reformulation of the argument from \cite[Theorem 5.2]{Zhong2021Brown}. For completeness, we provide proof. 
\begin{lemma}
	\label{lemma:subharmonic-pushforward}
Let  $S_1, S_2$ be two $C^\infty$ subharmonic function on $\mathbb{C}$. Given $\gamma\in\mathbb{C}$, denote the map $\Phi$ by
\[
   \Phi(\lambda)=\lambda+\gamma\frac{\partial S_1(\lambda)}{\partial\lambda}.
\]
Assume that: 
\begin{enumerate}[\rm(1)]
\item the map $\Phi$ is a homeomorphism of the complex plane, and
\item for any $\lambda\in\mathbb{C}$, we have 
\begin{equation*}
\frac{\partial S_1(\lambda)}{\partial\lambda}=\frac{\partial S_2(z)}{\partial z}, \quad
\frac{\partial S_1(\lambda)}{\partial\overline\lambda}=\frac{\partial S_2( z)}{\partial \overline z},
\end{equation*}
where $z=\Phi(\lambda)$. 
\end{enumerate}
Then, the Riesz measure $\mu_2$ of $S_2$ is the push-forward measure of the Riesz measure $\mu_1$ of $S_1$ under the map $\Phi$. In other words, for any Borel measurable set $E$, we have 
\[
    \mu_2(E)=\mu_1(\Phi^{-1}(E)). 
\]
\end{lemma}
\begin{proof}
Let $\gamma=\gamma_1+i\gamma_2$. For $\lambda=\lambda_1+i\lambda_2$ and $z=z_1+iz_2$, we denote the vector fields 
\[
   P^{(1)}(\lambda_1,\lambda_2)=\frac{1}{2}\frac{\partial S_1(\lambda)}{\partial\lambda_1}, \quad
    Q^{(1)}(\lambda_1,\lambda_2)=\frac{1}{2}\frac{\partial S_1(\lambda)}{\partial\lambda_2}
\]
and
\[
  P^{(2)}(z_1,z_2)=\frac{1}{2}\frac{\partial S_2(\lambda)}{\partial z_1}, \quad
   Q^{(2)}(z_1,z_2)=\frac{1}{2}\frac{\partial S_2(\lambda)}{\partial z_2}.
\]
Then, $z=\Phi(\lambda)$ can be expressed as
\[
  \begin{aligned}
   z_1&=\lambda_1+\left( \gamma_1 P^{(1)}(\lambda_1,\lambda_2)+\gamma_2 Q^{(1)}(\lambda_1,\lambda_2) \right),\\
   z_2&=\lambda_2+\left( \gamma_2 P^{(1)}(\lambda_1,\lambda_2)-\gamma_1 Q^{(1)}(\lambda_1,\lambda_2) \right).
  \end{aligned}
\]
Denote the differential $1-$form 
\[
  d_2=-Q^{(2)}(z_1, z_2)dz_1+P^{(1)}(z_1,z_2)dz_2.
\]
Let $d_1$ be the pulled-back $1-$form of $d_2$ under the map $\Phi$. In this case, it means that we change the variable from $(z_1, z_2)$ to $(\lambda_1, \lambda_2)$ because $\Phi$ is one-to-one. We have 
\begin{align*}
	 d_1&=-Q^{(2)}(z_1, z_2)d\big(\lambda_1+ \gamma_1 P^{(1)}+\gamma_2 Q^{(1)} \big)\\
	   &\qquad\qquad +P^{(2)}(z_1, z_2)d\big(\lambda_2+ \gamma_2 P^{(1)}-\gamma_1 Q^{(1)} \big)\\
	   &=-Q^{(1)}d\big(\lambda_1+ \gamma_1 P^{(1)}+\gamma_2 Q^{(1)} \big)\\
	   &\qquad\qquad +P^{(1)}d\big(\lambda_2+ \gamma_2 P^{(1)}-\gamma_1 Q^{(1)} \big).
\end{align*}
We can rewrite it as
\[
   d_1=-Q^{(1)}d\lambda_1+P^{(1)}d\lambda_2
     +d\left[ -\gamma_1 P^{(1)}Q^{(1)} +\frac{1}{2}\big( (P^{(1)})^2-(Q^{(1)})^2 \big) \right].
\]
Hence, for any simply connected domain ${D}$ with piecewise smooth boundary, we have 
\[
  \int_{\partial\Phi(D)}d_2 =\int_{\partial D} d_1. 
\]
By Green's formula and the definitions of $1-$form $d_1$ and $1-$form $d_2$, we have 
\[
  \mu_2(\Phi(D))=\frac{1}{2\pi}\int_{\Phi(D)} \nabla^2 S_2 dz_1dz_2
    =\frac{1}{2\pi}\int_{\partial\Phi(D)}d_2
\]
and similarly, 
\[
  \mu_1(D)=\frac{1}{2\pi}\int_{\partial D} d_1.
\]
This finishes the proof. 
\end{proof}

\begin{theorem}
	\label{thm:Brown-regularized-push-forward-property}
	For $\alpha\neq\beta$, 
	the regularized Brown measure $\mu_{x_0+g_{_{\alpha, \beta,\gamma}}}^{(\varepsilon)}$ is the push-forward measure of the regularized Brown measure $\mu_{x_0+g_{_{\alpha, \beta,0}}}^{(\varepsilon)}$ under the map $\Phi_{\alpha, \beta, \gamma}^{(\varepsilon)}$. 
\end{theorem}
\begin{proof}
	Recall that $S(x,\lambda,\varepsilon)=\phi\bigg(\log \big( (x-\lambda)^*(x-\lambda)+\varepsilon^2\big) \bigg)$. 
We set  $S_1(\lambda)=S(x_0+g_{\alpha, \beta,0}, \lambda,\varepsilon)$ and $S_2(z)=S(x_0+g_{\alpha, \beta,\gamma}, z,\varepsilon)$, and $S_0(\lambda)=S(x_0,\lambda,\varepsilon_0)$,
where $\varepsilon_0=\varepsilon_0(\lambda,\varepsilon)$ and $z=
\Phi_{\alpha, \beta, \gamma}^{(\varepsilon)}(\lambda)$ defined as \eqref{eqn:defn-Phi-regularization-map}. 
 Then the map $\Phi^{(\varepsilon)}_{\alpha, \beta, \gamma}$ can be rewritten as
\[
   \Phi^{(\varepsilon)}_{\alpha, \beta, \gamma}(\lambda)=\lambda+\gamma \frac{\partial S_0}{\partial\lambda}. 
\]
The map $\Phi^{(\varepsilon)}_{\alpha, \beta, \gamma}$ is a homeomorphism by Lemma \ref{lemma:Phi-varepsilon-alphabeta-injective}. By Definition \eqref{eqn:p-lambda-def}, we have 
\[
    p_\lambda^{\bf t_0} (\varepsilon_0)=\frac{\partial S_0(\lambda)}{\partial\lambda},
    \qquad p_z^{\bf t}(\varepsilon) =\frac{\partial S_2(z)}{\partial z},
\]
where ${\bf t_0}=(0,0,0)$ and ${\bf t}=(\alpha,\beta,\gamma)$. 
In addition, since $\Phi_{_{\alpha, \beta, 0}}^{(\varepsilon)}(\lambda)=\lambda$, we also have 
\[
   p_{\lambda}^{\bf t(0)}(\varepsilon)=\frac{\partial S_1(\lambda)}{\partial\lambda},
\]
where ${\bf t(0)}=(\alpha,\beta, 0)$. 
Hence, by choosing $\gamma=0$ and an arbitrary eligible $\gamma$, the subordination relation $p_\lambda^{\bf t_0} (\varepsilon_0) = p_z^{\bf t} (\varepsilon)$ from \eqref{eqn:p-lambda-t} shows 
\[
  \frac{\partial S_0(\lambda)}{\partial\lambda}= \frac{\partial S_1(\lambda)}{\partial\lambda}=\frac{\partial S_2(z)}{\partial z}
\]
and 
\[
  \frac{\partial S_0(\lambda)}{\partial\overline\lambda}= \frac{\partial S_1(\lambda)}{\partial\overline\lambda}=\frac{\partial S_2(z)}{\partial \overline z}. 
\]
Hence, $\Phi^{(\varepsilon)}_{\alpha, \beta, \gamma}(\lambda)=\lambda+\gamma \frac{\partial S_1(\lambda)}{\partial \lambda}$ and  conditions in Lemma \ref{lemma:subharmonic-pushforward} are satisfied, which yields the desired result. 
\end{proof}

%%%%%%%%%%%%%%%%%%%%%%%%%%%%%

\section{The Brown measure formulas}

In this section, we calculate the Brown measure of $x_0+g_{_{\alpha, \beta,0}},$ where $x_0\in\log^+(\mathcal{N})\subset \widetilde{\mathcal{N}}$ is a random variable that is $*$-free from $g_{_{\alpha, \beta,0}} \in \mathcal{M}$ with amalgamation over $\mathcal{B}$ in the operator-valued $W^*$-probability space $(\mathcal{A}, \mathbb{E}, \mathcal{B})$. Recall that $\mathcal{N}$ is $*$-free from $\mathcal{B}$ in $(\mathcal{A}, \phi).$

\subsection{The limit of subordination functions}
\label{section:sub-oa-limit}
Throughout this subsection, we choose $\alpha\neq \beta$. 
Recall that the subordination relation $G_{X+Y}(b) = G_X(\Omega_1(b))$ gives 
\begin{equation*}
\begin{split}
{g}_{11}  = -i  Q_\varepsilon^{\bf t} (z), \; {g}_{22} = -i  \widetilde{Q}_\varepsilon^{\bf t} (z), \; {g}_{12}  = P_{\bar{z}}^{\bf t} (\varepsilon), \; {g}_{21} = P_z^{\bf t} (\varepsilon),
\end{split}
\end{equation*}
where ${\bf t}=(\alpha,\beta,\gamma)$ and $g_{ij}$ was defined in \eqref{eqn:defn-g-ij-entries}. 

We now turn to study the limit of $\varepsilon_0$ as $\varepsilon$ goes to zero.  Set
\[
H(\lambda)=(\phi (\vert x_0-\lambda\vert^{-2}) )^{-1}=\left(\int_0^\infty \frac{1}{u^2}d\mu_{\vert x_0-\lambda\vert}(u) \right)^{-1}.
\]
If $\alpha>\beta>0$, then $\sigma=(\alpha-\beta)D>0$. We have 
\[
\frac{1}{\alpha-\beta}=F(\sigma,\lambda,\varepsilon)
=\frac{1}{\sigma}\cdot \phi\left\{ \left[ 
(\lambda-x_0)^*(\lambda-x_0) + \varepsilon^2 \frac{(\alpha-\beta)^2}{(\alpha-\beta e^{\sigma})^2 e^{-\sigma}}
\right] ^{-1} \right\}
\leq \frac{1}{\sigma H},
\]
which implies that
\[
\sigma\leq \frac{\alpha-\beta}{H}.
\]
Similarly, if $0\leq \alpha<\beta$, then
\[
\sigma\geq \frac{\alpha-\beta}{H}.
\]
We now set
\begin{equation}
\label{eqn:denf-Omega}
{\Xi_{\alpha,\beta}}=\left\{ \lambda: H< \frac{\alpha-\beta}{\log\alpha-\log\beta}  \right\},
\end{equation}
and denote $\sigma_0(\lambda)=\min\{ \log\left(\frac{\alpha}{\beta}\right), \frac{\alpha-\beta}{H}\}$
if $\alpha>\beta>0$. 

For any $\varepsilon\geq 0$, denote the function $f_\varepsilon$ of $\lambda$ by 
\[
f_\varepsilon(\lambda)=\phi\big((\vert x_0-\lambda\vert^2+\varepsilon)^{-1} \big).
\]
For any $\varepsilon>0$, the function $f_\varepsilon$ is a continuous function of $\lambda$. As the increasing limit of the $(f_\varepsilon)_{\varepsilon>0}$, we deduce that $f_0$ is lower semi-continuous. Hence, the set ${\Xi_{\alpha,\beta}}$ is open. 
\begin{lemma}
	\label{lemma:limit-s-lambda}
	If $\lambda\in{\Xi_{\alpha,\beta}}$, then $\lim_{\varepsilon\rightarrow 0^+}\varepsilon_0(\lambda,\varepsilon)$ exists and the limit is a finite positive number $s$ determined by
	\begin{equation}
	\label{eqn:fixed-pt-limit-s-0}
	\frac{\log\alpha-\log\beta}{\alpha-\beta} =
	\phi\left\{ \left[ 
	(\lambda-x_0)^*(\lambda-x_0) + s^2
	\right] ^{-1} \right\}.
	\end{equation}
	If $\lambda\notin {\Xi_{\alpha,\beta}}$, then $\lim_{\varepsilon\rightarrow 0^+}\varepsilon_0(\lambda,\varepsilon)=0$. 
	Denote the limit by $s=s(\lambda)$ for $\lambda\in\mathbb{C}$, then $s(\lambda)$ is a $C^\infty$ function of $\lambda\in{\Xi_{\alpha,\beta}}$. 
\end{lemma}
\begin{proof}
	We recall that $F(\sigma,\lambda,\varepsilon)=1/(\alpha-\beta)$ (see Equation \eqref{eqn:F-sigma-det-time}). For simplicity, we denote $F(\sigma) := F(\sigma,\lambda,\varepsilon).$ Moreover, 
		for $\varepsilon>0$, $\sigma=(\alpha-\beta)D$ is the unique solution of $F(\sigma)$ such that $\sigma<\log(\alpha/\beta)$.
	By symmetric consideration, we may assume that $\alpha>\beta>0$. 
	Then, $\lambda\in\Xi_{\alpha,\beta}$ if and only if 
	\[
	\log\left(\frac{\alpha}{\beta}\right)< \frac{\alpha-\beta}{H},
	\]
	which implies $\sigma_0(\lambda)=\log(\alpha/\beta)$. By Proposition \ref{prop:epsilon0-epsilon-analyticity}, $\sigma=\sigma(\lambda,\varepsilon)$ is a $C^\infty$-function of $(\lambda, \varepsilon)$ over $\mathbb{C} \times (0, \infty)$. We first claim that $\lim_{\varepsilon\rightarrow 0^+}\sigma(\lambda,\varepsilon)=\log(\alpha/\beta)$. Assume that there exists some $\varepsilon_k\rightarrow 0^+$ such that $\sigma=\sigma(\lambda, \varepsilon_k)=(\alpha-\beta)D(\lambda, \varepsilon_k)$ converges to some $\theta<\log(\alpha/\beta)$. Then, we have 
	\begin{align*}
	\lim_{\varepsilon_k\rightarrow 0^+} \varepsilon_0(\lambda,\varepsilon_k)=\lim_{ \varepsilon_k \rightarrow 0^+}
	\frac{\varepsilon_k (\alpha-\beta)}{\alpha-\beta e^{(\alpha-\beta)D}}e^{(\alpha-\beta)D/2}=0.
	\end{align*}
	Consequently, 
	\[
	\lim_{ \varepsilon_k\rightarrow 0^+} \sigma F(\sigma)= H^{-1}.
	\]
	On the other hand,
	\[
	\lim_{ \varepsilon_k\rightarrow 0^+} \sigma F(\sigma)=\frac{\theta}{\alpha-\beta}<\frac{\log\alpha-\log\beta}{\alpha-\beta}.
	\]
	This contradicts to $\lambda\in{\Xi_{\alpha,\beta}}$. Hence, $\lim_{\varepsilon\rightarrow 0^+}\sigma=\log(\alpha/\beta)= \sigma_0(\lambda)$ if $\lambda\in{\Xi_{\alpha,\beta}}$. By using the identity $D=\sigma F(\sigma)=\frac{\sigma}{\alpha-\beta}$, we have
	\[
	\frac{\log\alpha-\log\beta}{\alpha-\beta} =\lim_{\varepsilon\rightarrow 0^+}\sigma F(\sigma)=\lim_{\varepsilon\rightarrow 0^+}\phi\big\{ [\vert \lambda-x_0\vert^2+\varepsilon_0^2]^{-1} \big\}
	\]
	which implies that $\lim_{\varepsilon\rightarrow 0^+}\varepsilon_0$ exists and is determined by \eqref{eqn:fixed-pt-limit-s-0}.
	
	For any $s>0$ and $\lambda\in\mathbb{C}$, the map $s\mapsto  \phi\left\{ \left[ 
	\vert \lambda-x_0\vert^2 + s^2
	\right] ^{-1} \right\}$ is $C^\infty$ on $[x, \infty)$, and for any $s>x$, the map $\lambda\mapsto \phi\left\{ \left[ 
	|\lambda -x_0|^2 + s^2
	\right] ^{-1} \right\}$ is $C^\infty$ on $\mathbb{C}$. Then, the implicit function theorem implies that as unique solution of \eqref{eqn:fixed-pt-limit-s-0}, the function $s(\lambda)$ is $C^\infty$ on ${\Xi_{\alpha,\beta}}$. 
	
	Similarly, if $\lambda\notin {\Xi_{\alpha,\beta}}$, by assuming $\lim_{\varepsilon\rightarrow 0^+}\varepsilon_0(\lambda,\varepsilon)\neq 0$, we must have $\lim_{\varepsilon\rightarrow 0^+} \sigma=\log(\alpha/\beta)$. This yields that $H<\frac{\alpha-\beta}{\log\alpha-\log\beta}$. This contradiction implies that if $\lambda\notin {\Xi_{\alpha,\beta}}$, then $\lim_{\varepsilon\rightarrow 0^+}\varepsilon_0(\lambda,\varepsilon)=0$.
\end{proof}

\begin{remark}
	\label{remark:3.11-relation-with-semicircle}
	The function $s(\lambda)$ can be viewed as the boundary value of scalar-valued subordination function parameterized by $\lambda\in\mathbb{C}$.   
	See Definition \ref{eqn:def-mu1-scalar-subordination} and \cite[Definition 3.3 and Proposition 3.5]{Zhong2021Brown} for details. 
\end{remark}

\begin{lemma}
	\label{lemma:convergence-uniform-epsilon-0}
	The convergence  $\lim_{\varepsilon\rightarrow 0}\varepsilon_0(\lambda,\varepsilon)=s(\lambda)$ is uniform over any compact subset in $\lambda\in\mathbb{C}$. If $x_0$ is bounded, the convergence is uniform over $\mathbb{C}$. 
\end{lemma}
\begin{proof}
We first show that $s(\lambda)$ is a continuous function in $\mathbb{C}$. 
	Since $s(\lambda)$ is $C^\infty$ in the open set $\Xi_{\alpha, \beta}$, to show $s(\lambda)$ is continuous, it remains to show that for any $\lambda_0\in \mathbb{C}\backslash\Xi_{\alpha, \beta}$ and a sequence $\{\lambda_n\}\subset \Xi_{\alpha, \beta}$ converging to $\lambda_0$, we have 
	\[
	\lim_{n\rightarrow\infty} s(\lambda_n)= s(\lambda_0)=0. 
	\]
	Assume this is not true, by dropping to a subsequence if necessary, we may assume that there exists $\delta>0$ such for all $n$, $s(\lambda_n)>\delta$. In this case, note that
	\[
	\phi\left\{ \left[ 
	(\lambda_n-x_0)^*(\lambda_n-x_0) + s(\lambda_n)^2
	\right] ^{-1} \right\}<\phi\left\{ \left[ 
	(\lambda_n-x_0)^*(\lambda_n-x_0) + \delta^2
	\right] ^{-1} \right\}.
	\]
	By using \ref{eqn:fixed-pt-limit-s-0} and passing to the limit, we have
	\[
	\frac{\log\alpha-\log\beta}{\alpha-\beta}\leq \phi\left\{ \left[ 
	(\lambda_0-x_0)^*(\lambda_0-x_0) + \delta^2
	\right] ^{-1} \right\}
	\]
	which yields that $s(\lambda_0)\geq \delta>0$. It contradicts to our choice $\lambda_0\in \mathbb{C}\backslash\Xi_{\alpha, \beta}$. Hence $\lim_{n\rightarrow\infty} s(\lambda_n)\rightarrow s(\lambda_0)=0$ and the function $s(\lambda)$ is continuous in $\mathbb{C}$. 
	
	Without losing generality, we now assume that $\alpha>\beta$. For any $\lambda\in\mathbb{C}$, for $0<\varepsilon_1<\varepsilon_2$, note that $F(\sigma,\lambda,\varepsilon)$ decreases as $\varepsilon$ increases, hence we have 
	\[
	\frac{1}{\alpha-\beta}=F(\sigma_1, \lambda,\varepsilon_1)=F(\sigma_2,\lambda,\varepsilon_2)
	<F(\sigma_2,\lambda,\varepsilon_1)
	\]
	where $\sigma_i=\sigma(\lambda,\varepsilon_i) \, (i=1,2)$. By Lemma \ref{lemma:F-momoton}, we deduce that
	$\sigma_1>\sigma_2$. Recall that $\sigma(\lambda,\varepsilon_i) = (\alpha-\beta)D(\lambda,\varepsilon_i)$ and 
	\[ 
	D(\lambda,\varepsilon_i)= \phi \left\{  \left[ (\lambda-x_0)^*(\lambda -x_0) + \varepsilon_0(\lambda,\varepsilon_i)^2\right]^{-1}\right\}.
	\]
	Hence, $\varepsilon_0(\lambda,\varepsilon_1)<\varepsilon_0(\lambda,\varepsilon_2)$. We then conclude that
	$\varepsilon_0(\lambda,\varepsilon)$ converges to $s(\lambda)$ uniformly in any compact subset of $\mathbb{C}$ as $\varepsilon$ tends to zero by Dini's theorem.  
	
	Assume now that $x_0$ is a bounded operator. Since $\Xi_{\alpha, \beta}$ is a bounded set, to prove the convergence is uniform over $\mathbb{C}$, it suffices to show that $\varepsilon_0(\lambda,\varepsilon)<2\varepsilon$ if $\vert \lambda\vert$ is sufficiently large. Recall that
	\[
	\varepsilon_0(\lambda,\varepsilon)=\frac{\varepsilon(\alpha-\beta)}{\alpha-\beta e^\sigma} e^{\sigma/2},
	\]
	where $\sigma=\sigma(\lambda,\varepsilon)$. 
	Hence, $\varepsilon_0(\lambda,\varepsilon)\geq 2\varepsilon$ if and only if
	\begin{equation}
	\label{eqn:3.33-in-proof}
	\alpha-\beta>2\left( \alpha e^{-\sigma/2} -\beta e^{\sigma/2} \right).
	\end{equation}
	But $\sigma=(\alpha-\beta)D\leq (\alpha-\beta)/(M^2+4\varepsilon^2)$ if $\vert \lambda\vert>M+||x_0||$ and $\varepsilon_0(\lambda,\varepsilon)\geq 2\varepsilon$ by the formula for $D$ in \eqref{eqn:system-D-epsilon}. It follows that $\sigma$ could be arbitrarily small by choosing $M$ large enough. 
	Observe that \eqref{eqn:3.33-in-proof} can not hold when $\sigma$ is sufficiently small. Hence, we may choose $M$ large enough, so that $\varepsilon_0(\lambda,\varepsilon)<2\varepsilon$ for any $\vert \lambda\vert>M+\Vert x_0\Vert$. 
\end{proof}

\subsection{The Brown measure formula}
We adapt the proof of \cite[Theorem 4.6]{Zhong2021Brown} to calculate the density formula $\mu_{x_0+g_{_{\alpha, \beta,0}}}$ in ${\Xi_{\alpha,\beta}}$.
\begin{theorem}
\label{thm:Brown-formula-gamma-0}
The support of the Brown measure $\mu_{x_0+g_{_{\alpha, \beta,0}}}$ of $x_0+g_{_{\alpha, \beta,0}}$ is the closure of the set ${\Xi_{\alpha,\beta}}$ defined by \eqref{eqn:denf-Omega}, which can be rewritten as
\[
 	{\Xi_{\alpha,\beta}}=\left\{ \lambda: \phi (\vert x_0-\lambda\vert^{-2}) > \frac{\log\alpha-\log\beta}{\alpha-\beta}  \right\}.
\]
Let $s(\lambda)$ be the function determined by
\begin{equation}
 \label{eqn:4.1-in-proof}
	   \frac{\log\alpha-\log\beta}{\alpha-\beta} =
	\phi\left\{ \left[ 
	(\lambda-x_0)^*(\lambda-x_0) + s(\lambda)^2
	\right] ^{-1} \right\}
\end{equation}
as in Lemma \ref{lemma:limit-s-lambda}, and set 
\[h(\lambda,s)=(\lambda-x_0)^*(\lambda-x_0)+s^2\] and \[k(\lambda,s)=(\lambda-x_0)(\lambda-x_0)^*+s^2.\] 
The probability measure $\mu_{x_0+g_{_{\alpha, \beta,0}}}$ is absolutely continuous in the open set ${\Xi_{\alpha,\beta}}$ with density given by
\begin{equation}
\label{eqn:density-Brown-tri-circular}
  \frac{1}{\pi}s^2(\lambda)\phi(h^{-1}k^{-1})+\frac{1}{\pi}\frac{|\phi((\lambda-x_0)(h^{-1})^2)|^2}{\phi((h^{-1})^2)}.
\end{equation}
\end{theorem}

We need the following result. It should be well-known, but we include a proof for the sake of completeness.
\begin{lemma}\label{lem:tech-thm}
\begin{enumerate}[\rm (1)]
\item Let $x \in \widetilde{\mathcal{N}},$ for any $a\in (0,\infty)$ we have 
\begin{equation} 
x (x^* x+ a)^{-1} = (xx^* + a)^{-1} x.
\end{equation}
\item For $h(\lambda, s)$ and $k(\lambda, s)$ given in Theorem \ref{thm:Brown-formula-gamma-0}, we have
\begin{equation}\label{eq:partial-unbounded1}
\frac{\partial}{\partial \overline{\lambda}}\phi \left( h^{-1} \right) = -\phi \left[ h^{-1} \left( (\lambda-x_0) +2s \frac{\partial s}{\partial\overline{\lambda}} \right)  h^{-1}\right],
\end{equation}
and 
\begin{equation}\label{eq:partial-unbounded2}
\begin{split}
\frac{\partial}{\partial \overline{\lambda}}\phi\left(  h^{-1} (\lambda-x_0)^*\right)
=\phi\left[ h^{-1}- h^{-1}\left( (\lambda-x_0)+2s \frac{\partial s}{\partial\overline{\lambda}} \right)h^{-1} \right]. 
\end{split}
\end{equation}
\end{enumerate}
\end{lemma}

\begin{proof}
For (1), it is equivalent to show that $(xx^*+a)x(x^*x+a)^{-1}=x$. Since the affiliated algebra $\widetilde{\mathcal{N}}$ is an algebra, this follows from the following calculation
\[
   xx^*x(x^*x+a)^{-1}=x(x^*x+a-a)(x^*x+a)^{-1}=x-ax(x^*x+a)^{-1}.
\] 

For (2), if $x_0$ is bounded, we can obtain the result by applying \cite[Lemma 3.2]{HaagerupT2005Annals}.  For unbounded $x_0$, we denote
\[
 \begin{aligned}
 h=h(\lambda,\overline{\lambda},s)&=(\lambda-x_0)^*(\lambda-x_0)+s^2,\\ k=k(\lambda,\overline{\lambda},s)&=(\lambda-x_0)(\lambda-x_0)^*+s^2,
 \end{aligned}
\]
where $s=s(\lambda,\overline{\lambda})$ determined by \eqref{eqn:4.1-in-proof}. 
 By the identity $A^{-1} - B^{-1} = B^{-1}(B-A) A^{-1}$ we have
 \begin{align*}
  &h(\lambda,\overline{\lambda'},s)^{-1}
     -h(\lambda,\overline{\lambda},s)^{-1}\\
    &=-h(\lambda,\overline{\lambda'},s)^{-1}
     ((\overline{\lambda'}-\overline{\lambda})(\lambda-x_0) + s(\lambda,\overline{\lambda'})^2-s(\lambda,\overline{\lambda})^2 ) h(\lambda,\overline{\lambda},s)^{-1}
 \end{align*}
 Since $s$ is a $C^\infty$ function of $\lambda$ and $s>0$, it follows that
 \[
   \lim_{\overline{\lambda'}\rightarrow\overline{\lambda}}
     h(\lambda,\overline{\lambda'},s(\lambda,\overline{\lambda'}))^{-1}
      = h(\lambda,\overline{\lambda},s(\lambda,\overline{\lambda}))^{-1}.
 \]
Therefore, by using the notation of complex derivative, the partially derivation with respect to $\overline{\lambda}$ can be calculated as
\begin{align*}
\frac{\partial}{\partial \overline{\lambda}}\phi\left( h^{-1} \right) & =
   \lim_{\overline{\lambda'}\rightarrow\lambda}
     \frac{\phi(h(\lambda,\overline{\lambda'},s)^{-1})
     	-\phi(h(\lambda,\overline{\lambda},s)^{-1}))}{\overline{\lambda'}-\overline{\lambda}}\\
    &=-\lim_{\overline{\lambda'}\rightarrow\lambda}
    \phi\big[ h(\lambda,\overline{\lambda'},s)^{-1}
    (\lambda-x_0)  h(\lambda,\overline{\lambda},s)^{-1} \big]\\
    &\qquad -
     \lim_{\overline{\lambda'}\rightarrow\lambda} \left[   \frac{s(\lambda,\overline{\lambda'})^2-s(\lambda,\overline{\lambda})^2 )}{\overline{\lambda'}-\overline{\lambda}} 
      \cdot \phi\big[  h(\lambda,\overline{\lambda'},s)^{-1}
        h(\lambda,\overline{\lambda},s)^{-1}  \big]
      \right]\\
    &= -\phi \left[ h^{-1} \left( (\lambda-x_0) +2s \frac{\partial s}{\partial\overline{\lambda}} \right)  h^{-1}\right].\\ 
\end{align*}
The proof for \eqref{eq:partial-unbounded2} is similar.
\end{proof}

\begin{proof}[Proof of Theorem \ref{thm:Brown-formula-gamma-0}]
 Set ${\bf t}=(\alpha,\beta,0)$, ${\bf t_0}=(0,0,0)$ and  $b=\begin{bmatrix}
 i\varepsilon & \lambda\\
 \overline{\lambda} & i\varepsilon
\end{bmatrix}$, we have
\[
 M_2(\phi) [G_{X+Y}(b)]=\begin{bmatrix}
 -i  q_\varepsilon^{\bf t} (\lambda) & p_{\bar{\lambda}}^{\bf t}(\varepsilon)\\
 p_\lambda^{\bf t}(\varepsilon) & -i  {q}_\varepsilon^{\bf t}(\lambda)  \end{bmatrix}
\]
Applying \eqref{eqn:p-lambda-t} for $\gamma=0$, we have 
\[
  p_\lambda^{\bf t}(\varepsilon)=p_\lambda^{\bf t_0}(\varepsilon_0)=
    \phi \left\{ (\lambda-x_0)^* \left[ (\lambda -x_0)(\lambda -x_0)^* + \varepsilon_0^2 \right]^{-1}\right \}.
\]
By Lemma \ref{lemma:limit-s-lambda}, we have 
\begin{equation}
\label{eqn:limit-partial-lambda-in-proof}
	\begin{aligned}
	  &\lim_{\varepsilon\rightarrow 0^+}\phi \left\{ (\lambda-x_0)^* \left[ (\lambda -x_0)(\lambda -x_0)^* + \varepsilon_0^2 \right]^{-1}\right \}\\
=&\phi \left\{ (\lambda-x_0)^* \left[ (\lambda -x_0)(\lambda -x_0)^* + s(\lambda)^2 \right]^{-1}\right \}\\
=&\phi \left\{  \left[ (\lambda -x_0)^*(\lambda -x_0) + s(\lambda)^2 \right]^{-1} (\lambda-x_0)^*\right \}\\
 =&\phi\left(  h^{-1}(\lambda-x_0)^*\right).
	\end{aligned}
\end{equation}
By \cite[Lemma 4.19]{Bordenave-Chafai-circular} (up to a sign change), the Brown measure of $x_0+g_{_{\alpha, \beta,0}}$ is equal to 
\begin{equation}
\label{eqn:partial-lambda-in-proof}
	 \begin{aligned}
 \mu_{x_0+g_{\alpha, \beta, 0}}&=\frac{1}{\pi} \lim_{\varepsilon\rightarrow 0^+}\frac{\partial}{\partial \overline{\lambda}}p_\lambda^{\bf t}(\varepsilon)=\frac{1}{\pi}\frac{\partial}{\partial \overline{\lambda}}\phi\left(  h^{-1} (\lambda-x_0)^* \right)
	 \end{aligned}
\end{equation}
in the distribution sense. Applying Lemma \ref{lem:tech-thm}, we have 
\begin{equation}
\label{eqn:4.5-in-proof}
\begin{aligned}
\frac{\partial}{\partial \overline{\lambda}}\phi\left( h^{-1} (\lambda-x_0)^*  \right)
=\phi\left[ h^{-1} \bigg( \unit-(\lambda-x_0)h^{-1}(\lambda-x_0)^* \bigg) \right]\\
\qquad\qquad -2s(\lambda)\frac{\partial s(\lambda)}{\partial\overline{\lambda}} \phi\left[ (\lambda-x_0)^* (h^{-1})^2 \right]. 
\end{aligned}
\end{equation}
Applying the identity $x(x^*x+\varepsilon)^{-1}x^*=(xx^*+\varepsilon)^{-1}xx^*$, we obtain
\begin{align}
  \phi\left[ h^{-1} \bigg( \unit-(\lambda-x_0)h^{-1}(\lambda-x_0)^* \bigg) \right]
     =s(\lambda)^2 \phi(h^{-1}k^{-1}). 
\end{align}
The function $s(\lambda)$ is the implicit function determined by \eqref{eqn:4.1-in-proof} when $\lambda\in\Xi_{\alpha, \beta}$, which can be rewritten as (by tracial property)
\[
      \phi(h^{-1})=\frac{\log\alpha-\log\beta}{\alpha-\beta}.
\]
Applying implicit differentiation $\frac{\partial}{\partial\overline{\lambda}}$, we have 
\[
    \phi\left[ h^{-1} \left( (\lambda-x_0) +2s(\lambda)\frac{\partial s(\lambda)}{\partial\overline{\lambda}} \right)  h^{-1}\right]=0
\]
which yields
\begin{equation}
  \label{eqn:4.7-in-proof}
	  -2s(\lambda)\frac{\partial s(\lambda)}{\partial\overline{\lambda}}
	 =\frac{\phi( (\lambda-x_0)(h^{-1})^2)}{\phi((h^{-1})^2)}.
\end{equation}
Since $\overline{\phi( (\lambda-x_0)(h^{-1})^2) }=\phi( (\lambda-x_0)^*(h^{-1})^2)$, we hence get the density formula \ref{eqn:density-Brown-tri-circular} by plugging \eqref{eqn:4.6-in-proof} and \eqref{eqn:4.7-in-proof} to \eqref{eqn:4.5-in-proof}. 

Note that if $\lambda\notin\overline{{\Xi_{\alpha,\beta}}}$, then $s(\lambda)=0$ in some neighborhood of $\lambda$, and 
\begin{equation}
\label{eqn:bounded-L-inverse-2}
\phi (\vert \lambda-x_0\vert^{-2}) \leq \frac{\log\alpha-\log\beta}{\alpha-\beta},
\end{equation}
which implies that the above limit \eqref{eqn:limit-partial-lambda-in-proof} is finite by an estimation using Cauchy-Schwartz inequality. 
 It follows that we can take the limit
\begin{equation}
\label{eqn:4.6-in-proof}
	 \lim_{\varepsilon\rightarrow 0^+}p_\lambda^{\bf t}(\varepsilon)=
	\phi \left\{ (\lambda-x_0)^* \left[ (\lambda -x_0)(\lambda -x_0)^*  \right]^{-1}\right \}
	=\lim_{\varepsilon\rightarrow 0} p_\lambda^{\bf t_0}(\varepsilon). 
\end{equation}
Hence, $\mu_{x_0+g_{_{\alpha, \beta,0}}}$ and $\mu_{x_0}$ coincide in the open set $\mathbb{C}\backslash\overline{\Xi_{\alpha, \beta}}$. 
By the inequality \eqref{eqn:bounded-L-inverse-2}
and \cite[Theorem 4.5]{Zhong2021Brown}, for any $\lambda\in \mathbb{C}\backslash \overline{{\Xi_{\alpha,\beta}}}$, then $\lambda$ is not in the support of $\mu_{x_0}$. 
 Consequently the density of $\mu_{x_0+g_{_{\alpha, \beta,0}}}$ is zero on $\mathbb{C}\backslash \overline{{\Xi_{\alpha,\beta}}}$ by \eqref{eqn:partial-lambda-in-proof} and \eqref{eqn:4.6-in-proof}. 
\end{proof}

\begin{remark}
Let $t=\frac{\alpha-\beta}{\log\alpha-\log\beta}$. The proof for Theorem \ref{thm:Brown-formula-gamma-0} shows that the partial derivatives of $\log S(x_0+g_{\alpha,\beta,0}, \lambda, 0)$ are the same as 
the partial derivatives of $\log S(x_0+c_t, \lambda,0)$ by comparing the proof of Theorem \ref{thm:BrownFormula-x0-ct-general} (see also \cite[Theorem 4.2]{Zhong2021Brown} for bounded case). Hence, the Brown measure of $x_0+g_{\alpha,\beta,0}$ is the same as the Brown measure of $x_0+c_t$. In particular, by \cite[Theorem 4.6]{Zhong2021Brown},
the measure $\mu_{x_0+g_{_{\alpha, \beta,0}}}$ is absolutely continuous with respect to the Lebesgue measure on $\mathbb{C}$. We refer to Corollary \ref{cor:no-atoms} and Theorem \ref{thm:BrownFormula-x0-ct-general} for $x_0$ to be unbounded operator. 
\end{remark}

\subsection{The push-forward map}
In this section, we show that the Brown measure $\mu_{x_0+g_{_{\alpha, \beta,\gamma}}}$ is the push-forward measure of $\mu_{x_0+g_{_{\alpha, \beta,0}}}$ under the map $\Phi_{\alpha,\beta, \gamma}$ defined by
\[
   \Phi_{\alpha, \beta,\gamma}(\lambda)=\lambda+\gamma \phi \left\{ (\lambda-x_0)^* \left[ (\lambda-x_0)(\lambda-x_0)^* + s(\lambda)^2\right]^{-1}\right\},
\]
where $s(\lambda)=\lim_{\varepsilon\rightarrow 0}\varepsilon_0(\lambda,\varepsilon)$ as in Lemma \ref{lemma:limit-s-lambda}. However, this map might be singular. We explain how the strategy in \cite[Section 5]{Zhong2021Brown} applies to our case. 

\begin{lemma}
The function $\Phi^{(\varepsilon)}_{\alpha, \beta, \gamma}$ converges uniformly to $\Phi_{\alpha, \beta, \gamma}$ in any compact subset of $\mathbb{C}$ as $\varepsilon$ tends to zero. The convergence is uniform over $\mathbb{C}$ if $x_0$ is bounded. 
\end{lemma}
\begin{proof}
For $\varepsilon_2>\varepsilon_1>0$, apply the same estimation as in \cite[Lemma 5.3]{Zhong2021Brown}, we have 
\[
 \vert p_\lambda^{\bf t_0}(\varepsilon_0(\lambda,\varepsilon_2))
   -p_\lambda^{\bf t_0}(\varepsilon_0(\lambda,\varepsilon_1))\vert \leq \frac{\log\alpha-\log\beta}{\alpha-\beta}\vert \varepsilon_0(\lambda,\varepsilon_2)-\varepsilon_0(\lambda,\varepsilon_1)\vert.
\]
Since $\Phi^{(\varepsilon)}_{\alpha, \beta, \gamma}$ can be written as
\[
  \Phi^{(\varepsilon)}_{\alpha, \beta, \gamma}(\lambda)=\lambda+\gamma p_\lambda^{\bf t_0}(\varepsilon_0(\lambda,\varepsilon))
\]
where ${\bf t_0}=(0,0,0)$.
The result then follows from convergence result of $\varepsilon_0(\lambda,\varepsilon)$ to $s(\lambda)$ showed in Lemma \ref{lemma:convergence-uniform-epsilon-0}. 
\end{proof}

\begin{theorem}
	 \label{thm:Brown-push-forward-property}
	The Brown measure $\mu_{x_0+g_{_{\alpha, \beta,\gamma}}}$ is the push-forward measure of $\mu_{x_0+g_{_{\alpha, \beta,0}}}$ under the map $\Phi_{\alpha, \beta, \gamma}$. Hence, we have the following commutative diagram.
	\begin{center}
		\begin{tikzpicture}
		\matrix (m) [matrix of math nodes,row sep=3em,column sep=4em,minimum width=2em]
		{
			\mu_{x_0+g_{_{\alpha, \beta,0}}}^{(\varepsilon)} & \mu_{x_0+g_{_{\alpha, \beta,\gamma}}}^{(\varepsilon)} \\
			\mu_{x_0+g_{_{\alpha, \beta,0}}} & \mu_{x_0+g_{_{\alpha, \beta,\gamma}}}\\};
		\path[-stealth]
		(m-1-1) edge node [left] {$\varepsilon\rightarrow 0$} (m-2-1)
		edge [double] node [below] {$\Phi_{\alpha, \beta,\gamma}^{(\varepsilon)}$} (m-1-2)
		(m-2-1.east|-m-2-2) edge [double] node [below] {{$\Phi_{\alpha, \beta,\gamma}$}}
		(m-2-2)
		(m-1-2) edge node [right] {$\varepsilon\rightarrow 0$} (m-2-2)
		(m-2-1);
		\end{tikzpicture}
	\end{center}
\end{theorem}

\begin{proof}
The regularized Brown measure $\mu_{x_0+g_{_{\alpha, \beta,\gamma}}}^{(\varepsilon)}$ is the push-forward measure of the regularized Brown measure $\mu_{x_0+g_{_{\alpha, \beta,0}}}^{(\varepsilon)}$ under the map $\Phi_{\alpha, \beta, \gamma}^{(\varepsilon)}$. For any $x\in\log^+(\mathcal{A})$, the regularized Brown measure $\mu_{x}^{(\varepsilon)}$ converges to the Brown measure $\mu_x$ weakly as $\varepsilon$ tends to zero. Since $\Phi_{\alpha, \beta, \gamma}^{(\varepsilon)}$ converges to $\Phi_{\alpha, \beta, \gamma}$ uniformly in any compact subset of $\mathbb{C}$, we then conclude the desired result. 
\end{proof}

\begin{example}[The Brown measure of triangular elliptic operator]
	\label{example:tri-elliptic-0}
	Given $\alpha,\beta>0$, set $t=\frac{\alpha-\beta}{\log\alpha-\log\beta}$. When $x_0=0$, the set ${\Xi_{\alpha,\beta}}$ is the circle centered at the origin with radius $\sqrt{t}$,
	\[
	{\Xi_{\alpha,\beta}}=\{ \lambda\in\mathbb{C}: \vert \lambda\vert<t\}
	\]
	and  $s(\lambda)^2+\vert \lambda\vert^2=t$ for $\lambda\in{\Xi_{\alpha,\beta}}$. Using notation in Theorem \ref{thm:Brown-formula-gamma-0},
	\[
	h(\lambda, s)=k(\lambda,s)=\vert \lambda\vert^2+s(\lambda)^2=t.
	\]
	Hence, the density of the Brown measure of $g_{_{\alpha, \beta,0}}$ is 
	\[
	d\mu_{g_{\alpha, \beta,0}}=\frac{1}{\pi}\frac{s^2}{t^2}+\frac{1}{\pi}\frac{\vert \lambda\vert^2}{t^2}=\frac{1}{\pi t} dxdy.
	\]
	In other words, it has the same Brown measure as $c_t$, the circular operator with variance $t$.  
	If $T$ is a quasi-nilpotent DT operator, and $c_\varepsilon$ is a circular operator with variance $\varepsilon$, let $\alpha={1+\varepsilon}, \beta={\varepsilon}$, then $T+c_\varepsilon$ has the same $*$-moments as $g_{_{\alpha, \beta,0}}$. Hence, the Brown measure $T+c_\varepsilon$ is the uniform measure on the circle $\Big\{ \lambda: \vert \lambda\vert\leq  \frac{1}{\sqrt{\log(1+\varepsilon^{-1})}} \Big\}$. This recovers a result of Aagaard-Haagerup \cite[Theorem 4.3]{HaagerupAagaard2004}.
	
	For $\gamma\in\mathbb{C}$ such that $\vert \gamma\vert\leq\sqrt{\alpha\beta}$, we now have 
	\[
	\Phi_{\alpha, \beta,\gamma}(\lambda)=\lambda+\gamma\cdot \frac{\overline{\lambda}}{\vert \lambda\vert^2+s^2},
	\]
	for $\vert \lambda\vert<t$. Hence, the Brown measure of $g_{_{\alpha, \beta,\gamma}}$ is supported in the ellipse with parametrization 
	\[
	\sqrt{t} e^{i\theta}+\frac{\vert \gamma\vert}{\sqrt{t}}e^{i(\psi-\theta)}
	\]
	where $\psi=\arg (\gamma)$. For $0\leq r<\sqrt{t}$, we have
	\[
	\Phi_{\alpha, \beta,\gamma}(re^{i\theta})=re^{i\theta}+\frac{\vert \gamma\vert r}{t}e^{i(\psi-\theta)}.
	\]
	Hence, for $\gamma=\gamma_1+i\gamma_2$, the Jacobian matrix of $\Phi_{\bf t}$ at $\lambda=\lambda_1+i\lambda_2$ is given by
	\[
	\text{Jacobian}(\Phi_{\alpha, \beta,\gamma})=\begin{bmatrix}
	1+\frac{\gamma_1}{t}  & \frac{\gamma_2}{t}\\
	\frac{\gamma_2}{t} & 1-\frac{\gamma_1}{t}
	\end{bmatrix}
	\]
	whose determinant is equal to $1-\frac{\vert \gamma\vert^2}{t}$. Hence, the Brown measure of $g_{_{\alpha, \beta,\gamma}}$ is the uniform measure in the ellipse by Theorem \ref{thm:Brown-push-forward-property}. This recovers the result obtained in \cite[Section 6]{BSS2018}
\end{example}

The interested reader is referred to \cite{HoHall2020Brown, Zhong2021Brown} for more examples of the push-forward maps and explicit Brown measure formulas. 

%%%%%%%%%%%%%%%%%%%%%%%%%%%%%%

\section{Some regularity results on the pushforward map}
The push-forward map $\Phi_{\alpha, \beta, \gamma}$ is the limit of the family of homeomorphisms $\Phi^{(\varepsilon)}_{\alpha,\beta, \gamma},$ which could be singular in general. In this section, we study its properties under some regularity assumptions. Recall that $\mathcal{N}$ is $*$-free from $\mathcal{M}$ with amalgamation over $\mathcal{B}$ in the operator-valued $W^*$-probability space $(\mathcal{A}, \mathbb{E}, \mathcal{B})$ constructed in Section \ref{section:tri-elliptic-review}. The unital subalgebra $\mathcal{N}$ is also $*$-free from $\mathcal{B}$ in $(\mathcal{A}, \phi).$

\begin{lemma}
\label{lemma:limit-S-y-0}
Let $y=g_{_{\alpha, \beta,0}} \in \mathcal{M}$ and $x_0 \in \log^+(\mathcal{N})$. The function $\lambda\mapsto S(x_0+y,\lambda,0)=\log\Delta(\vert x_0+y-\lambda\vert^2)$ is a $C^\infty$-function of $\lambda$ in the open set ${\Xi_{\alpha,\beta}}$. 
\end{lemma}
\begin{proof}
By the definition \eqref{eqn:defn-S-x0-y-epsilon} of the function $S$, we observe that
\[
   \frac{1}{2} \frac{d}{d\varepsilon} S(x_0+y,\lambda,\varepsilon)= q_\varepsilon^{\bf t}(\lambda),
\]
where ${\bf t}=(\alpha,\beta,0)$, and by \eqref{eqn:p-lambda-def}
\[
 q_\varepsilon^{\bf t}(\lambda)=\varepsilon\phi \left\{ \left[ (\lambda -x_0-y)^*(\lambda -x_0-y) + \varepsilon^2 \right]^{-1}\right \}.
\]
We now apply Lemma \ref{lemma:subordination-OA-1} to the case $\gamma=0$. In this case, 
the free cumulant $\kappa(y,fy)=0$ for any $f\in\mathcal{B}$, and hence $z=\Phi_{\alpha, \beta,0}^{(\varepsilon)}(\lambda)=\lambda$.
The subordination relation \eqref{eqn:sub-operator-1} yields
\[
   g_{11}(\lambda,\varepsilon_1,\varepsilon_2)=-iQ^{\bf t}_\varepsilon(\lambda),
\]
and by \eqref{eqn:defn-P-Q-derivatives} and \eqref{eqn:defn-g-ij-entries}
\[
{g}_{11}(\lambda,\varepsilon_1,\varepsilon_2) =-i \varepsilon_2{\mathbb{E}}\left\{  \left[ (\lambda-x_0)(\lambda-x_0)^* + \varepsilon_1\varepsilon_2\right]^{-1}\right\}
\]
and
\[
 Q^{\bf t}_\varepsilon(\lambda)=\varepsilon {\mathbb{E}} \left\{ \left[ (\lambda -x_0-y)(\lambda -x_0-y)^* + \varepsilon^2 \right]^{-1}\right \}.
\]
By taking trace, we showed in Section \ref{section:sub-oa-limit} (see \eqref{eqn:phi-epsilon-1} and \eqref{eqn:formula-g11-D}) that
\[
  i\phi(g_{11})
  =\phi(\varepsilon_2) D
  =\phi(\varepsilon_2) \phi \left\{  \left[ (\lambda-x_0)(\lambda-x_0)^* + \varepsilon_1\varepsilon_2\right]^{-1}\right\},
\]
and $\varepsilon_0^2=\varepsilon_1\varepsilon_2$, and 
\[
\phi(\varepsilon_2) = \varepsilon_0\frac{(e^{\sigma}-1)}{\sigma} e^{-\sigma/2},
\]
where $\sigma=\sigma(\lambda,\varepsilon)=(\alpha-\beta)D$. 
We now put
\[
  f(\sigma)=\frac{e^{\sigma/2}+e^{-\sigma/2}}{\sigma}.
\]
Then we can rewrite $ i\phi(g_{11})$ as
\[
   i\phi(g_{11})=\left( \varepsilon_0 \phi \left\{  \left[ (\lambda-x_0)(\lambda-x_0)^* + \varepsilon_0^2\right]^{-1}\right\} \right) \cdot f(\sigma(\lambda,\varepsilon)).
\]
Recall that $\phi(Q_\varepsilon^{\bf t}(\lambda))=q_\varepsilon^{\bf t}(\lambda)$. 
Denote ${\bf t_0}=(0,0,0)$. Following definition \eqref{eqn:p-lambda-def}, the identity $i\phi(g_{11})=\phi(Q_\varepsilon^{\bf t}(\lambda))$ can be rewritten as
\[
    q_{\varepsilon_0(\lambda,\varepsilon)}^{\bf t_0}(\lambda)\cdot f(\sigma(\lambda,\varepsilon))=q_\varepsilon^{\bf t}(\lambda).
\]
Therefore, by taking the integration of the above identity over some interval $[\varepsilon, \delta]$, we obtain
\begin{equation}
  S(x_0+y,\lambda,\delta)- S(x_0+y,\lambda,\varepsilon)=2\int^\delta_\varepsilon f(\sigma(\lambda,u)) q_{\varepsilon_0(\lambda,u)}^{\bf t_0}(\lambda)\, du.
\end{equation}
By Proposition \ref{prop:epsilon0-epsilon-analyticity}, we know that $\varepsilon_0(\lambda,\varepsilon)$ and $D(\lambda, \varepsilon)$ are $C^\infty$-functions of $(\lambda,\varepsilon)$. Hence, $(\lambda,\varepsilon)\rightarrow  f(\sigma(\lambda,\varepsilon))$ is a $C^\infty$-function of $(\lambda,\varepsilon)$ over ${\Xi_{\alpha,\beta}}\times [0,\infty)$. Moreover, by Lemma \ref{lemma:limit-s-lambda}, recall that, for $\lambda\in{\Xi_{\alpha,\beta}}$, we have  $\lim_{\varepsilon\rightarrow 0^+}\varepsilon_0=s(\lambda)>0$ and $s(\lambda)$ is a $C^\infty$-function of $\lambda.$
Therefore, 
\[
  \begin{aligned}
   S(x_0+y,\lambda,0)&=
   S(x_0+y,\lambda,\delta)-\lim_{\varepsilon\rightarrow 0^+} 2\int_\varepsilon^\delta f(\sigma(\lambda,u)) q_{\varepsilon_0(\lambda,u)}^{\bf t_0}(\lambda)\, du\\
   &=S(x_0+y,\lambda,\delta)-2\int_0^\delta f(\sigma(\lambda,u)) q_{\varepsilon_0(\lambda,u)}^{\bf t_0}(\lambda)\, du.
  \end{aligned}
\]
By the above discussion, the right hand side of the above equation is a $C^\infty$-function of $\lambda\in{\Xi_{\alpha,\beta}}$. We remark that it is clear that $S(x_0 + y, \lambda, \delta)$ is a $C^\infty$-function of $\lambda$ due to the regularity of 
$S(x_0 + y, \lambda, \varepsilon)$ for $\varepsilon>0$. Hence, $S(x_0+y,\lambda,0)$ is a $C^\infty$-function of $\lambda$ in ${\Xi_{\alpha,\beta}}$. 
\end{proof}

\begin{proposition}
	 \label{prop:partial-lambda-pushforward}
	 Given ${\bf t(\gamma)}=(\alpha,\beta,\gamma)$ and set  ${\bf t}={\bf t}(0)=(\alpha,\beta, 0)$ and ${\bf t_0}=(0,0,0)$, let $x_0\in\log^+(\mathcal{N})$ be an operator that is $*$-free from $\{g_{_{\alpha, \beta,\gamma}}, g_{_{\alpha, \beta,0}}\}$ with amalgamation over $\mathcal{B}$. 
If the map $\Phi_{\alpha, \beta,\gamma}$ is non-singular at some $\lambda\in{\Xi_{\alpha,\beta}}$, then the map
$(z,\varepsilon)\mapsto S(x_0+y,z,\varepsilon)$
 has a $C^\infty$-extension in some neighborhood of $(\Phi_{\alpha, \beta,\gamma}(\lambda),0)$. Hence, 
\begin{equation}
  \label{eqn:limit-p-z-gamma}
\lim_{\varepsilon\rightarrow 0^+}  p_{\Phi_{\alpha, \beta,\gamma}(\lambda)}^{\bf t(\gamma)}(\varepsilon)
    =\lim_{\varepsilon\rightarrow 0^+}  p_\lambda^{\bf t}(\varepsilon)
    =\lim_{\varepsilon\rightarrow 0^+}  p_\lambda^{{\bf t_0}}(\varepsilon_0(\lambda,\varepsilon)).
\end{equation}

Morever, if the map $\Phi_{\alpha, \beta,\gamma}$ is non-singular at any $\lambda\in{\Xi_{\alpha,\beta}}$, then it is also one-to-one. 
\end{proposition}

\begin{proof}
Again, by the definition \eqref{eqn:defn-S-x0-y-epsilon} of the function $S(x_0+y,z,\varepsilon)$, we observe that
 \begin{equation}
 	\label{eqn:derivative-S-epsilon}
 	 \frac{1}{2} \frac{d}{d\varepsilon} S(x_0+g_{_{\alpha, \beta,\gamma}},z,\varepsilon)= q_\varepsilon^{\bf t(\gamma)}(z).
 \end{equation}
Recall that the subordination relation in \eqref{eqn:sub-operator-1} reads
\[
  g_{11}(\lambda,\varepsilon_1,\varepsilon_2)=-iQ^{\bf t(\gamma)}_\varepsilon(z)
\]
where
\[
z=\Phi_{\alpha, \beta, \gamma}^{(\varepsilon)}(\lambda)
=\lambda+\gamma p_\lambda^{\bf t_0}(\varepsilon_0).
\] 
If $\gamma=0$, then $\Phi_{\alpha, \beta, 0}^{(\varepsilon)}(\lambda)=\lambda$. 
By choosing $\gamma=0$ and an arbitrary eligible $\gamma$ respectively, 
we have
\[
  Q^{\bf t(\gamma)}_\varepsilon(z)=Q^{\bf t}_\varepsilon(\lambda)=ig_{11}(\lambda,\varepsilon_1,\varepsilon_2),
\]
 By taking the trace, we have $q_\varepsilon^{\bf t(\gamma)}(z)=q^{\bf t}_\varepsilon(\lambda)$. Hence,
for $z=\Phi_{\alpha, \beta, \gamma}^{(\varepsilon)}(\lambda)$, by integrating \eqref{eqn:derivative-S-epsilon}, we obtain 
\begin{equation}
\begin{split}
	 \label{eqn:S-y-gamma-identity}
	 S(x_0+g_{_{\alpha, \beta,\gamma}}, z,\varepsilon)&-S(x_0+g_{_{\alpha, \beta,\gamma}}, z,1)\\
	 &=S(x_0+g_{_{\alpha, \beta,0}}, \lambda,\varepsilon)-S(x_0+g_{_{\alpha, \beta,0}}, \lambda,1).
\end{split}
\end{equation}

Lemma \ref{lemma:limit-S-y-0} shows that the right hand side of the above equation is a $C^\infty$-function of $\lambda$ when $\varepsilon$ goes to zero. Recall that, for $\lambda\in{\Xi_{\alpha,\beta}}$, $\lim_{\varepsilon\rightarrow 0^+}\varepsilon_0(\lambda,\varepsilon)=s(\lambda)$. We then have
\[
 \lim_{\varepsilon\rightarrow 0^+}\Phi_{\alpha, \beta, \gamma}^{(\varepsilon)}(\lambda)= \lim_{\varepsilon\rightarrow 0^+} \lambda+\gamma \cdot p_\lambda^{\bf t_0}(\varepsilon_0)
  =\lambda+\gamma  \cdot p_\lambda^{\bf t_0}(s(\lambda))=\Phi_{\alpha, \beta,\gamma}(\lambda).
\]
Thus, the assumption that $\Phi_{\alpha, \beta,\gamma}$ is non-singular at $\lambda$ implies that the map $(\lambda, \varepsilon)\mapsto (z,\varepsilon)$ is non-singular at $(\lambda,0)$. This yields that the analyticity of right hand side of \eqref{eqn:S-y-gamma-identity} implies that the left hand side of \eqref{eqn:S-y-gamma-identity} has a $C^\infty$-extension in some neighborhood of $(\Phi_{\alpha, \beta,\gamma}(\lambda), 0)$. 

Recall that $g_{21}=g_{21}(\lambda,\varepsilon_1,\varepsilon_2)$ is a constant by Lemma \ref{lem:compute-tilde-g}. 
Again by choosing $\gamma=0$ and an arbitrary $\gamma$ respectively in the subordination relation $g_{21}=P_z^{\bf t(\gamma)}(\varepsilon)=p_z^{\bf t(\gamma)}(\varepsilon)$, we have 
\[
   p_z^{\bf t(\gamma)}(\varepsilon)=p_\lambda^{\bf t}(\varepsilon)=p_\lambda^{\bf t_0}(\varepsilon_0),
\]
where $z=\Phi_{\alpha, \beta, \gamma}^{(\varepsilon)}(\lambda)$ and $\varepsilon_0=\varepsilon_0(\lambda,\varepsilon)$ is given by \eqref{eqn:system-D-epsilon}.
Recall that 
\[
p_z^{\bf t(\gamma)}(\varepsilon)=\frac{d}{d\varepsilon} S(x_0+g_{_{\alpha, \beta,\gamma}}, z,\varepsilon).
\]
Recall that, by definitions \eqref{eqn:p-lambda-def} and \eqref{eqn:defn-S-x0-y-epsilon}, 
for $y=g_{\alpha, \beta,\gamma}$, we have
\begin{align*}
  p_z^{{\bf t(\gamma)}}(\varepsilon) &= \phi \left\{ (z-x_0-y)^* \left[ (z -x_0-y)(z -x_0-y)^* + \varepsilon^2 \right]^{-1}\right \}\\
    &=\frac{\partial S(x_0+y,z,\varepsilon)}{\partial{z}}.
\end{align*}
Hence, the regularity of $S(x_0+g_{_{\alpha, \beta,\gamma}},z,\varepsilon)$ in the neighborhood of $(\Phi_{\alpha, \beta,\gamma}(\lambda), 0)$ implies that
\begin{equation}
  \label{eqn:partial-lambda-regularity-limit}
      \lim_{\varepsilon\rightarrow 0^+}  p_{\Phi_{\alpha, \beta,\gamma}(\lambda)}^{\bf t(\gamma)}(\varepsilon)
=\lim_{\varepsilon\rightarrow 0^+}  p_\lambda^{\bf t}(\varepsilon)
   =\lim_{\varepsilon\rightarrow 0^+}  p_\lambda^{{\bf t_0}}(\varepsilon_0(\lambda,\varepsilon)).
\end{equation}

Assume now that the map $\Phi_{\alpha, \beta,\gamma}$ is non-singular at any $\lambda\in{\Xi_{\alpha,\beta}}$, then \eqref{eqn:limit-p-z-gamma} holds at any $\lambda\in{\Xi_{\alpha,\beta}}$. Hence, if $\lambda_1, \lambda_2\in {\Xi_{\alpha,\beta}}$ and  $\Phi_{_{\alpha\beta,\gamma}}(\lambda_1)=\Phi_{_{\alpha\beta,\gamma}}(\lambda_2)=z$, we have 
\[
     z=\lambda_1+\gamma p_{\lambda_1}^{\bf t_0}(s(\lambda_1))=
     \lambda_2+\gamma p_{\lambda_2}^{\bf t_0}(s(\lambda_2)).
\]
Then \eqref{eqn:partial-lambda-regularity-limit} shows
\[
  \lim_{\varepsilon\rightarrow 0^+} p_{\lambda_1}^{\bf t_0}(\varepsilon_0 (\lambda_1, \varepsilon))
    =\lim_{\varepsilon\rightarrow 0^+} p_{\lambda_2}^{\bf t_0}(\varepsilon_0 (\lambda_2, \varepsilon))
    =p_{z}^{\bf t(\gamma)}(0).
\]
That is $p_{\lambda_1}^{\bf t_0}(s(\lambda_1))=p_{\lambda_2}^{\bf t_0}(s(\lambda_2))$. Hence, we deduce that $\lambda_1=\lambda_2$ and the map $\Phi_{_{\alpha\beta,\gamma}}$ is one-to-one in ${\Xi_{\alpha,\beta}}$. 
\end{proof}

%%%%%%%%%%%%%%%%%%%%%%%%%%%%%

\section{The sum of an elliptic operator and an unbounded operator}
We proceed to study the case when $\alpha=\beta$, in which case $g_{\alpha, \beta,\gamma}$ is a twisted elliptic operator.
Consider now an operator $x_0\in \log^+({\mathcal A})$ and a twisted elliptic operator $g_{t,\gamma}$ that is $*$-free from $x_0$. By using operator-valued subordination functions for unbounded operators from Lemma \ref{lemma:subordination-OA-1}, we are able to study the Brown measure of $x_0+g_{t,\gamma}$. Since $x_0$ could be unbounded, we extend main results in \cite{Zhong2021Brown}.

We denote 
\begin{equation}
\label{eqn:X0-Ct}
X=
\begin{bmatrix}
0 & x_0\\
x_0^* & 0
\end{bmatrix}, 
\qquad 
Y=
\begin{bmatrix}
0 & g_{t,\gamma}\\
g_{t,\gamma}^* & 0
\end{bmatrix}.
\end{equation}
Then $X$ is affiliated with ${M}_2(\mathcal{A})$, and $\{X, Y\}$ are free  with amalgamation over $M_2(\mathbb C)$ in the noncommutative probability space $(M_2(\mathcal{A}), M_2(\phi), M_2(\mathbb{C}))$.
By Theorem \ref{thm:subordination-unbounded}, there exist two analytic self-maps $\Omega_1, \Omega_2$
of upper half-plane $\mathbb{H}^+({M}_2(\mathbb{C}))$ of $M_2(\mathbb C)$
such that
\begin{equation}\label{eqn:subord-X0-Ct}
(\Omega_1(b)+\Omega_2(b)-b)^{-1}=G_X(\Omega_1(b))=G_Y(\Omega_2(b))=G_{X+Y}(b),
\end{equation}
for all $b\in M_2(\mathbb C)$ with $\Im b>0$. In other words, subordination relation \eqref{eqn:subord-operator} for $X+Y$ is still available, where $X$ could be unbounded.  

We introduce a new approach via regularized Brown measure to study the distinguished measure $\mu_{x_0+c_t}$. This allows us to prove several regularity results for this Brown measure. In particular, we show that $\mu_{x_0+c_t}$ is absolutely continuous with respect to the Lebesgue measure. 
\subsection{The subordination functions}
We first calculate a formula for $\Omega_1(\Theta(\lambda,\varepsilon))$ and study its entries following the strategy in \cite[Section 3]{Zhong2021Brown}. The next result is \cite[Proposition 3.5]{Zhong2021Brown}, and we provide a direct proof for convenience. 
\begin{lemma}
	 \label{lemma:sub-additive-semicircular}
	Let $\mu_1$ be a symmetric probability measure, not necessarily compactly supported, and let $\mu_2$ be the semicircular distribution with variance $t$. Denote $\mu=\mu_1\boxplus\mu_2$. Let $\omega_1, \omega_2$ be subordination function such that 
	\[
	G_\mu(z)=G_{\mu_1}(\omega_1(z))=G_{\mu_2}(\omega_2(z)).
	\]
	Set $w(\varepsilon)=\Im \omega_1(i\varepsilon)$. Then, $w=w(\varepsilon)$ is the unique solution in $(0,\infty)$ of the following equation
	\begin{equation}
	\label{eqn:fixed-pt-scalar}
	w=\varepsilon+tw\int_{\mathbb{R}}\frac{1}{u^2+w^2}d\mu_1(u).
	\end{equation}
\end{lemma}
\begin{proof}
	Note that the $R$-transform of $\mu_2$ is $R_{\mu_2}(z)=tz$. Since $R_{\mu}=R_{\mu_1}+R_{\mu_2}$, we have 
	\[
	G_\mu^{\langle -1 \rangle}(z)=G^{\langle -1 \rangle}_{\mu_1}(z)+R_{\mu_2}(z)=G^{-1}_{\mu_1}(z)+tz. 
	\]
	By replacing $z$ with $G_\mu(z)$, we obtain $\omega_1(z)=z-tG_\mu(z)$ which holds for all $z\in\mathbb{C}^+$ by analytic continuation. In particular, for $z=i\varepsilon$, we have 
	\begin{equation}
	\label{eqn:8.10-in-proof}
	\omega_1(i\varepsilon)=i\varepsilon-tG_\mu(i\varepsilon)=i\varepsilon-tG_{\mu_1}(\omega_1(i\varepsilon)).
	\end{equation}
	It is clear that $\omega_1(i\varepsilon)$ is a pure imaginary number and hence $\omega_1(i\varepsilon)=iw(\varepsilon)$. Note that, 
	\[
	G_{\mu_1}(iw)=\int_{\mathbb{R}}\frac{1}{iw-u}d\mu_1(u)=-iw\int_{\mathbb{R}}\frac{1}{u^2+w^2}d\mu_1(u)
	\]
	by the symmetry of $\mu_1$. Taking the imaginary part of \eqref{eqn:8.10-in-proof}, we see that $w(\varepsilon)$ satisfy \eqref{eqn:fixed-pt-scalar}. 
	
	We next show that \eqref{eqn:fixed-pt-scalar} has a unique solution in $(0,\infty)$. Denote 
	\begin{equation}
	\label{eqn:function-k-fixed-pt-scalar}
	k(w,\varepsilon)=\frac{w}{w-\varepsilon}\left( \int_{\mathbb{R}}\frac{1}{u^2+w^2}d_{\mu_1}(u) \right).
	\end{equation}
	Then \eqref{eqn:fixed-pt-scalar} is equivalent to $k(w,\varepsilon)=1/t$. Observe that $w\mapsto k(w,\varepsilon)$ is a strictly decreasing function in $(0,\infty)$. Hence, \eqref{eqn:fixed-pt-scalar} has a unique solution in $(0,\infty)$
\end{proof}

\begin{definition}
	\label{eqn:def-mu1-scalar-subordination}
	For any probability measure $\mu$ on $\mathbb{R}$, we denote by $\widetilde{\mu}$ the symmetrization of $\mu$ defined by
	\[
	   \widetilde{\mu}(B)=\frac{1}{2}(\mu(B)+\mu(-B)),
	\]
	for any Borel measurable set $B$ on $\mathbb{R}$. 
	Given $\lambda\in\mathbb{C}$, let $\mu_1=\widetilde{\mu}_{\vert \lambda-x_0\vert}$ and let $\mu_2$ be the semicircular distribution with variance $t$. Let $\omega_1$ be the subordination function (depending on $\lambda$) such that 
	\[
	G_{\mu_1\boxplus\mu_2}(z)=G_{\mu_1}(\omega_1(z)).
	\]
	We denote $w(\varepsilon;\lambda,t)=\Im\omega_1(i\varepsilon)$. 
\end{definition}

We choose
\[
b=\Theta(z,\varepsilon)=
\begin{bmatrix}
i\varepsilon & z\\
\overline{z} & i\varepsilon
\end{bmatrix}
\]
where $\varepsilon>0$ and $z\in \mathbb{C}$.  The following result extends \cite[Theorem 3.8]{Zhong2021Brown} to unbounded operator $x_0$.

\begin{theorem}
	\label{thm:subordination-elliptic-case}
	Let $x_0\in\log^+(\mathcal{A})$ be an operator that is $*$-free from the elliptic operator $g_{t,\gamma}$. 
	For any $\varepsilon>0$ and $z\in\mathbb{C}$, we set
	\begin{equation}
	\label{eqn-z-lambda-subordination}
	\lambda
	=z-\gamma\cdot \phi\bigg( (z-x_0-g_{t,\gamma})^*\big( (z-x_0-g_{t,\gamma})(z-x_0-g_{t,\gamma})^*+\varepsilon^2 \big)^{-1} \bigg).
	\end{equation}
	Then $\Omega_1(\Theta(z,\varepsilon))=\Theta(\lambda, w(\varepsilon;\lambda,t))$.
	That is,
	\begin{align}
	\Omega_1\left( 
	\begin{bmatrix}
	i\varepsilon & z\\
	\overline{z} & i\varepsilon
	\end{bmatrix}\right)
	=\begin{bmatrix}
	i w(\varepsilon;\lambda,t)  & \lambda\\
	\overline{\lambda} & i w(\varepsilon; \lambda,t)
	\end{bmatrix}.
	\end{align}
	The subordination relation $G_{X+Y}(\Theta(z,\varepsilon))=G_X(\Omega_1(\Theta(z,\varepsilon)))$ is expressed as
	\begin{align}\label{eqn:sub-operator-cor3.6} 
	\mathbb{E}\left( \begin{bmatrix}
	i\varepsilon & z-(x_0+g_{t,\gamma})\\
	\overline{z}-(x_0+g_{t,\gamma})^* & i\varepsilon
	\end{bmatrix}^{-1} \right)
	=\mathbb{E}\left( \begin{bmatrix}
	iw(\varepsilon;\lambda,t) & \lambda-x_0\\
	\overline{\lambda}-x_0^* & iw(\varepsilon;\lambda,t)
	\end{bmatrix}^{-1} \right),
	\end{align}
	which is also equivalent to
	\begin{equation}
	\label{eqn:subordination-Cauchy-entries}
	\begin{aligned}
	&\varepsilon\phi\bigg( \big(({z}-x_0-g_{t,\gamma})({z}-x_0-g_{t,\gamma})^*+\varepsilon^2\big)^{-1} \bigg)\\
	&\qquad\qquad=w(\varepsilon;\lambda,t)\phi\bigg( \big( (\lambda-x_0)(\lambda-x_0)^*+w(\varepsilon;\lambda,t)^2 \big)^{-1} \bigg),\\
	&\phi\bigg( ({z}-x_0-g_{t,\gamma})^*\big( ({z}-x_0-g_{t,\gamma})({z}-x_0-g_{t,\gamma})^*+\varepsilon^2\big)^{-1} \bigg)\\
	&\qquad\qquad=\phi\bigg( (\lambda-x_0)^*\big( (\lambda-x_0)(\lambda-x_0)^*+w(\varepsilon;\lambda,t)^2\big)^{-1} \bigg).
	\end{aligned}
	\end{equation}
\end{theorem}

\begin{proof}
Apply Lemma \ref{lemma:subordination-OA-1} to the case $\alpha=\beta=t$. In this case, 
\[
 Q_\varepsilon^{\bf t}(z)=\widetilde{Q}^{\bf t}_\varepsilon (z)
 \]
 and hence $\varepsilon_1=\varepsilon_2 \in (\varepsilon,\infty)$. Moreover, by \eqref{eqn:sub-operator-1} and the definition \eqref{eqn:defn-g-ij-entries}, then $\varepsilon_1$ can be expressed as
 \[
  \varepsilon_1=\varepsilon+\varepsilon_1
    \phi\bigg( \big( (\lambda-x_0)(\lambda-x_0)^*+\varepsilon_1^2 \big)^{-1} \bigg).
 \]
 This shows that $\varepsilon_1=\varepsilon_2=w(\varepsilon;\lambda,t)$ by \eqref{eqn:fixed-pt-scalar} and 
 Definition \ref{eqn:def-mu1-scalar-subordination}.
\end{proof}

Recall that the open set $\Xi_t$ is defined in \eqref{defn:Xi-t} as
\[
  \Xi_t=\left\{ \lambda\in\mathbb{C} : \phi \left[\big( (x_0-\lambda)^*(x_0-\lambda)\big)^{-1}\right]>\frac{1}{t}  \right\}.
\]
\begin{definition}
	For $\lambda\in\Xi_t$, let $w(0;\lambda,t)$ be the unique solution $w\in(0,\infty)$ to the following equation
	\begin{equation}
	\label{eqn:w-epsiton-0-identity}
	\phi\left[\big( (x_0-\lambda)^*(x_0-\lambda) +w(0; \lambda,t)^2 \big)^{-1}\right]=\frac{1}{t}.
	\end{equation}
	For $\lambda\in \mathbb{C}\backslash\Xi_t$, set $w(0;\lambda,t)=0$. 
\end{definition}

We define the function $\Phi^{(\varepsilon)}_{t,\gamma}$ on $\mathbb{C}$ by
\begin{equation}
\label{defn:Phi-t-gamma-epsitlon}
\Phi^{(\varepsilon)}_{t,\gamma} (\lambda)= \lambda+\gamma\cdot p_\lambda^{(0)}( w(\varepsilon;\lambda,t)  ),
\qquad \lambda\in\mathbb{C}
\end{equation}
where 
\[
p_\lambda^{(0)}( w(\varepsilon;\lambda,t)  ) =
\phi\bigg[ (\lambda-x_0)^*\big( (\lambda-x_0)(\lambda-x_0)^*+w(\varepsilon; \lambda,t)^2 \big)^{-1}\bigg].
\]
This can also be rewritten as 
\[
\Phi^{(\varepsilon)}_{t,\gamma}(\lambda)=\lambda+\gamma\cdot \frac{\partial S}{\partial \lambda}(x_0, \lambda,w(\varepsilon;\lambda,t)).
\]
The map $\Phi_{t,\gamma}$ is defined as
\[
\Phi_{t,\gamma} (\lambda)= \lambda+\gamma\cdot p_\lambda^{(0)}( w(0;\lambda,t)  ),
\qquad \lambda\in\mathbb{C}.
\]

The following result is a special case of Lemma \ref{lemma:Phi-varepsilon-alphabeta-injective} by letting $\alpha=\beta=t$.
\begin{corollary}
	\label{cor:regularization-Phi-one2one-map}
	The map $\Phi^{(\varepsilon)}_{t,\gamma}$ is a homeomorphism of the complex plane for any $\varepsilon >0.$
		Its inverse map is 
	   \begin{equation}
	   	 J^{(\varepsilon)}_{t,\gamma}(z)=z-\gamma\cdot p_z^{(t,\gamma)}(\varepsilon), 
	   \end{equation}
	  where 
	  \[
	    p_z^{(t,\gamma)}(\varepsilon)=\phi\bigg( ({z}-x_0-g_{t,\gamma})^*\big( ({z}-x_0-g_{t,\gamma})({z}-x_0-g_{t,\gamma})^*+\varepsilon^2 \big)^{-1} \bigg).
	  \]
\end{corollary}

\begin{lemma}
	 \label{lemma:regularity-sub-lambda}
	For any $\varepsilon> 0$, the function $\lambda\mapsto w(\varepsilon;\lambda,t)$ is a $C^\infty$ function. The function $\lambda\mapsto w(0;\lambda,t)$ is a continuous function of $\lambda$ on $\mathbb{C}$ and is a $C^\infty$ function in the open set $\Xi_t$. 
	
   Moreover, the function $w(\varepsilon;\lambda,t)$ converges uniformly to $w(0;\lambda,t)$ in any compact subset of $\mathbb{C}$ as $\varepsilon$ tends to zero. 
\end{lemma}
\begin{proof}
	By Definition \ref{eqn:def-mu1-scalar-subordination}, for any $\varepsilon>0$, the function $w(\varepsilon;\lambda,t)$ is the imaginary part of the subordination function parameterized by $\lambda$. In particular, $w(\varepsilon;\lambda,t)=\omega_1(i\varepsilon)>\varepsilon$, where $\omega_1$ is the subordination as in Definition \ref{eqn:def-mu1-scalar-subordination}. By Lemma \ref{lemma:sub-additive-semicircular}, the function $w(\varepsilon;\lambda,t)$ is the unique $w$ such that
	\begin{equation}
	\label{eqn:fixed-pt-scalar-lambda}
	  \begin{aligned}
	  w&=\varepsilon+tw\int_{\mathbb{R}}\frac{1}{u^2+w^2}d\,\widetilde{\mu}_{\vert \lambda-x_0\vert}(u)\\
	  &=\varepsilon+tw\phi[((\lambda-x_0)^*(\lambda-x_0)+w^2 )^{-1}].
	  \end{aligned}
	\end{equation}
	Following \eqref{eqn:function-k-fixed-pt-scalar}, we rewrite it as
	\begin{equation}
	   \label{eqn:7.14-in-proof}
		 \begin{aligned}
		 k(w,\varepsilon)&=\frac{w}{w-\varepsilon}\left( \int_{\mathbb{R}}\frac{1}{u^2+w^2}d\,\widetilde{\mu}_{\vert \lambda-x_0\vert}(u) \right)\\
		 &=\frac{w}{w-\varepsilon}\phi[( (\lambda-x_0)^*(\lambda-x_0)+w^2 )^{-1}]
		 =\frac{1}{t}.
		 \end{aligned}
	\end{equation}
	Note that $\frac{\partial}{\partial w}k(w,\varepsilon)<0$ for $w>\varepsilon$ and $\lambda\mapsto k(w,\varepsilon)$ is a smooth function. By implicit function theorem, $\lambda\mapsto w(\varepsilon;\lambda,t)$ is a $C^\infty$ function. Similarly, for $\lambda\in\Xi_t$, the function $w(0;\lambda,t)$ is determined by \eqref{eqn:w-epsiton-0-identity}, which yields that $\lambda\mapsto w(0;\lambda,t)$ is a $C^\infty$ function in the open set $\Xi_t$ by applying the implicit function theorem. 
	
   Recall that $w(\varepsilon;\lambda,t)$ is the unique solution of 
	$k(w,\varepsilon)=1/t$ as in \eqref{eqn:7.14-in-proof}. We observe that
	\[
	k(w(\varepsilon_1;\lambda,t),\varepsilon_1)<k(w(\varepsilon_1;\lambda,t),\varepsilon_2)
	\]
	if $0<\varepsilon_1<\varepsilon_2<w(\varepsilon_1;\lambda,t)$. Hence, $w(\varepsilon_1;\lambda,t)<w(\varepsilon_2;\lambda,t)$ since $\frac{\partial}{\partial w}k(w,\varepsilon)<0$. By Dini's theorem, the convergence of $w(\varepsilon;\lambda,t)$ to $w(0;\lambda,t)$ is uniform in any compact subset of $\mathbb{C}$. 
\end{proof}

\subsection{The Fuglede-Kadison determinant formula}
By using Theorem \ref{thm:subordination-elliptic-case}, we can apply exactly the same methods as in \cite[Lemma 3.11, Theorem 3.12]{Zhong2021Brown} to obtain a formula for $\Delta\big( (x_0+g_{t,\gamma})^*(x_0+g_{t,\gamma})+\varepsilon^2 \big)$. We only state a result for circular operator below (See \cite[Section 5.2]{BercoviciZhong2022} for an alternative proof). 

\begin{theorem}\cite[Theorem 3.12]{Zhong2021Brown}
	\label{thm:main-FK-det-ct-0}
	For $\lambda\in\mathbb{C}$, we have the following Fuglede-Kadison determinant formulas. 	
	\begin{enumerate}[(1)]
		\item If $\lambda\in\Xi_t$, then 
		\begin{equation}
		\label{eqn:main-FK-det-0-V2}
		\begin{aligned}
		\Delta \big(x_0+c_t-\lambda \big)^2
		&={\Delta\big( (x_0-\lambda)^*(x_0-\lambda)+w(0; \lambda,t)^2 \big)}\\ &\qquad\qquad\times{\exp\bigg(-\frac{(w(0;\lambda,t))^2}{t}\bigg)}.
		\end{aligned}
		\end{equation}	
		\item If $\lambda\notin\Xi_t$, then 
		\begin{equation}
		  \label{eqn:main-FK-det-outside}
		 \Delta(x_0+c_t-\lambda)=\Delta(x_0-\lambda).
		\end{equation}
	\end{enumerate}
\end{theorem}

The following result is adapted from \cite[Theorem 4.6]{Zhong2021Brown}. We include a proof for the reader's convenience.
\begin{corollary}
	 \label{cor:no-atoms}
For any $\lambda\in\mathbb{C}$, we have 
\[
    \log \Delta(x_0+c_t-\lambda)>-\infty.
\]
The Brown measure $\mu_{x_0+c_t}$ has no atom and the support of $\mu_{x_0+c_t}$ is contained in the closure $\overline{\Xi_t}$. 
\end{corollary}
\begin{proof}
If $\lambda\in\Xi_t$, since $w(0;\lambda,t)>0$, it follows from \eqref{eqn:main-FK-det-0-V2} that $\log \Delta(x_0+c_t-\lambda)>-\infty$. 
If $\lambda\notin\Xi_t$, then $\phi(\vert x_0-\lambda \vert^{-2})\leq 1/t$, and hence 
\begin{align*}
 2\int_0^1 \log t d\mu_{\vert x_0-\lambda\vert}(t)
   &>-\int_0^1\frac{1}{t^2}d\mu_{\vert x_0-\lambda\vert}(t)
   >-\phi(\vert x_0-\lambda \vert^{-2})\geq-1/t. 
\end{align*}
Therefore, $\lambda\notin \Xi_t$, 
\[
  \log \Delta(x_0+c_t-\lambda)=\log\Delta(x_0-\lambda)>-\infty
\]
By \cite[Proposition 2.14, 2.16]{HaagerupSchultz2007}, we deduce that $\mu_{x_0+c_t}(\{\lambda\})=\mu_{x_0+c_t-\lambda}(\{0\})=0$ for any $\lambda\in\mathbb{C}$.

The Brown measure of $\mu_{(x_0-\lambda)^{-1}}$ is the pushforward measure of 
$\mu_{x_0}$ under the map $z\mapsto (z-\lambda)^{-1}$. 
 By Weil's inequality for operators in $L^p(\mathcal{A})$ (see \cite[Theorem 2.19]{HaagerupSchultz2007}), we have 
\[
  \int_\mathbb{C}\frac{1}{\vert z-\lambda\vert^2}d\mu_{x_0}(z)
    =\int_\mathbb{C}{\vert z\vert^2}d\mu_{(x_0-\lambda)^{-1}}(z)
    \leq \Vert (x_0-\lambda)^{-1}\Vert^2=\phi(\vert x_0-\lambda \vert^{-2} ).
\]
Consequently, if $\lambda\notin \overline{\Xi_t}$, then 
\[
  \int_\mathbb{C}\frac{1}{\vert z-\lambda\vert^2}d\mu_{x_0}(z)\leq \frac{1}{t}
\]
in some neighborhood of $\lambda$. This implies that $\lambda$ is not in the support of $\mu_{x_0}$. On the other hand, by \eqref{eqn:main-FK-det-outside}, 
$\mu_{x_0}$ and $\mu_{x_0+c_t}$ coincide when restricting to $\mathbb{C}\backslash \overline{\Xi_t}$. Hence, $\supp(\mu_{x_0+c_t})\subset \overline{\Xi_t}$. 
\end{proof}

The following regularity result extends \cite[Proposition 5]{Sniady2002} to unbounded operators. This is an analogue of \cite[Corollary 4.8]{HaagerupSchultz2009} where we use a circular operator in place of a circular Cauchy operator. 

\begin{proposition}
The function $t\mapsto \Delta(x_0+c_t-\lambda)$ is increasing for $t\geq 0$, and 
\begin{equation}
  \label{eqn:limit-det-circular-sum}
	\lim_{t\rightarrow 0^+}\Delta(x_0+c_t-\lambda)= \Delta(x_0-\lambda)
\end{equation}
for any $\lambda\in \mathbb{C}$. 
\end{proposition}
\begin{proof}
By the definition of the function $w(0;\lambda,t)$ in \eqref{eqn:w-epsiton-0-identity}, we see that
\[
   \lim_{t\rightarrow 0^+} w(0;\lambda,t)=0. 
\]
Moreover, by \eqref{eqn:w-epsiton-0-identity} we have
\[
   \frac{w(0;\lambda,t)^2}{t}
   = w(0;\lambda,t)^2	\phi\left[\big( (x_0-\lambda)^*(x_0-\lambda) +w(0; \lambda,t)^2 \big)^{-1}\right],
\]
which induces that $ \lim_{t\rightarrow 0^+} w(0;\lambda,t)^2/t =0.$

Hence, if $\lambda\in\Xi_{t_0}$ for some $t_0>0$, then by Part (1) of Theorem \ref{thm:main-FK-det-ct-0}, we have 
\begin{align*}
  &\lim_{t\rightarrow 0^+}	\Delta \big(x_0+c_t-\lambda \big)^2\\
  =&\lim_{t\rightarrow 0^+}\left[{\Delta\big( (x_0-\lambda)^*(x_0-\lambda)+w(0; \lambda,t)^2 \big)} \cdot {\exp\bigg(-\frac{(w(0;\lambda,t))^2}{t}\bigg)}\right]\\
  =&\Delta(x_0-\lambda)^2.
\end{align*}
This establishes \eqref{eqn:limit-det-circular-sum}. 

We note that $\Xi_{t_1}\subset \Xi_{t_2}$ for any $t_2>t_1>0$. Hence, if $\lambda\in\Xi_{t}$ for some $t>0$, write $w=w(0;\lambda,t)$, then
\begin{align*}
 \log \Delta \big(x_0+c_t-\lambda \big)^2
   &=\log {\Delta\big( (x_0-\lambda)^*(x_0-\lambda)+w^2 \big)}\\
   &\qquad\qquad-w^2	\phi\left[\big( (x_0-\lambda)^*(x_0-\lambda) +w^2 \big)^{-1}\right]\\
   &=\int_0^\infty \left( \log(u^2+w^2)-\frac{w^2}{u^2+w^2} \right)d\mu_{\vert x_0-\lambda\vert}(u).
\end{align*}
Note that the integrand on the right hand side of the above equation is a strictly increasing function of $w \in \mathbb{R}^+$. From the defining equation \eqref{eqn:w-epsiton-0-identity}, we see that $w(0;\lambda,t_2)>w(0;\lambda,t_1)$ 
for any $t_2>t_1$ and $\lambda\in \Xi_{t_2}$. It follows that $\Delta(x_0+c_{t_2}-\lambda)>\Delta(x_0+c_{t_1}-\lambda)$ for any $t_2>t_1$ and $\lambda\in \Xi_{t_2}$. 
If $\lambda \notin \Xi_{t_2},$
then by Theorem \ref{thm:main-FK-det-ct-0} we have
\begin{equation*}
\Delta(x_0+c_{t_2}-\lambda) = \Delta(x_0+c_{t_1}-\lambda) = \Delta(x_0 -\lambda).
\end{equation*}
This established the monotonicity of the function $t\mapsto \Delta(x_0+c_t-\lambda)$ for any $\lambda\in\mathbb{C}$ fixed. 
\end{proof}

\subsection{The Brown measure formulas}
In this section, we study the Brown measures $\mu_{x_0+c_t}$ and $\mu_{x_0+g_{t,\gamma}}$. 
Biane-Lehner proposed the question about Brown measures $\mu_{x_0+c_t}$ in their pioneering work \cite{BianeLehner2001}. 
Our first result strengthens \cite[Theorem 1.4]{Bordenave-Chafai-circular} and \cite[Theorem 4.2]{Zhong2021Brown} to unbounded operator $x_0$ with an arbitrary distribution. We also address the issue of the nonexistence of singular continuous part, which was left open in \cite{Bordenave-Chafai-circular} and \cite{Zhong2021Brown}. The upper bound result generalizes \cite[Theorem 3.14]{HoZhong2020Brown} for an arbitrary operator, not necessarily self-adjoint. 

\begin{theorem}
	\label{thm:BrownFormula-x0-ct-general}
	The Brown measure of $x_0+c_t$ is absolutely continuous with respect to the Lebesgue measure. It is supported in the closure set $\overline{\Xi_t}$.
	The density function $\rho_{t,x_0}$ of the Brown measure at $\lambda\in \Xi_t$ is given by
	\begin{equation}
	\label{eqn:Brown-density-circular-positive}
 \rho_{t,x_0}(\lambda)=	\frac{1}{\pi}\left( \frac{\vert\phi((h^{-1})^2(\lambda-x_0))\vert^2}{\phi((h^{-1})^2)}
	+w^2\phi(h^{-1}k^{-1})  \right)
	\end{equation}
	where $w=w(0; \lambda,t)$ is determined by 
	\[
	\phi ( (x_0-\lambda)^*(x_0-\lambda)+w^2)^{-1} )=\frac{1}{t},
	\]
	and 
	$h=h(\lambda,w(0;\lambda,t))$ and $k=k(\lambda,w(0;\lambda,t))$ for 
	\[
	h(\lambda,w)=(\lambda-x_0)^*(\lambda-x_0)+w^2
	\]
	and
	\[
	k(\lambda,w)=(\lambda-x_0)(\lambda-x_0)^*+w^2.
	\]
	The density of the Brown measure of $x_0+c_t$ is strictly positive in the set $\Xi_t$. 
	Moreover, 
	\[
	  0<\rho_{t,x_0}(\lambda)\leq \frac{1}{\pi t}
	\]
	for any $\lambda\in\Xi_t,$ and $\rho_{t, x_0} (\lambda) = 1/ (\pi t)$ if and only if $x_0 = \lambda_0 \unit$ for some $\lambda_0 \in \mathbb{C}.$
\end{theorem}

\begin{lemma}
	 \label{lemma:density-regularized-circular}
	The regularized Brown measure $\mu_{x_0+c_t}^{(\varepsilon)}$ of $x_0+c_t$ is absolutely continuous with respect to the Lebesgue measure. 
The density function $\rho_{t,x_0}^{(\varepsilon)}$ of the Brown measure at $\lambda\in \mathbb{C}$ is given by
\begin{equation}
\label{eqn:Brown-density-circular-positive-regular}
\rho_{t,x_0}^{(\varepsilon)}(\lambda)=	\frac{1}{\pi}\left( \frac{\vert\phi((h^{-1})^2(\lambda-x_0))\vert^2}{\phi((h^{-1})^2)+{\varepsilon}/({2tw^3})}
+w^2\phi(h^{-1}k^{-1})  \right)
\end{equation}
where $w=w(\varepsilon; \lambda,t)$ is determined by 
\[
 w=\varepsilon+tw\phi[((\lambda-x_0)^*(\lambda-x_0)+w^2 )^{-1}],
\]
and 
$h=h(\lambda,w(\varepsilon;\lambda,t))$ and $k=k(\lambda,w(\varepsilon;\lambda,t))$ for 
\[
h(\lambda,w)=(\lambda-x_0)^*(\lambda-x_0)+w^2
\]
and
\[
k(\lambda,w)=(\lambda-x_0)(\lambda-x_0)^*+w^2.
\]
Moreover, 
\begin{equation}
 \label{eqn:density-regular-bound}
	0<\rho_{t,x_0}^{(\varepsilon)}(\lambda)<\frac{1}{\pi t}
\end{equation}
for any $\lambda\in\mathbb{C}$. 
\end{lemma}
\begin{proof}
Recall that $\lambda\mapsto w(\varepsilon;\lambda,t)$ is a $C^\infty$ function by Lemma \ref{lemma:regularity-sub-lambda}. Put
\[
  g(\lambda)=\log\Delta( (\lambda-x_0-c_t)^*(\lambda-x_0-c_t)+\varepsilon^2 ).
\]
Then, the density of the regularized Brown measure can be calculated as
\[
   \rho_{t,x_0}^{(\varepsilon)}(\lambda)
    =\frac{1}{\pi}\frac{\partial^2}{\partial\overline{\lambda}\partial\lambda}g(\lambda).
\]
By choosing $\gamma=0$, the subordination relation \eqref{eqn:subordination-Cauchy-entries} can be rewritten as
\begin{equation*}
 \begin{aligned}
  \frac{\partial}{\partial \lambda}g(\lambda)
 &=\phi\bigg( ({\lambda}-x_0-c_{t})^*\big( ({\lambda}-x_0-c_{t})({\lambda}-x_0-c_{t})^*+\varepsilon^2\big)^{-1} \bigg)\\
 &=\phi\bigg( (\lambda-x_0)^*\big( (\lambda-x_0)(\lambda-x_0)^*+w(\varepsilon;\lambda,t)^2\big)^{-1} \bigg)\\
 &=\phi\bigg( \big( (\lambda-x_0)^*(\lambda-x_0)+w(\varepsilon;\lambda,t)^2\big)^{-1}(\lambda-x_0)^* \bigg)\\
 &=\phi\left(h^{-1} (\lambda-x_0)^* \right).
 \end{aligned}
\end{equation*}

Note that $g(\lambda)$ is a $C^\infty$ function of $\lambda$. We apply Lemma \ref{lem:tech-thm} to get 
 	\begin{align}
 \frac{\partial^2}{\partial\overline{\lambda
 	}\partial\lambda}g(\lambda)
 &=\frac{\partial}{\partial\overline{\lambda}}\phi\left(h^{-1} (\lambda-x_0)^*\right)\nonumber\\
 &=\phi\left( -h^{-1} \left( (\lambda -x_0)+2w \frac{\partial {w}}{\partial\overline{\lambda}}  \right) h^{-1} (\lambda-x_0)^*  + h^{-1}  \right) \nonumber\\
 &=\phi\left( -h^{-1} (\lambda-x_0) h^{-1} (\lambda-x_0)^*  + h^{-1}  \right)
 -2w \frac{\partial {w}}{\partial\overline{\lambda}} 
 \phi\left((\lambda-x_0)^* (h^{-1})^2  \right).\label{eqn:deriative-second-terms12}
 \end{align}
We now apply the identity $x(x^*x+a)^{-1}x^*=(xx^*+a)^{-1}xx^*$
to $x=\lambda-x_0$ and $a=w^2$, we find that
\begin{align*}
1- (\lambda-x_0) h^{-1} (\lambda-x_0)^* 
&=1-x(x^*x+w^2)^{-1}x^*\\
&=1-(xx^*+w^2)^{-1}xx^*\\
&=w^2 (xx^*+w^2)^{-1}\\
&=w^2 k^{-1}.
\end{align*}
Therefore, the first term in the right hand side of \eqref{eqn:deriative-second-terms12} can be rewritten as
\begin{equation}
 \label{eqn:derivative-term1}
	 \phi\left( -h^{-1} (\lambda-x_0) h^{-1} (\lambda-x_0)^*  + h^{-1}  \right) =w^2\phi(h^{-1}k^{-1}). 
\end{equation}

We next calculate $\frac{\partial w}{\partial\overline{\lambda}}$ by taking implicit differentiation. Recall that $w=w(\varepsilon;\lambda,t)$ satisfies \eqref{eqn:fixed-pt-scalar-lambda}, which can be rewritten as
\begin{equation}
	w-\varepsilon=tw \phi(h^{-1}). 
\end{equation}
By taking the derivative for both sides, we obtain
\begin{align*}
  \frac{\partial w}{\partial\overline{\lambda}}
   &=-tw\left[\phi\left( h^{-1}\left( (\lambda-x_0)+2w \frac{\partial {w}}{\partial\overline{\lambda}}  \right) h^{-1}   \right) \right]
     +t\phi(h^{-1})\frac{\partial w}{\partial\overline{\lambda}},
\end{align*}
which yields
\begin{align*}
    \phi\big((h^{-1})^2 (\lambda-x_0) \big)
   &=-2w\frac{\partial w}{\partial\overline{\lambda}}
     \left( \phi((h^{-1})^2)-\frac{1}{2w^2}\phi(h^{-1})+\frac{1}{2w^2t} \right)\\
   &=-2w\frac{\partial w}{\partial\overline{\lambda}}
   \left( \phi((h^{-1})^2)+\frac{\varepsilon}{2tw^3} \right),
\end{align*}
where we used $\phi(h^{-1})=\frac{w-\varepsilon}{tw}$ to derive the last identity above. 
 Therefore, for the second term in the right hand side of \eqref{eqn:deriative-second-terms12} we have
 \begin{equation}
  \label{eqn:derivative-term2}
  \begin{aligned}
 -2w \frac{\partial {w}}{\partial\overline{\lambda}} \phi((h^{-1})^{2} (\lambda-x_0)^* )
 =\frac{\vert \phi((h^{-1})^{2} (\lambda-x_0) )\vert ^2}{ \phi((h^{-1})^2)+{\varepsilon}/({2tw^3})}.
  \end{aligned}
 \end{equation}
By plugging \eqref{eqn:derivative-term1} and \eqref{eqn:derivative-term2} to \eqref{eqn:deriative-second-terms12}, we obtain the density formula as desired. 

We now turn to prove the upper bound of the density function. We write $x_\lambda=\lambda-x_0$ and recall that $h=x_\lambda^*x_\lambda+w^2$. Again by using Lemma \ref{lem:tech-thm}, we have  $x(x^*x+a)^{-1}=(xx^*+a)^{-1}x$ for any $x$ and $a\in (0,\infty)$, it follows that
\begin{align*}
 \phi((h^{-1})^{2} (\lambda-x_0) )&=\phi((h^{-1})^{2} x_\lambda)\\
    &=\phi(x_\lambda h^{-1}h^{-1})\\
    &=\phi(k^{-1}x_\lambda h^{-1}). 
\end{align*}
By the Cauchy-Schwarz inequality, we have 
	\begin{align}
 \vert\phi(k^{-1}x_\lambda h^{-1})\vert^2&=
 \vert \phi\big( (h^{-1/2}k^{-1/2} ) (k^{-1/2}x_\lambda h^{-1/2}) \big)\vert^2\nonumber\\
 &\leq \phi(h^{-1}k^{-1})\cdot \phi(k^{-1/2}x_\lambda h^{-1} x_\lambda^* k^{-1/2}  )\nonumber\\
 &=\phi(h^{-1}k^{-1})\cdot \phi(x_\lambda h^{-1} x_\lambda^* k^{-1} )\nonumber\\
 &=\phi(h^{-1}k^{-1})\cdot \phi(x_\lambda h^{-1}  h^{-1}x_\lambda^* )\nonumber\\
 &=\phi(h^{-1}k^{-1})\cdot \phi(h^{-1}  h^{-1}x_\lambda^* x_\lambda ).\label{eqn:saturate-bound}
\end{align}
Therefore, 
\begin{align}
&\frac{\vert\phi((h^{-1})^2(\lambda-x_0))\vert^2}{\phi((h^{-1})^2)+{\varepsilon}/({2tw^3})}
+w^2\phi(h^{-1}k^{-1})\nonumber\\
 &<  \frac{\phi(h^{-1}k^{-1})\cdot \phi(h^{-1}  h^{-1}x_\lambda^* x_\lambda )}{\phi((h^{-1})^2)} 
+w^2\phi(h^{-1}k^{-1})\nonumber\\
&=\frac{\phi(h^{-1}k^{-1})\cdot \phi (h^{-1}) }{\phi((h^{-1})^2)}
=\frac{\phi(h^{-1}k^{-1}) }{\phi((h^{-1})^2)}\cdot \frac{w-\varepsilon}{wt}
<\frac{1}{t}\frac{\phi(h^{-1}k^{-1}) }{\phi((h^{-1})^2)}. \label{eqn:saturate-bound-1}
\end{align}
Finally, since $\phi((h^{-1})^2)=\phi((k^{-1})^2)$, by the Cauchy-Schwarz inequality, we have 
\[\phi(h^{-1}k^{-1})\leq\phi((h^{-1})^2).\] 
This establishes that $0<\rho_{t,x_0}^{(\varepsilon)}(\lambda)<1/(\pi {t})$ for any $\lambda \in \mathbb{C}$. 
\end{proof}

We are now ready to prove Theorem \ref{thm:BrownFormula-x0-ct-general}.

\begin{proof}[Proof of Theorem \ref{thm:BrownFormula-x0-ct-general}]
Let $E$ be any Borel measurable set in $\mathbb{C}$ with zero Lebesgue measure. Let $U$ be an open set such that $E\subset U$. Since $\mu_{x_0+c_t}^{(\varepsilon)}$ converges to $\mu_{x_0+c_t}$ weakly as $\varepsilon$ tends to zero, it follows that
\[
   \mu_{x_0+c_t}(U)\leq \liminf_{\varepsilon\rightarrow 0}\mu_{x_0+c_t}^{(\varepsilon)}(U)\leq \frac{m(U)}{\pi {t}},
\]
by \eqref{eqn:density-regular-bound}, where $m(U)$ denotes the Lebesgue measure of $U$. Hence, $\mu_{x_0+c_t}(E)=0$ and the Brown measure $\mu_{x_0+c_t}$ is absolutely continuous. 

By Lemma \ref{lemma:regularity-sub-lambda}, $w(\varepsilon;\lambda,t)\rightarrow w(0;\lambda,t)$ uniformly in any compact set of $\mathbb{C}$. Since $w(0;\lambda,t)>0$ for $\lambda\in \Xi_t$. Hence, the formula for the density function $\rho_{t,x_0}$ follows from \eqref{eqn:Brown-density-circular-positive-regular} by letting $\varepsilon$ go to zero. In addition, by the upper bound for the regularized Brown measure, we have $\rho_{t,x_0}(\lambda)\leq 1/(\pi {t})$. Now we move to the condition that saturates the bound $1/(\pi t).$ Similar to the proof for the upper bound in Lemma \ref{lemma:density-regularized-circular}. We note that in \eqref{eqn:saturate-bound}, the equality occurs if and only if there exists some $z_0 \in \mathbb{C}$ 
such that 
\[ h^{-1/2}k^{-1/2} = z_0 h^{-1/2} x_\lambda^* k^{-1/2},
\]
which is equivalent to $x_0 = \lambda_0 \unit$ for some $\lambda_0 \in \mathbb{C}.$ Hence, if $x_0 = \lambda_0 \unit,$ then it is clear that $\rho_{t, x_0}(\lambda) = 1/ (\pi t)$ for any $\lambda \in \Xi_t.$ On the other hand, if $x_0$ is not proportional to the identity, by applying a similar estimate as \eqref{eqn:saturate-bound-1} for $\varepsilon$ equals to zero, we have for any $\lambda \in \Xi_t,$ 
\begin{align*}
\rho_{t,x_0}(\lambda)& =\frac{1}{\pi}\left( \frac{\vert\phi((h^{-1})^2(\lambda-x_0))\vert^2}{\phi((h^{-1})^2)}
	+w^2\phi(h^{-1}k^{-1})  \right)\\
&<  \frac{1}{\pi}\left( \frac{\phi(h^{-1}k^{-1})\cdot \phi(h^{-1}  h^{-1}x_\lambda^* x_\lambda )}{\phi((h^{-1})^2)} 
+w^2\phi(h^{-1}k^{-1}) \right)\\
&=\frac{1}{\pi}\left( \frac{\phi(h^{-1}k^{-1})\cdot \phi (h^{-1}) }{\phi((h^{-1})^2)} \right)
=\frac{1}{\pi}\left( \frac{\phi(h^{-1}k^{-1}) }{\phi((h^{-1})^2)}\cdot \frac{w}{wt} \right)
<\frac{1}{\pi t}. 
\end{align*}
\end{proof}

\begin{remark}
After we submitted our paper, Erd\H{o}s and Ji posted 
an interesting paper at arXiv \cite{erdos2023density}, where they studied the edge behavior of the Brown measure $\mu_{x_0 + c_t}$ at the boundary of $\Xi_t$.
\end{remark}

\begin{theorem}
	\label{thm:Brown-pushforwd-elliptic-case}
	For any $\varepsilon>0$, the regularized Brown measure $\mu_{x_0+g_{t,\gamma}}^{(\varepsilon)}$ is the push-forward measure of the regularized Brown measure of $\mu_{x_0+c_t}^{(\varepsilon)}$ under that map $\Phi^{(\varepsilon)}_{t,\gamma}$. 
	The map $\Phi^{(\varepsilon)}_{t,\gamma}$ converges uniformly to $\Phi_{t,\gamma}$ in any compact subset of $\mathbb{C}$ as $\varepsilon$ tends to zero. Consequently, the Brown measure of $\mu_{x_0+g_{t,\gamma}}$ is the push-forward measure of the Brown measure $\mu_{x_0+c_t}$ under the map $\Phi_{t,\gamma}$. 
	
	Moreover, if the map $\Phi_{t, \gamma}$ is non-singular at any $\lambda\in\Xi_t$, then it is also one-to-one in $\Xi_t$. 
\end{theorem}
\begin{proof}
	Apply the proof for \cite[Theorem 5.2]{Zhong2021Brown} (see also Theorem \ref{thm:Brown-regularized-push-forward-property}) to get the first part. We set  $S_1(\lambda)=S(x_0+c_t, \lambda,\varepsilon)$ and $S_2(z)=S(x_0+g_{t,\gamma}, z,\varepsilon)$, and $S_0(\lambda)=S(x_0,\lambda,\varepsilon_0)$,
	where $\varepsilon_0=w(\varepsilon;\lambda,t)$ and $z=\Phi_{t, \gamma}(\lambda)$. 
	Similar to the proof for Theorem \ref{thm:Brown-regularized-push-forward-property}, the map $\Phi^{(\varepsilon)}_{t, \gamma}$ can be rewritten as
	\[
	\Phi^{(\varepsilon)}_{t, \gamma}(\lambda)=\lambda+\gamma\frac{\partial S_0(\lambda)}{\partial\lambda}=\lambda+\gamma \frac{\partial S_1(\lambda)}{\partial\lambda},
	\]
	by choosing $\gamma=0$ in the subordination relation \eqref{eqn:subordination-Cauchy-entries}. 
	The map $\Phi^{(\varepsilon)}_{t, \gamma}$ is a homeomorphism by Corollary \ref{cor:regularization-Phi-one2one-map}. 
	In addition, the subordination relation \eqref{eqn:subordination-Cauchy-entries} shows 
	\[
	\frac{\partial S_0(\lambda)}{\partial\lambda}= \frac{\partial S_1(\lambda)}{\partial\lambda}=\frac{\partial S_2(z)}{\partial z}
	\]
	and 
	\[
	\frac{\partial S_0(\lambda)}{\partial\overline\lambda}= \frac{\partial S_1(\lambda)}{\partial\overline\lambda}=\frac{\partial S_2(z)}{\partial \overline z}. 
	\]
	Hence, conditions in Lemma \ref{lemma:subharmonic-pushforward} are satisfied. The pushforward result for regularized Brown measures then follows.  s
	
	 The  proof of \cite[Lemma 5.3]{Zhong2021Brown} shows 
	\[
	|p_\lambda^{(0)}(w(\varepsilon_2;\lambda,t))-p_\lambda^{(0)}(w(\varepsilon_1;\lambda,t))|
	\leq \frac{2}{t}(w(\varepsilon_2;\lambda,t)-w(\varepsilon_1;\lambda,t)).
	\]
	Hence, local uniform convergence of $w(\varepsilon;\lambda,t)$ in Lemma \ref{lemma:regularity-sub-lambda} implies that $\Phi^{(\varepsilon)}_{t,\gamma}$ converges uniformly to $\Phi_{t,\gamma}$ in any compact subset of $\mathbb{C}$ as $\varepsilon$ tends to zero. We then conclude that push-forward connection holds after passing to the limit. 
	
	Finally, if the map $\Phi_{t, \gamma}$ is non-singular, then the last claim follows from a similar argument for the proof of Proposition \ref{prop:partial-lambda-pushforward}. 
\end{proof}

\begin{remark}
By Theorem \ref{thm:BrownFormula-x0-ct-general} and Theorem \ref{thm:Brown-pushforwd-elliptic-case}, there is no problem to extend examples from \cite{Zhong2021Brown} to unbounded case when $x_0$ is self-adjoint or $R$-diagonal. Hence, all results from \cite[Section 6 and 7]{Zhong2021Brown} still hold for unbounded $x_0$. In particular, the pushforward result obtained in \cite[Theorem 6.4, Corollary 6.9]{Ho2021unbounded} are very special examples of Theorem \ref{thm:Brown-pushforwd-elliptic-case}. 
\end{remark}

%%%%%%%%%%%%

\bigskip

\noindent {\bf Acknowledgements}: We would like to thank the referees for their careful readings and valuable suggestions. This work was supported by the National Natural Science Foundation of China [grant number 12031004]; Collaboration Grants for Mathematicians from Simons Foundation; and National Science Foundation award LEAPS-MPS-2316836.

\bibliographystyle{acm}
 \bibliography{BrownMeasure}

\end{document}